\documentclass[a4paper, 11pt]{article}

\newcounter{D}
\usepackage{amsmath,amsthm}
\usepackage{amsfonts}
\usepackage{amssymb}
\usepackage{bbm}
\usepackage{color}
\usepackage{geometry}
\usepackage{yfonts}
\usepackage{verbatim}
\usepackage{caption}
\usepackage{mathtools}
\usepackage{natbib}
\usepackage{framed}
\newcommand{\supp}{\mathrm{supp}}

\newcommand{\p}{\partial}

\renewcommand{\d}{\mathrm{d}}

\newcommand{\eqnb}{\begin{equation}}
\newcommand{\eqnbs}{\begin{equation*}}
\newcommand{\eqnbsa}{\begin{equation*}\begin{aligned}}
\newcommand{\eqnba}{\begin{equation}\begin{aligned}}
\newcommand{\eqnbl}[1]{\begin{equation}\label{#1}}
\newcommand{\eqnbal}[1]{\begin{equation}\label{#1}\begin{aligned}}
\newcommand{\eqnes}{\end{equation*}}
\newcommand{\eqne}{\end{equation}}
\newcommand{\splb}{ \begin{split}}
\newcommand{\sple}{\end{split}}
\newcommand{\eqnesa}{\end{aligned}\end{equation*}}
\newcommand{\eqnea}{\end{aligned}\end{equation}}

\newcommand{\fm}{\mathfrak{m}}
\newcommand{\fmm}{\mathfrak{M}}
\newcommand{\fn}{\mathfrak{n}}

\newcounter{mainthms}
\newcounter{thm}
\numberwithin{thm}{section}
\numberwithin{equation}{section}

\newcommand{\lewy}{\left\lbrace}
\newcommand{\prawy}{\right\rbrace}
\newcommand{\RR}{\mathbb{R}}

\newcommand{\ZZ}{\mathbb{Z}}

\newcommand{\NN}{\mathbb{N}}

\newtheorem{theorem}[thm]{Theorem}%[section]
\newtheorem*{theorem*}{Theorem}%[section]
\newtheorem{maintheorem}[mainthms]{Theorem}%[section]
\newtheorem{proposition}[thm]{Proposition}%[section]
\newtheorem{lemma}[thm]{Lemma}%[section]

\newtheorem{corollary}[thm]{Corollary}%[section]
\newtheorem{df}[thm]{Definition}
\newtheorem{warning}[thm]{Warning}

\begin{document}
\title{On weak solutions to the Navier--Stokes inequality with internal singularities}
\author{Wojciech S. O\.za\'nski}
\date{}
\maketitle
%#######################################################
\begin{abstract}
We construct weak solutions to the Navier--Stokes inequality,
\[
u\cdot \left( \p_t u -\nu \Delta u + (u\cdot \nabla ) u +\nabla p \right) \leq 0
\]
in $\RR^3$ that blow up at a single point $(x_0,T_0)$ or on a set $S \times \{ T_0 \}$, where $S\subset \RR^3$ is a Cantor set whose Hausdorff dimension is at least $\xi$ for any preassigned $\xi\in (0,1)$. Such solutions were constructed by Scheffer, \emph{Comm. Math. Phys.}, 1985 \& 1987. Here we offer a simpler perspective on these constructions. We sharpen the approach to construct smooth solutions to the Navier--Stokes inequality on the time interval $[0,1]$ satisfying the ``approximate equality'' 
\[
\left\| u\cdot \left( \p_t u-\nu \Delta u + (u\cdot \nabla ) u +\nabla p \right) \right\|_{L^\infty }\leq \vartheta,
\]
and the ``norm inflation'' $\| u(1) \|_{L^\infty } \geq \mathcal{N} \| u(0) \|_{L^\infty }$ for any preassigned $\mathcal{N}>0$, $\vartheta >0$. Furthermore we extend the approach to construct a weak solution to the Euler inequality
\[
u\cdot \left( \p_t u+ (u\cdot \nabla ) u +\nabla p \right) \leq 0
\]
that satisfies the approximate equality with $\nu =0$ and blows up on the Cantor set $S\times \{ T_0 \}$ as above.
\end{abstract}

\tableofcontents
\section{Introduction}

The Navier--Stokes equations,
\[\begin{split}
\p_t u -\nu \Delta u + (u\cdot \nabla )u +\nabla p &=0,\\
\mathrm{div}\, u &=0,
\end{split}
\]
where $u$ denotes the velocity of a fluid, $p$ the scalar pressure and $\nu >0$ the viscosity, comprise a fundamental model for viscous, incompressible flows. In the case of the whole space $\RR^3$ the pressure function is given (at each time instant $t$) by the formula
\eqnb\label{def_of_the_corresponding_pressure}
p \coloneqq \sum_{i,j=1}^3 \p_{ij} \Psi \ast (u_iu_j),
\eqne
where $\Psi (x) \coloneqq (4\pi |x|)^{-1}$ denotes the fundamental solution of the Laplace equation in $\RR^3$ and ``$\ast$'' denotes the convolution. The formula above, which we shall refer to simply as the \emph{pressure function corresponding to }$u$, can be derived by calculating the divergence of the Navier--Stokes equation.

The fundamental mathematical theory of the Navier--Stokes equations goes back to the pioneering work of \cite{Leray_1934} (see \cite{Leray_review} for a comprehensive review of this paper in more modern language), who used a Picard iteration scheme to prove existence and uniqueness of local-in-time strong solutions. Moreover, \cite{Leray_1934} and \cite{Hopf_1951} proved a global-in-time existence (without uniqueness) of weak solutions satisfying the energy inequality,
\eqnb\label{energy_inequality_intro}
\| u(t) \|^2 + 2\nu \int_s^t \| \nabla u (\tau )\|^2 \d \tau \leq \| u(s) \|^2
\eqne
for almost every $s\geq 0$ and every $t>s$ (often called \emph{Leray-Hopf weak solutions}) in the case of the whole space $\RR^3$ (Leray) as well as in the case of a bounded, smooth domain $\Omega \subset \RR^3$ (Hopf). Although the fundamental question of global-in-time existence and uniqueness of strong solutions remains unresolved (as does the question of uniqueness of Leray-Hopf weak solutions; however, see \cite{buckmaster_vicol} for nonuniqueness of (non-Leray-Hopf) weak solutions), many significant results contributed to the theory of the Navier--Stokes equations during the second half of the twentieth century. One such contribution is the partial regularity theory introduced by Scheffer (1976\emph{a}, 1976\emph{b}, 1977, 1978 \& 1980)\nocite{scheffer_hausdorff_measure}\nocite{scheffer_partial_reg}\nocite{scheffer_turbulence}\nocite{scheffer_dim_4}\nocite{scheffer_NSE_on_bdd_domain} and subsequently developed by Caffarelli, Kohn \& Nirenberg (1982)\nocite{CKN}; see also \cite{lin_1998}, \cite{ladyzhenskaya_seregin}, \cite{vasseur_2007} and \cite{kukavica_partial_reg_2009} for alternative approaches. This theory gives sufficient conditions on local regularity of solutions in space-time. Namely, letting $Q_r(z)\coloneqq B_r (x) \times (t-r^2,t)$, a space-time cylinder centred at $z=(x,t)$, the central result of this theory, proved by \cite{CKN}, is the following.
\begin{theorem}[Partial regularity of the Navier--Stokes equations]\label{thm_partial_reg_NSE}
Let $u_0\in L^2 (\RR^3) $ be weakly divergence-free and let $u$ be a ``suitable weak solution'' of the Navier--Stokes equations on $\RR^3$ with initial condition condition $u_0$. If 
\eqnb\label{PR1}
\frac{1}{r^2} \int_{Q_r} |u|^3 + |p|^{3/2} <\varepsilon_0
\eqne
for any cylinder $Q_r=Q_r (z)$, $r>0$, then $u$ is bounded in $Q_{r/2} (z)$.

Moreover if 
\eqnb\label{PR2}
\limsup_{r\to 0} \frac{1}{r} \int_{Q_r} | \nabla u |^2 < \varepsilon_1
\eqne
then $u$ is bounded in a cylinder $Q_\rho (z)$ for some $\rho >0$.
\end{theorem}
Here $\varepsilon_0,\varepsilon_1>0$ are certain universal constants (sufficiently small), and the notion of a ``suitable weak solution'' refers to a Leray-Hopf weak solution that satisfies the \emph{local energy inequality},
\eqnb\label{local_energy_inequality}
2\nu \int_0^\infty \int_{\RR^3}  |\nabla u |^2 \varphi  \leq \int_0^\infty \int_{\RR^3} \left( |u|^2 (\p_t \varphi + \nu \Delta \varphi )  +(|u|^2+2p)(u\cdot \nabla )\varphi   \right)
\eqne
for all non-negative $\varphi \in C_0^\infty (\RR^3 \times (0,\infty ))$, where $p$ is the pressure function corresponding to $u$ (see \eqref{def_of_the_corresponding_pressure}). The existence of global-in-time suitable weak solutions given divergence-free initial data $u_0 \in L^2 (\RR^3)$ was proved by \cite{scheffer_hausdorff_measure} (and by \cite{CKN} in the case of a bounded domain).

The partial regularity theorem (Theorem \ref{thm_partial_reg_NSE}) is a key ingredient in the $L_{3,\infty}$ regularity criterion for the three-dimensional Navier--Stokes equations (see Escauriaza, Seregin \& \v{S}ver\'{a}k 2003\nocite{ESS_2003}) and the uniqueness of Lagrangian trajectories for suitable weak solutions \citep{rob_sad_2009}; similar ideas have also been used for other models, such as the surface growth model $\p_t u=-u_{xxxx}-\p_x^2 u_x^2$ \citep{SGM}, which is a one-dimensional model of the Navier--Stokes equations \citep{blomker_romito_reg_blowup, blomker_romito_loc_ex}.

A remarkable fact about the partial regularity theory is that the quantities involved in the local regularity criteria (that is $|u|^3$, $|p|^{3/2}$ and $|\nabla u|^2$), are known to be globally integrable for any vector field satisfying $\sup_{t>0} \| u(t) \| <\infty$, $\nabla u \in L^2 (\RR^3 \times (0,\infty ))$ (which follows by interpolation, see for example, Lemma 3.5 and inequality (5.7) in \cite{NSE_book}); thus in particular for any Leray-Hopf weak solution. Thus Theorem \ref{thm_partial_reg_NSE} shows that, in a sense, if these quantities localise near a given point $z\in \RR^3 \times (0,\infty)$ in a way that is ``not too bad'', then $z$ is not a singular point, and thus there cannot be ``too many'' singular points. In fact, by letting $S\subset \RR^3 \times (0,\infty )$ denote the singular set, that is 
\[
S\coloneqq \{ (x,t) \in \RR^3 \times (0,\infty ) \colon u \text{ is unbounded in any neighbourhood of }(x,t)\},
\]
this can be made precise by estimating the ``dimension'' of $S$. Namely, a simple consequence of \eqref{PR1} and \eqref{PR2} is that
\eqnb\label{dB_dH_bounds_intro}
d_B(S)\leq 5/3,\qquad \text{ and }\qquad  d_H(S) \leq 1 ,
\eqne
respectively\footnote{In fact, \eqref{PR2} implies a stronger estimate than $d_H (S) \leq 1$; namely that $\mathcal{P}^1(S)=0$, where $\mathcal{P}^1(S)$ is the \emph{parabolic Hausdorff measure} of $S$ (see Theorem 16.2 in \cite{NSE_book} for details).}, see Theorem 15.8 and Theorem 16.2 in \cite{NSE_book}. Here $d_B$ and $d_H$ denote the \emph{box-counting dimension} (also called the \emph{fractal dimension} or the \emph{Minkowski dimension}) and the \emph{Hausdorff dimension}. The relevant definitions can be found in \cite{falconer}, who also proves (in Proposition 3.4) an important property that $d_H(K) \leq d_B (K)$ for any compact set $K$. 

Before discussing the bounds on the dimension of the singular set \eqref{dB_dH_bounds_intro} in detail, we point out that it is valid not only for suitable weak solutions, but also for a wider family of vector fields. This motivates the following definition.
\begin{df}[Weak solution to the Navier--Stokes inequality]\label{def_weak_sol_of_NSI}
A divergence-free vector field $u\colon \RR^3 \times (0,\infty )$ satisfying $\sup_{t>0} \| u(t) \| <\infty$, $\nabla u \in L^2 (\RR^3 \times (0,\infty ))$ is a \emph{weak solution of the Navier--Stokes inequality} if it satisfies the local energy inequality \eqref{local_energy_inequality}.
\end{df}

Observe that we have incorporated the definition of the pressure function into the local energy inequality \eqref{local_energy_inequality}. Namely (since we will only focus on the case of the whole space $\RR^3$) the pressure function is given by \eqref{def_of_the_corresponding_pressure}. We now briefly discuss the regularity of weak solutions to the Navier--Stokes inequality. First, the energy inequality \eqref{energy_inequality_intro} gives that $\sup_{t>0} \| u(t) \| <\infty$, $\nabla u \in L^2 (\RR^3 \times (0,\infty ))$, which in turn implies (by interpolation) that $u\in L^{10/3} (\RR^3 \times (0,\infty ))$. From this and the Calder\'on-Zygmund inequality (see, for example, Theorem B.6 in \cite{NSE_book}) one can deduce that $p\in L^{5/3} (\RR^3 \times (0,\infty )$. In particular all terms in the local energy inequality \eqref{local_energy_inequality} are well-defined.

The point of the above definition is that a weak solution of the Navier--Stokes inequality need not satisfy any partial differential equation, but merely the local energy inequality.
In fact, weak solutions of the Navier--Stokes inequality satisfy all the assumptions that are sufficient for Caffarelli, Kohn \& Nirenberg's proof of the partial regularity theory (as stated Theorem \ref{thm_partial_reg_NSE}).

 The name \emph{Navier--Stokes inequality} (which we shall refer to simply by writing NSI) is motivated by the fact that the local energy inequality \eqref{local_energy_inequality} is in fact a weak form of the inequality
\eqnb\label{navier_stokes_inequality}
u\cdot \left( \p_t u -\nu \Delta u +(u\cdot \nabla ) u + \nabla p \right) \leq 0.
\eqne
In order to see this fact, note that the NSI can be rewritten, for smooth $u$ and $p$, in the form
\[
\frac{1}{2} \p_t |u|^2 - \frac{\nu}{2}  \Delta |u|^2 +\nu | \nabla u |^2+ u\cdot \nabla \left( \frac{1}{2} |u|^2 + p \right) \leq 0,
\]
where we used the calculus identity $u\cdot \Delta u = \Delta (|u|^2/2) - |\nabla u |^2$. Multiplication by $2\varphi$ and integration by parts gives \eqref{local_energy_inequality}. 

Furthermore, setting
\[
f\coloneqq \p_t u -\nu \Delta u +(u\cdot \nabla ) u + \nabla p,
\]
one can think of the Navier--Stokes inequality \eqref{navier_stokes_inequality} as the inhomogeneous Navier--Stokes equations with forcing $f$,
\[\p_t u -\nu \Delta u +(u\cdot \nabla ) u + \nabla p =f ,\]
where $f$ acts against the direction of the flow $u$, that is $f\cdot u \leq 0$.

Returning to the bounds \eqref{dB_dH_bounds_intro} on the dimension of the singular set, it turns out that the bound $d_B(S)\leq 5/3$ (for suitable weak solutions of the NSE) can be improved. Indeed, first \cite{kukavica_2009} proved the estimate $d_B(S)\leq 135/82 (\approx 1.65)$ and the bound was later refined by \cite{kukavica_pei_2012}, \cite{koh_yang_2016}, \cite{wang_wu_2017}, down to the most recent bound $d_B(S) \leq 2400/1903  (\approx 1.261)$ obtained by \cite{he_wang_zhou_2017}.
As for the Hausdorff dimension, the bound $d_H(S)\leq 1$ has not been improved. In fact, the ingenious construction of counterexamples by Scheffer (1985 \& 1987), which are the subject of this article, show that this bound is sharp for weak solutions of the NSI (of course, it is not known whether it is sharp for suitable weak solutions of the NSE).

 The first of his results (proved in \cite{scheffer_a_sol}) is the following.
\begin{maintheorem}[Weak solution of NSI with point singularity]\label{point_blowup_thm}
There exist $\nu_0>0$ and a function $\mathfrak{u}\colon \RR^3 \times [0,\infty ) \to \RR^3$ that is a weak solution of the Navier--Stokes inequality with any $\nu \in [0,\nu_0]$ such that $\mathfrak{u}(t)\in C^\infty$, $\supp \,\mathfrak{u}(t)\subset G$ for all $t$  for some compact set $G\Subset \RR$ (independent of $t$). Moreover $\mathfrak{u}$ is unbounded in every neighbourhood of $(x_0,T_0)$, for some $x_0 \in \RR^3$, $T_0>0$.
\end{maintheorem}
It is clear, using an appropriate rescaling, that the statement of the above theorem is equivalent to the one where $\nu= 1$ and $(x_0,T_0)=(0,1)$. Indeed, if $\mathfrak{u}$ is the velocity field given by the theorem then $\sqrt{T_0/\nu_0 } \mathfrak{u} (x_0+ \sqrt{T_0 \nu_0 }x ,T_0 t)$ satisfies Theorem \ref{point_blowup_thm} with $\nu=1$, $(x_0,T_0)=(0,1)$. 

In a subsequent paper \cite{scheffer_nearly} constructed weak solutions of the Navier--Stokes inequality that blow up on a Cantor set $S\times \{ T_0 \}$ with $d_H (S)\geq \xi$ for any preassigned $\xi \in (0,1)$. 
\begin{maintheorem}[Nearly one-dimensional singular set]\label{1D_blowup_thm_intro}
Given any $\xi \in (0,1)$ there exists $\nu_0>0$, a compact set $G\Subset \RR^3$ and a function $\mathfrak{u}\colon \RR^3 \times [0,\infty ) \to \RR^3$ that is a weak solution to the Navier--Stokes inequality such that $u(t) \in C^\infty$, $\supp \,\mathfrak{u}(t)\subset G$ for all $t$, and 
\[
\xi \leq d_H (S) \leq 1 ,
\]
where
\[
S\coloneqq \{ (x,t) \in \RR^3 \times (0,\infty ) \, : \, \mathfrak{u}(x,t) \text{ is unbounded in any neighbourhood of } (x,t) \}.
\]
\end{maintheorem}
The above results make use of an alternative form of the local energy inequality. Namely, the local energy inequality \eqref{local_energy_inequality} is satisfied if the \emph{local energy inequality on the time interval } $[S,{S'}]$,
\eqnb\label{alternative_LEI}
\begin{split}
\int_{\RR^3} |&u(x,S')|^2 \varphi \, \d x - \int_{\RR^3} |u(x,S)|^2 \varphi \, \d x + 
2\nu \int_S^{S'} \int_{\RR^3} | \nabla u |^2 \varphi  \\
&\leq  \int_S^{S'} \int_{\RR^3} \left( |u|^2 +2p\right) u \cdot \nabla \varphi + \int_S^{S'} \int_{\RR^3}  |u|^2 \left( \p_t \varphi + \nu \Delta \varphi \right),
\end{split}
\eqne
holds for all $S,{S'}>0$ with $S<{S'}$, which is clear by taking $S,S'$ such that $\supp\, \varphi \subset \RR^3 \times (S,S')$. An advantage of this alternative form of the local energy inequality is that it demonstrates how to combine solutions of the Navier--Stokes inequality one after another. Namely, \eqref{alternative_LEI} shows that a necessary and sufficient condition for two vector fields $u^{(1)}\colon \RR^3 \times [t_0,t_1] \to \RR^3$, $u^{(2)}\colon \RR^3 \times [t_1,t_2] \to \RR^3$ satisfying the local energy inequality on the time intervals $[t_0,t_1]$, $[t_1,t_2]$, respectively, to combine (one after another) into a vector field satisfying the local energy inequality on the time interval $[t_0,t_2]$ is that
\eqnb\label{what_you_need_to_combine_intro}
|u^{(2)} (x,t_1) | \leq |u^{(1)} (x,t_1) | \qquad \text{ for a.e. } x\in \RR^3.
\eqne

It turns out that Scheffer's dense proofs of the two theorems can be rephrased in a more succinct and intuitive form, which we present in this article. As a part of the simplification process we introduce the notion of a \emph{structure} on an open subset of the upper half-plane (see Definition \ref{def_of_structure}), which allows one to construct a compactly supported, divergence-free vector field $u$ in $\RR^3$ with prescribed absolute value $|u|$ and with a number of other useful properties (see Section \ref{sec_a_structure} and Lemma \ref{properties_of_u[v,f]}). Moreover, we point out the key concepts used in the construction of the blow-up. Namely, we introduce the notion of the \emph{pressure interaction function} (corresponding to a given subset of the half-plane and its structure, see Section \ref{sec_pressure_int_fcn}), which articulates a certain nonlocal property of the pressure function (see Lemma \ref{lem_properties_of_F}), and we formalise the concept of the \emph{geometric arrangement} (see Section \ref{sec_the_setting}), that is a certain configuration of subsets of the upper half-plane (and their structures) which, in a sense, ``magnifies'' the pressure interaction. We also expose some other concepts used in the proof, such as an analysis of rescalings of vector fields and some ideas related to dealing with the nonlocal character of the pressure function. In addition to these simplifications, we point out how Theorem \ref{1D_blowup_thm_intro} is obtained as a straightforward extension of Theorem \ref{point_blowup_thm}.

Furthermore, we improve Theorem \ref{1D_blowup_thm_intro} in the case $\nu_0=0$ to construct a solution of the ``Euler inequality'',
\[
u\cdot \left( \p_t u+ (u\cdot \nabla ) u +\nabla p \right) \leq 0,
\]
that blows up on the Cantor set and satisfies the ``approximate equality''
\eqnb\label{approximate_equality}
\| u\cdot \left( \p_t u+ (u\cdot \nabla ) u +\nabla p \right) \|_{L^\infty } \leq \vartheta
\eqne
for any preassigned $\vartheta >0$. To this end we use the construction from the proof of Theorem \ref{1D_blowup_thm_intro} and present a simple argument showing how the approximate equality requirement (with any $\vartheta $) enforces $\nu=0$; we thereby obtain the following result.
\begin{maintheorem}\label{thm_Euler_almost_equality}
Given $\xi \in (0,1)$ and $\vartheta >0$ there exists a function $\mathfrak{u} \colon \RR^3 \times [0,\infty ) \to \RR^3$ satisfying conditions (i)-(iv) of Theorem \ref{1D_blowup_thm_intro} with $\nu=0$ such that
\[
\| u\cdot \left( \p_t u+ (u\cdot \nabla ) u +\nabla p \right) \|_{L^\infty } \leq \vartheta .
\]
\end{maintheorem}
In other words, there exists a divergence-free solution to the inhomogeneous Euler equation,
\[
\p_t u+ (u\cdot \nabla ) u +\nabla p =f,
\]
with the forcing $f$ ``almost orthogonal'' to the velocity field, that is $-\vartheta \leq u\cdot f \leq 0$, and that blows up on the Cantor set.

It is not clear how to obtain a weak solution to the Navier--Stokes inequality (with some $\nu>0$) that blows up and satisfies the approximate equality. However, one can sharpen Scheffer's constructions to obtain the following ``norm inflation'' result.
\begin{maintheorem}[Smooth solution of NSI with norm inflation]\label{mainthm_large_gain}
Given $\mathcal{N}>0$, $\vartheta >0$ there exists $\eta>0$ and a nontrivial solution $\mathfrak{u} \in C^\infty (\RR^3 \times (-\eta,1+\eta );\RR^3)$ to the Navier--Stokes inequality \eqref{navier_stokes_inequality} satisfying the approximate equality 
\eqnb\label{approx_equality_nse}
\| u\cdot \left( \p_t u-\nu \Delta u+ (u\cdot \nabla ) u +\nabla p \right) \|_{L^\infty } \leq \vartheta ,
\eqne
for all $\nu\in [0,1]$, $\supp \, \mathfrak{u} (t)=G$ for all $t$ (where $G\subset \RR^3$ is compact), and
\[
\| u(1) \|_{L^\infty } \geq \mathcal{N} \| u(0) \|_{L^\infty}.
\]
\end{maintheorem}

The structure of the article is as follows. In Section \ref{sec_proof_of_thm1} below we present a sketch of the proof of Theorem \ref{point_blowup_thm}. In the following Section \ref{sec_remarks_to_the_sketch} we observe some the basic properties of the vector field $\mathfrak{u}$ obtained in the sketch and we point out how such a vector field can be used as a benchmark for various results in the theory of the Navier--Stokes equations, particularly blow-up criteria. The sketch of the proof of Theorem \ref{point_blowup_thm} is based on the existence of certain objects, which, after introducing a number of preliminary concepts in Section \ref{sec_prelims}, we construct in Section \ref{sec_the_setting}. The construction of these objects is based on a certain ``geometric arrangement'', which we discuss in Section \ref{sec_geom_arrangement}. We prove Theorem \ref{mainthm_large_gain} (which a corollary of Theorem \ref{point_blowup_thm}) in Section \ref{sec_norm_inflation}. In Section \ref{sec_pf_of_thm2} we prove Theorem \ref{1D_blowup_thm_intro} and, at the end of the section (in Section \ref{sec_pf_of_thm3}), we prove Theorem \ref{thm_Euler_almost_equality}.

\section{Sketch of the proof of Theorem 1}\label{sec_proof_of_thm1}

Here we present a simple argument which proves Theorem \ref{point_blowup_thm} given the following assumptions. Namely suppose for a moment that there exists $T>0$, a compact set $G\subset \RR^3$ and a divergence-free vector field $u$ such that $u \in C^\infty (\RR^3 \times [0 , T]; \RR^3)$, $\supp\, u(t) =G$ for all $t\in [0 , T]$, and the Navier--Stokes inequality
\eqnb\label{NSI_u}
\p_t |u |^2 \leq -u\cdot \nabla \left( |u|^2 +2p \right) + 2 \nu \, u\cdot \Delta u 
\eqne
 holds in $\RR^3 \times [0,T]$ for all $\nu\in [0,\nu_0]$ for some $\nu_0>0$, where $p(t)$ is the pressure function corresponding to $u$ (recall \eqref{def_of_the_corresponding_pressure}). Here $C^\infty (\RR^3 \times [0 , T]; \RR^3)$ is a short-hand notation for the space of vector functions that are infinitely differentiable on $\RR^3\times (-\eta , T+ \eta )$ for some $\eta>0$.

Suppose further that, during time interval $[0,T]$ $u$ admits the following interior gain of magnitude property: that for some $\tau \in (0,1)$, $z\in \RR^3$ the affine map 
\[
\Gamma (x) \coloneqq \tau x + z,
\]
maps $G$ into itself and that, at time $T$, $u$ attains a large gain in magnitude; namely that 
\eqnb\label{u_grows}
\left| u(\Gamma (x) , T) \right| \geq \tau^{-1} \left| u(x,0) \right|, \qquad x\in \RR^3.
\eqne
Such a gain in magnitude allows us to consider a rescaled copy of $u$ and, in a sense, slot it into the part of the support $G$ in which the gain occurred. Namely, considering
\[
u^{(1)} (x,t) \coloneqq \tau^{-1} u (\Gamma^{-1} (x), \tau^{-2} (t-T))
\]
we see that $u^{(1)}$ satisfies the Navier--Stokes inequality \eqref{NSI_u} on $\RR^3\times [T,(1+\tau^2) T]$, $\supp\, u^{(1)} (t) = \Gamma (G)$ for all $t\in [T,(1+\tau^2)T]$ and that \eqref{u_grows} gives 
\eqnb\label{switch_into_u1}
\left| u^{(1)} (x,T ) \right| \leq \left| u (x,T ) \right|, \qquad  x\in \RR^3
\eqne
(and so $u$, $u^{(1)}$ can be combined ``one after another'', recall \eqref{what_you_need_to_combine_intro} above). Thus, since $u^{(1)}$ is larger in magnitude than $u$ (by the factor of $\tau$) and its time of existence is $[T,(1+\tau^2 )T]$, we see that by iterating such a switching we can obtain a vector field $\mathfrak{u}$ that grows indefinitely in magnitude, while its support shrinks to a point (and thus will satisfy all the claims of Theorem \ref{point_blowup_thm}), see Fig. \ref{switching_process_figure_with_supp}. To be more precise we let $t_0 \coloneqq 0$, 
\[
t_j \coloneqq T \sum_{k=0}^{j-1} \tau^{2k}\qquad \text{ for } j\geq 1 ,
\]
$T_0 \coloneqq \lim_{j\to \infty } t_j = T/(1-\tau^2 )$, $u^{(0)}\coloneqq u$, and
\eqnb\label{uj_rescaling_def}
u^{(j)} (x,t) \coloneqq \tau^{-j} u \left( \Gamma^{-j} (x) , \tau^{-2j} (t-t_j) \right), \qquad j\geq 1,
\eqne
see Fig. \ref{switching_process_figure_with_supp}. Clearly 
\eqnb\label{what_is_supp_uj}
\supp \, u^{(j)} (t) = \Gamma^j (G) \qquad \text{ for }t\in [t_j, t_{j+1}]
\eqne
and, as in \eqref{switch_into_u1}, \eqref{u_grows} gives that the magnitude of the consecutive vector fields shrinks at every switching time, that is
\eqnb\label{decrease_at_switching_of_ujs}
\left| u^{(j)} (x,t_j ) \right| \leq \left| u^{(j-1)} (x,t_j ) \right|, \qquad  x\in \RR^3, j \geq 1,
\eqne
see Fig. \ref{switching_process_figure_with_supp}. \\
\begin{figure}[h]
\centering
 \includegraphics[width=\textwidth]{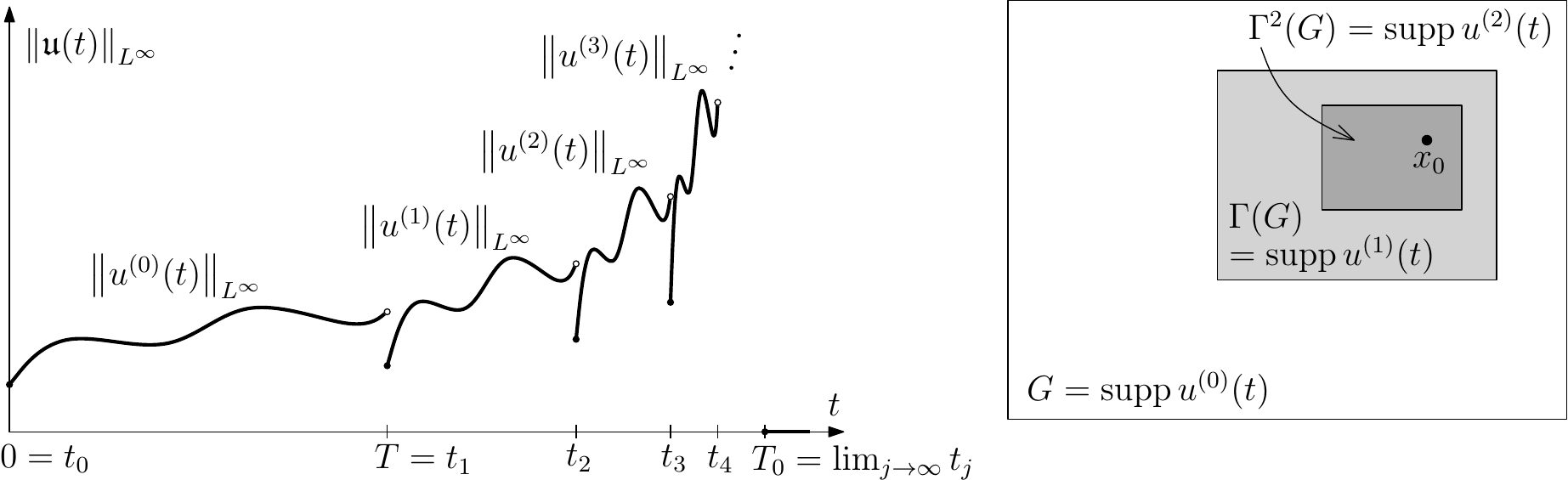}
 \nopagebreak
\captionsetup{width=0.9\textwidth}  \captionof{figure}{The switching procedure: the blow-up of $\| \mathfrak{u}(t) \|_\infty$ (left) and the shrinking support of $\mathfrak{u}(t)$ (right) as $t\to T_0^-$.}\label{switching_process_figure_with_supp} 
\end{figure}
Thus letting 
 \[
\mathfrak{u} (t)  \coloneqq \begin{cases}
 u^{(j)}(t)  \qquad &\text{ if } t\in [t_{j},t_{j+1}) \text{ for some }j\geq 0, \\
 0 &\text{ if } t\geq  T_0,
\end{cases}
\]
we obtain a vector field that satisfies the claims of Theorem \ref{point_blowup_thm}. Indeed, by construction $\mathfrak{u}$ is divergence-free, smooth in space, its support in space is contained in $G$, and $\mathfrak{u}$ is unbounded in every neighbourhood of $(x_0,T_0)$, where
\[
\{ x_0 \} \coloneqq \bigcap_{j\geq 0} \Gamma^j (G)=\lewy \frac{z}{1-\tau} \prawy .
\]
As for the regularity $\mathfrak{u} \in L^\infty ((0,\infty ); L^2 (\RR^3 ))$ and $\nabla \mathfrak{u} \in L^2 (\RR^3 \times (0,\infty ))$ (recall Definition \ref{def_weak_sol_of_NSI}) we write for any $t \in [t_j,t_{j+1}]$, $j\geq 0$,
\eqnb\label{regularity_L2_is_finite}
\| \mathfrak{u} (t) \| = \| u^{(j)} (t) \| \leq \sup_{t\in [t_{j},t_{j+1})} \| u^{(j)} (t) \| = \tau^{j/2} \sup_{t\in [t_0,t_1]} \| u^{(0)} (t) \| \leq   \sup_{t\in [t_0,t_1]} \| u^{(0)} (t) \| < \infty ,
\eqne
where we used the fact that $\tau \in (0,1)$ and we used the shorthand notation $\| \cdot \| \equiv \| \cdot \|_{L^2 (\RR^3)}$. Similarly,
\eqnb\label{regularity_L2_of_nabla}
\int_0^\infty \| \nabla \mathfrak{u} (t) \|^2  = \sum_{j=0}^\infty \int_{t_{j}}^{t_{j+1}} \| \nabla u^{(j)} (t) \|^2 = \int_{t_{0}}^{t_1} \| \nabla u^{(0)} (t) \|^2 \sum_{j=0}^\infty \tau^{j}  <\infty,
\eqne
as required. 

As for the local energy inequality \eqref{local_energy_inequality}, we see that, by construction, the local energy inequality \eqref{alternative_LEI} is satisfied on any time interval $[S,S']\subset [0,T_0)$. Since $\| \mathfrak{u}(t) \| \to 0$ as $t\to T_0^-$ (since $\tau^{2j}\to 0$ as $j\to \infty$, see the calculation above) and the regularity $\sup_{t>0} \| u(t) \| <\infty$, $\nabla u \in L^2 (\RR^3 \times (0,\infty ))$ gives global-in-time integrability of all the terms appearing under the space-time integrals in \eqref{alternative_LEI} the Dominated Convergence Theorem lets us take the limit $S'\to T_0$ to obtain the local energy inequality on any interval $[S,S']\subset [0,\infty )$, as requried.

Therefore we have established the proof of Theorem \ref{point_blowup_thm} given the existence of $T$, $G$, $u$, $\nu_0$, $\tau$ and $z$ with the properties listed above. These objects are constructed in Section \ref{sec_the_setting} (which includes a particularly enlightening proof of the Navier--Stokes inequality \eqref{NSI_u}, see Section \ref{sec_pf_of_claims_of_thm}). We now discuss some interesting properties of the vector field $\mathfrak{u}$ which are consequences of the above switching procedure.

\subsection{Remarks}\label{sec_remarks_to_the_sketch}
Note that $\mathfrak{u}$ enjoys a self-similar property
\[
\mathfrak{u}(x_0-x,T_0-s) = \tau^j \mathfrak{u}(x_0-\tau^j x , T_0 -\tau^{2j}s ),\qquad x\in \RR^3, s\in (0,T_0], j\geq 0,
\]
which is also the property characteristic for the Leray hypothetical self-focusing strong solutions to the Navier--Stokes equations (that is (3.12) in \cite{Leray_1934}, in which $x_0=0$; note however such solutions do not exist, as was shown by Ne\v{c}as, R\r{u}\v{z}i\v{c}ka \& \v{S}ver\'{a}k (1996)\nocite{necas_ruzicka_sverak}), except that here the self-similarity holds only for the discrete scaling factors $\tau^j$, $j\geq 0$.

Moreover, $\mathfrak{u}$ satisfies the energy inequality
\eqnb\label{EE_mathfraku}
\| \mathfrak{u} (\tau_2) \|_{L^2}^2 + 2 \nu \int_{\tau_1}^{\tau_2} \| \nabla \mathfrak{u} (t) \|_{L^2}^2 \d t \leq \| \mathfrak{u} (\tau_1) \|_{L^2}^2 ,\qquad \nu\in [0,\nu_0 ]
\eqne
for every $\tau_1 \in [0,\infty )$ such that $\tau_1 \not \in \{ t_j \}_{j\geq 1}$, and every $\tau_2 >\tau_1$ (where we used the shorthand notation $L^2 \equiv L^2 (\RR^3 )$), which can be verified as follows. Let $1\leq j_1 \leq j_2 $ and take 
\[\phi (x,t)= \psi (x) \mathcal{T}(t),\]
where $\psi \in C_0^\infty (\RR^3)$ is such that $\psi \geq 0$, $\psi=1 $ on $G$ and $\mathcal{T} \in C_0^\infty ((0,\infty))$ is such that $\mathcal{T}=1$ on $[t_{j_1},t_{j_2}]$ and $\supp \, \mathcal{T} \subset (t_{j_1-1},t_{j_2+1})$. Then the local energy inequality \eqref{alternative_LEI} and the fact that $\mathrm{supp}\,u(t) \subset G$ give
\[
2\nu \int_{t_{j_1-1}}^{t_{j_2+1}}  \mathcal{T} (t)   \| \nabla \mathfrak{u} (t) \|^2_{L^2} \d t \leq \int_{t_{j_1-1}}^{t_{j_1}}   \|\mathfrak{u}(t) \|^2_{L^2}  \mathcal{T}'(t) \, \d t + \int_{t_{j_2}}^{t_{j_2+1}}   \|\mathfrak{u}(t) \|^2_{L^2}  \mathcal{T}'(t) \, \d t .
\]
Given $\varepsilon >0$ and $\tau_1\in (t_{j_1-1},t_{j_1})$, $\tau_2\in (t_{j_2},t_{j_2+1})$ let $\mathcal{T}(t)\coloneqq J_\varepsilon \chi_{(\tau_1,\tau_2)}(t)$, where $\chi $ is an indicator function and $J_\varepsilon$ denotes the (usual) mollification operator. Given such a choice of $\mathcal{T}$ we can use the smoothness of $\mathfrak{u}$ on each of the intervals $(t_j,t_{j+1})$, $j\geq 0$ to take the limit $\varepsilon \to 0^+$ in the inequality above to obtain the energy inequality \eqref{EE_mathfraku} for $\tau_1 \in (t_{j_1-1},t_{j_1})$, $\tau_2 \in (t_{j_2},t_{j_2+1})$. Thus, since $\mathfrak{u}$ is right-continuous in time and its magnitude does not increase at a switching time (recall \eqref{decrease_at_switching_of_ujs}), the last inequality is valid also for $\tau_1 \in [t_{j_1-1},t_{j_1})$, $\tau_2 \in [t_{j_2},t_{j_2+1}]$, as required. 
  
Furthermore, although the vector field $\mathfrak{u}$ is not a solution of the Navier--Stokes equations, it can be used to benchmark some results in the theory of these equations, for example the regularity criteria. A regularity criterion is a condition guaranteeing that a local-in-time strong solution $u$ of the Navier--Stokes equations on a time interval $[0,T)$ does not blow-up as $t\to T^-$. For example, $u(t)$ does not blow-up if it satisfies any of the following.
\begin{enumerate}
\item[(1)] The \emph{Beale-Kato-Majda criterion} (due to \cite{beale_kato_majda}): 
\[ \int_0^T \| \mathrm{curl}\, u(t) \|_{L^\infty } <\infty, \] 
\item[(2)] The \emph{Serrin condition} (due to \cite{serrin_1963}):
\[
\int_0^T \| u(t) \|_{L^s}^r < \infty \qquad \text{ for any } s\geq 3, r\geq 2 \text{ satisfying } \frac{2}{r}+\frac{3}{s}=1,
\]
or
\item[(3)] \emph{Control of the direction of vorticity} (due to \cite{fefferman_constantin_1993}):
\eqnb\label{dir_of_vort}
\text{ for some } \Omega, \rho >0 \quad \left| P^\perp_{\xi (x,t)} (\xi (x+y,t)) \right| \leq |y|/\rho
\eqne
for $x,y,t$ such that 
\[t\in [0,T],\quad |\mathrm{curl}\,u(x,t)|, |\mathrm{curl}\,u(x+y,t)|>\Omega.\]
Here $\xi (x,t)\coloneqq \mathrm{curl}\,u(x,t)/|\mathrm{curl}\,u(x,t)|$ is the direction of vorticity $\mathrm{curl}\,u(x,t)$, and $P^\perp_x y \coloneqq \sin \alpha$, where $\alpha$ denotes the angle between the vectors $x,y\in \p B(0,1)\subset \RR^3$. 
\end{enumerate}
Remarkably, $\mathfrak{u}$ does not satisfy any of the above criteria, which is a consequence of the switching argument applied in the previous section (as for (3) above note that the direction of $\mathrm{curl}\,u^{(0)}$ is not constant and so the direction of $u^{(j)}$ cannot be controlled as in \eqref{dir_of_vort} as $j\to \infty$).

However, $\mathfrak{u}$ does satisfy the \emph{$L_{3,\infty}$ criterion} (due to \cite{ESS_2003}, see also \cite{seregin_2007,seregin_2012}): if
\[
 \| u(t) \|_{L^3 (\RR^3 )} \text{ remains bounded as } t\to T^-
\]
then $u(t)$ (a local-in-time strong solution on time interval $[0,T)$) does not blow-up as $t\to T^-$. Indeed the $L^3$ norm of $\mathfrak{u}(t)$ remains bounded by $\sup_{t\in [0,T]} \| u^{(0)} (t) \|_{L^3 (\RR^3 )}$.
This shows that the $L_{3,\infty}$ regularity criterion uses, in an essential way, properties of solutions of the Navier--Stokes equations (rather than merely the Navier--Stokes inequality \eqref{navier_stokes_inequality}).\\

In the next three sections we complete the sketch of the proof of Theorem \ref{point_blowup_thm}, that is we construct constants $T>0$, $\nu_0>0$, $\tau \in (0,1)$, $z\in \RR^3$, the set $G$ and the vector field $u$ with the properties listed in the beginning of Section \ref{sec_proof_of_thm1}. For this we first introduce a number of preliminary results regarding axisymmetric vector fields in $\RR^3$, properties of the pressure function as well as introduce the concept of a \emph{structure} on a subset $U$ of the upper half plane (Section \ref{sec_prelims}). Then, in Section \ref{sec_the_setting}, we perform the construction of $T>0$, $\nu_0>0$, $\tau \in (0,1)$, $z\in \RR^3$, $G$, $u$ and we show the required claims. The construction is based on a certain \emph{geometric arrangement}, which is the heart of the proof of Theorem \ref{point_blowup_thm} and which we discuss in detail in Section \ref{sec_geom_arrangement}.

\section{Preliminaries}\label{sec_prelims}
We will say that a function is \emph{smooth} on an open set if it is of class $C^\infty$ on this set. We use the notation $\p_\lambda$ for the partial derivative with respect to a variable $\lambda$. We often simplify the notation corresponding to the partial derivative with respect to $x_i$ by writing
\[
\p_i \equiv \p_{x_i}.
\]
We do not apply the summation convention over repeated indices. We let
\[
P\coloneqq \{ (x_1,x_2 )\in \RR^2 \colon x_2 >0  \} 
\]
denote the upper half plane. We frequently use the convention 
\eqnb\label{the_convention_with_t}
h_{t} (\cdot ) \equiv h(\cdot , t),
\eqne
that is the subscript $t$ denotes dependence on $t$ (rather than the $t$-derivative, which we denote by $\p_t$). By writing
\[ \text{``\emph{outside} } G\text{''} \qquad \text{ we mean } \qquad   \text{ ``\emph{for} } x\not \in G\text{''}.\]
By $\overline{U}$ we denote the closure of an open set $U$. We often write that a function is \emph{a solution to} a theorem (or proposition/lemma) if it satisfies the claim of the theorem.
\subsection{The rotation $R_\varphi$}\label{section_rotation}
We denote by $R_\varphi$ the rotation around the $x_1$ axis by an angle $\varphi$, that is
\[
R_\varphi (x_1,x_2,x_3) = (x_1, x_2 \cos \varphi - x_3 \sin \varphi , x_2 \sin \varphi + x_3 \cos \varphi ). 
\]
We will refer to $R_\varphi$ (for some $\varphi$) simply as \emph{the rotation}, since it is the only operation of rotation that we will consider. 
It is clear that any $x\in \RR^3$ is either a point on the $x_1$ axis, a point in $P$ or a rotation $R_\varphi (y_1,y_2,0)$ of some $y\in P$ by some angle $\varphi \in (0,2\pi )$.
For $U \Subset  {P}$ set
\eqnb\label{R_of_U}
R(U) \coloneqq \{x\in \RR^3 \, : \, x=R_\varphi (y,0) \text{ for some } \varphi \in [0,2\pi ),\, y\in U \},
\eqne
the \emph{rotation of $U$} (see Fig. \ref{rotation_projection_figure}). Clearly, if $U_1$, $U_2$ are disjoint subsets of $P$ then $R(U_1)$, $R(U_2)$ are disjoint subsets of $\RR^3$.
We will denote by $R^{-1}\colon \RR^3 \to \overline{P}$ the \emph{cylindrical projection} defined by
\eqnb\label{cylindrical_projection}
R^{-1} (y_1,y_2,y_3) \coloneqq \left( y_1, \sqrt{y_2^2+y_3^2} \right) .
\eqne
The projection $R^{-1}$ is in fact the left-inverse of $R$, that is $R^{-1} R = \mathrm{id}$. It is not a right-inverse, but $R \, R^{-1} (V) \supset V$ for any $V\subset \left( \RR^3 \setminus Ox_1 \right)$ (where $Ox_1$ denotes the $x_1$ axis), as is clear from Fig. \ref{rotation_projection_figure}.
\begin{figure}[h]
\centering
 \includegraphics[width=\textwidth]{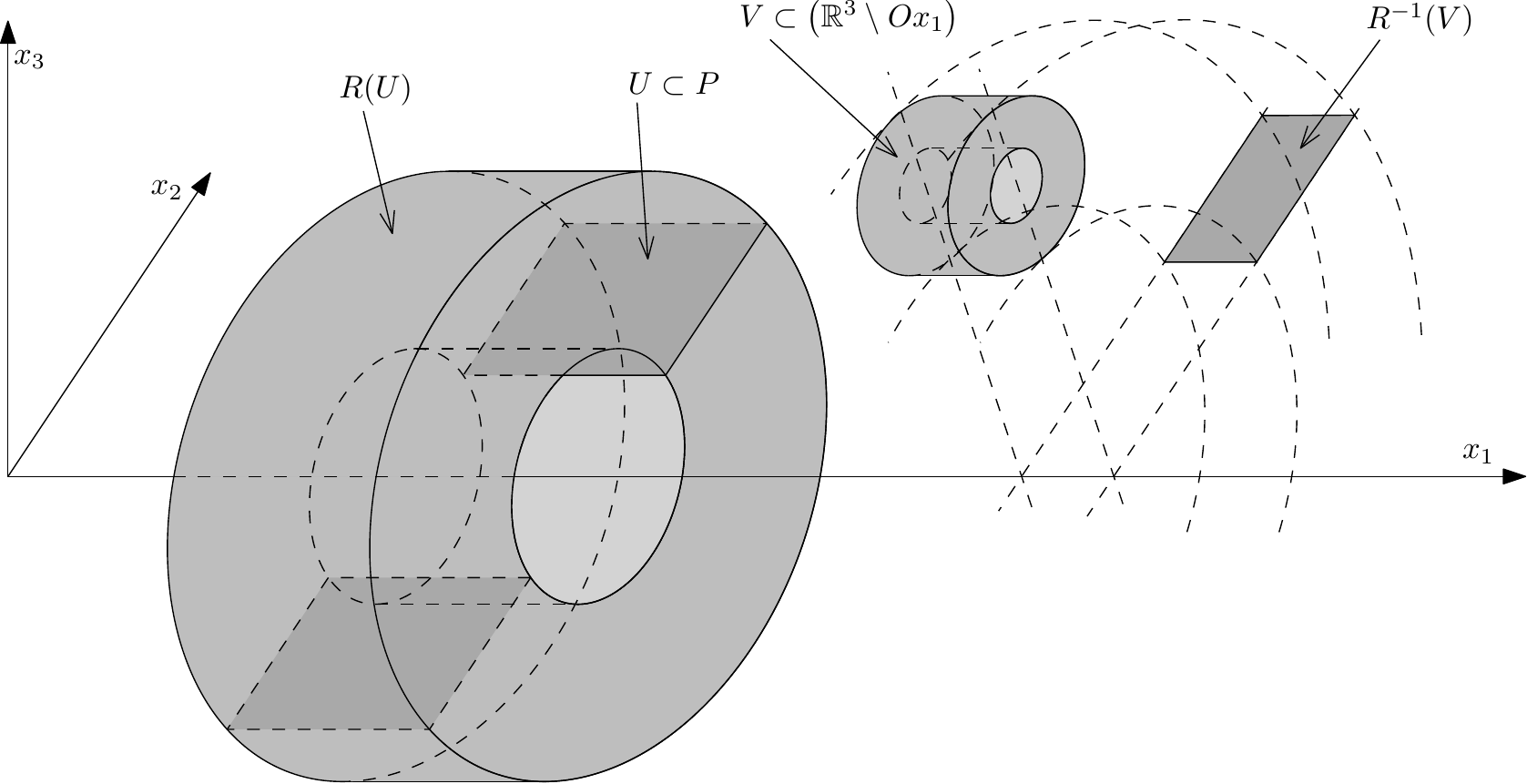}
 \nopagebreak
  \captionof{figure}{The rotation $R$ and the cylindrical projection $R^{-1}$.}\label{rotation_projection_figure} 
\end{figure}
We say that a velocity field $u\colon \RR^3 \to \RR^3$ is \emph{axisymmetric} if
\eqnb\label{rot_invariant_velocity}
u (R_\varphi x) = R_\varphi u(x) \qquad \text{ for } \varphi \in [0,2 \pi ), x\in \RR^3,
\eqne
while a scalar function $q \colon \RR^3 \to \RR$ is \emph{axisymmetric} if
\[
q (R_\varphi x) =  q(x) \qquad \text{ for } \varphi \in [0,2 \pi ), x\in \RR^3;
\]
in other words $q(x) = q(R^{-1} x)$ for $x\in \RR^3$.
Observe that if a vector field $u\in C^2$ and a scalar function $q\in C^1$ are axisymmetric then the vector function $(u\cdot \nabla )u$ and the scalar functions 
\eqnb\label{scalar_fcns_rotationally_invariant}
|u|^2,\quad \mathrm{div}\, u,\quad u\cdot \nabla | u|^2,\quad u\cdot \nabla q,\quad u\cdot \Delta u \quad\text{ and }\quad \sum_{i,j=1}^3 \p_i u_j \p_j u_i 
\eqne
are axisymmetric. These facts can be verified by a simple calculation and by making use of the algebraic identity
\[
\sum_{i,j=1}^3 \p_i u_j \p_j u_i = \mathrm{div} \, ((u\cdot \nabla )u) - u \cdot \nabla (\mathrm{div} \,u),
\]
see Appendix \ref{app_prelim_calc} for details.
\subsection{The pressure function}\label{sec_prop_of_p}
Given a vector field $u\colon \RR^3 \to \RR^3$ consider the pressure function $p\colon \RR^3 \to \RR$ corresponding to $u$, that is
\[
p(x) \coloneqq \int_{\RR^3} \sum_{i,j=1}^3 \frac{ \p_i u_j (y) \p_j u_i (y)}{4\pi |x-y |}\,  \d y,
\]
recall \eqref{def_of_the_corresponding_pressure}. Here we briefly comment on some geometric properties of the pressure function, which will be crucial in constructing a velocity field $u$ satisfying the Navier--Stokes inequality \eqref{NSI_u} (see for instance Lemma \ref{lem_properties_of_F}). 

First, if $u\in C_0^\infty (\RR^3)$ then the corresponding pressure function is smooth on $\RR^3$ with  
\eqnb\label{decay_of_the_pressure_function}
|\nabla p(x)| \leq \widetilde{C} |x|^{-4}\quad \text{ and }\quad |D^2 p (x) |\leq \widetilde{C} |x|^{-5}
\eqne
for some $\widetilde{C}>0 $ (which depends on $u$), which follows from integration by parts. Moreover, $p$ satisfies the limiting property
\eqnb\label{property_of_p_magic}
\lim_{x_1 \to \pm \infty } x_1^4 \p_1 p(x_1,0,0) = \frac{\pm 3}{4\pi } \int_{\RR^3 } \left( |u (y)|^2- 3  u_1^2 (y) \right) \d y,
\eqne
which can be verified directly.
Finally, if $u\in C_0^\infty (\RR^3)$ is axisymmetric then the change of variable $z=R_{-\varphi}y$ and \eqref{scalar_fcns_rotationally_invariant} give
\eqnb\label{corr_pressure_is_rot_inv}
p(R_\varphi x) = \int_{\RR^3} \sum_{i,j=1}^3 \frac{ \p_i u_j (y) \p_j u_i (y)}{4\pi |R_\phi x-y |} \d y= \int_{\RR^3} \sum_{i,j=1}^3 \frac{ \p_i u_j (R_\phi z) \p_j u_i (R_\phi z)}{4\pi | x-z |} \d z = p (x)
\eqne
for all $\varphi \in [0,2 \pi )$. That is the pressure function corresponding to a axisymmetric vector field is axisymmetric.
\subsection{The functions $u{[}v,f{]}$, $p{[}v,f{]}$}\label{u[v,f]_section}
Now let $v$ be a $2$D vector field and $f$ be a scalar function defined on $P$ such that
\eqnb\label{requirements_on_v_f}
v\in C_0^\infty (P; \RR^2), f\in C_0^\infty (P; [0,\infty ) ), \text{ and } f>|v|\text{ on }\supp \, v. 
\eqne 
For such $v,f$ we define $u[v,f] \colon \RR^3 \to \RR^3$ to be the axisymmetric vector field satisfying
\eqnb\label{def_of_us_previous}
u[v,f] (x_1,x_2,0) \coloneqq  \left( v_1 (x_1,x_2), v_2 (x_1,x_2), \sqrt{f(x_1,x_2)^2-|v(x_1,x_2)|^2} \right)
\eqne
for $(x_1,x_2) \in \overline{P}$. (Here $\overline{P}$ denotes the closure of $P$.) Note that such definition immediately gives
\eqnb\label{abs_value_of_u_is_f}
 |u[v,f]|=f. \eqne
 We emphasize that \eqref{requirements_on_v_f} also implies that $\mathrm{supp}\,u[v,f]$ is strictly separated from the $Ox_1$ axis (the rotation axis).\\
 
Moreover, the definition can we rewritten in a simple, equivalent form using cylindrical coordinates $x_1,\rho, \varphi$. Namely
\eqnb\label{def_of_us}
u[v,f](x_1,\rho , \varphi) =  v_1 (x_1,\rho ) \widehat{x}_1 + v_2 (x_1,\rho ) \widehat{\rho } + \sqrt{f(x_1,\rho)^2-|v(x_1,\rho)|^2} \,\widehat{\varphi } ,
\eqne
where the cylindrical coordinates are defined using the representation
\[
\begin{cases}
x_1 = x_1, \\
x_2 = \rho \cos \varphi, \\
x_3 = \rho \sin \varphi
\end{cases}
\]
and the cylindrical unit vectors $\widehat{x}_1$, $\widehat{\rho}$, $\widehat{\varphi}$ are 
\eqnb\label{def_of_cylindrical_unit_vectors}
\begin{cases}
\widehat{x}_1 (x_1,\rho, \varphi ) \coloneqq (1,0,0), \\
\widehat{\rho }  (x_1,\rho, \varphi ) \coloneqq (0,\cos \varphi, \sin \varphi ), \\
\widehat{\varphi }  (x_1,\rho, \varphi ) \coloneqq (0,-\sin \varphi, \cos \varphi) .
\end{cases}
\eqne
%see Fig. \ref{cylindrical_coords_figure}.
%\begin{figure}[h]
%\centering
% \includegraphics[width=\textwidth]{pics/cylindrical_coords.pdf}
% \nopagebreak
%  \captionof{figure}{The cylindrical coordinates.}\label{cylindrical_coords_figure} 
%\end{figure}
In particular, for this coordinate system the chain rule gives
\eqnb\label{chain_rule}
\begin{split}
\p_{\rho } &= \cos \varphi \, \p_{x_2} + \sin \varphi \, \p_{x_3},\\
\p_{\varphi } &= - \rho \sin \varphi \, \p_{x_2} + \rho \cos \varphi \,\p_{x_3} .
\end{split}
\eqne
Clearly, if $\supp \, v , \, \supp \, q \subset U$ for some $U\subset P$ then $ \supp \, u[v,q] \subset R(U)$.  Moreover, since both $v$ and $f$ have compact support in $P$ and since $f>|v|$ on $\supp \, v$ (so that $\sqrt{f^2-|v|^2} \in C_0^\infty (P)$) it is clear that $u[v,f]\in C^\infty_0 (\RR^3 ; \RR^3)$. The vector field $u[v,f]$ enjoys some further useful properties, which we show below. 
\begin{lemma}[Properties of $u{[}v,f{]}$]\label{properties_of_u[v,f]}$\mbox{}$
\begin{enumerate}
\item[\textnormal{(i)}] The vector field $u[v,f]$ is divergence free if and only if $v$ satisfies
\[
\mathrm{div} ( x_2 \, v (x_1,x_2) ) =0
\qquad \text{ for all }(x_1,x_2)\in P.\]
\item[\textnormal{(ii)}] If $v\equiv 0$ then
\[
\Delta u[0,f] (x_1,\rho , \varphi ) = Lf(x_1,\rho) \widehat{\varphi},
\]
where 
\eqnb\label{def_of_L}
Lf(x_1,x_2) \coloneqq \Delta f(x_1,x_2) + \frac{1}{x_2} \p_{x_2} f (x_1,x_2) - \frac{1}{x_2^2} f(x_1,x_2).
\eqne
In particular
\eqnb\label{Delta_u_equals_Lf}
\Delta u[0,f] (x_1,x_2,0) = (0,0,Lf(x_1,x_2)).
\eqne
\item[\textnormal{(iii)}] For all $x_1,x_2\in \RR$
\eqnb\label{d3_of_|u|_vanishes}
\p_{x_3} |u[v,f]| (x_1,x_2,0) =0.
\eqne
\end{enumerate}
\end{lemma}
\begin{proof} The lemma is a consequence of elementary calculations using cylindrical coordinates, which we now briefly discuss. 

As for (i) recall that the divergence of a vector field $u$ described in cylindrical coordinates as $u= u_1 \widehat{x}_1 + u_{\rho } \widehat{\rho } + u_{\varphi } \widehat{\varphi }$ is
\[
\mathrm{div}\, u = \p_{x_1} u_1 + \frac{1}{\rho } \p_\rho  \left( \rho u_{\rho } \right) + \frac{1}{\rho } \p_\varphi u_\varphi .
\]
Thus since $u[v,f]_\varphi = \sqrt{f^2-|v|^2}$ does not depend on $\varphi $ we obtain (i).

As for (ii) recall that the Laplacian of any function $F=F(x_1,\rho,\varphi)$ is 
\[
\Delta F =\p_{x_1 x_1}F+ \frac{1}{\rho} \p_\rho \left( \rho \,\p_\rho F \right) + \frac{1}{\rho^2} \p_{\varphi \varphi} F .
\]
Thus, since $u[0,f]=f\widehat{\varphi}$ and because the unit vector $\widehat{\varphi} $ depends only on $\varphi$ and satisfies $\p_{\varphi \varphi } \widehat{\varphi } = - \widehat{\varphi }$ (recall \eqref{def_of_cylindrical_unit_vectors}) we obtain
\[
\begin{split}
\Delta u[0,f] &= \p_{x_1 x_1}f(x_1,\rho) \widehat{\varphi}+ \frac{1}{\rho} \p_\rho \left( \rho \,\p_\rho f(x_1,\varphi ) \right) \widehat{\varphi } + \frac{f(x_1,\rho)}{\rho^2} \p_{\varphi \varphi} \widehat{\varphi} \\
&= \p_{x_1 x_1}f(x_1,\rho) \widehat{\varphi}+ \frac{1}{\rho} \p_\rho f(x_1,\varphi )  \widehat{\varphi }+ \p_{\rho \rho} f(x_1,\varphi )  \widehat{\varphi } - \frac{f(x_1,\rho)}{\rho^2}  \widehat{\varphi} \\
&= Lf(x_1,\rho ) \widehat{\varphi }.
\end{split}
\]
In particular, taking $\varphi =0$ gives \eqref{Delta_u_equals_Lf}.

As for (iii) it is enough to note that, since $|u[v,f]|=f(x_1,\rho)$ is axisymmetric, the derivative in question is in fact a derivative along a level set of $|u[v,f]|$ (that is along a a circle around the $x_1$ axis). In other words the relations \eqref{chain_rule} give
\eqnb\label{what_is_deriv_in_x_3}
\p_{x_3} = \sin \varphi \,\p_\rho + \frac{\cos \varphi }{\rho } \p_\varphi 
\eqne
and so, because $|u[v,f]|=f$ does not depend on $\varphi$,
\[
\p_{x_3} |u[v,f]| = \sin \varphi \left( \p_\rho f \right),
\]
which vanishes when $\varphi =0,\pi$. \qedhere
\end{proof}
We define $p^* [v,f] \colon \RR^3 \to \RR$ to be the pressure function corresponding to $u[v,f]$, that is 
\eqnb\label{def_of_p*}
p^*[v,f] (x) \coloneqq  \int_{\RR^3} \sum_{i,j=1}^3 \frac{ \p_i u_j [v,f] (y) \p_j u_i[v,f] (y)}{4\pi |x-y|} \d y ,\\
\eqne
and we denote its restriction to $\RR^2$ by $p[v,f]$, that is
\eqnb\label{def_of_ps}
p[v,f](x_1,x_2)\coloneqq  p^*[v,f] (x_1,x_2,0) .
\eqne
It is clear that, since $u[v,f]\in C_0^\infty (\RR^3 )$,
\eqnb\label{pressure_p[v,f]_is_smooth}
p[v,f]\in C^\infty (\RR^2 )
\eqne
Furthermore, since $u[v,f]$ is axisymmetric, the same is true of $p^*[v,f]$ (recall \eqref{corr_pressure_is_rot_inv}). In particular, in the same way as in the proof of Lemma \ref{properties_of_u[v,f]} (iii) above, we obtain that
\eqnb\label{dx3_of_p_is_zero}
\p_{x_3} p^* [v,f] (x_1,x_2,0) =0 \quad \text{ for all } x_1,x_2\in \RR .
\eqne
Similarly, 
\eqnb\label{dx2_of_p_is_zero}
\p_{x_2} p^* [v,f] (x_1,0,x_3) =0\quad \text{ for all } x_1,x_3\in \RR ,
\eqne
using the relation 
\eqnb\label{what_is_deriv_in_x_2} \p_{x_2} = \cos \varphi \,\p_\rho -\frac{\sin \varphi }{\rho } \p_\varphi, \eqne
which is a consequence of \eqref{chain_rule}. Thus taking $x_3=0$ in \eqref{dx2_of_p_is_zero} we obtain
\eqnb
\p_{x_2} p[v,f] (x_1,0)=0 \quad \text{ for } x_1 \in \RR .
\eqne
The function $p[v,f]$ enjoys some further properties, which we state in a lemma.
\begin{lemma}[Properties of $p{[}v,f{]}$]\label{lem_prop_of_p[v,f]}
Let $v=(v_1,v_2),f$ be as in \eqref{requirements_on_v_f}. Then
\begin{enumerate}
\item[\textnormal{(i)}] $p[v,f]=p[-v,f]$,
\item[\textnormal{(ii)}] if additionally $v_2(\cdot ,x_2)$ is odd and $v_1(\cdot ,x_2)$, $f(\cdot, x_2)$ are even for each fixed $x_2$ then $p[v,f]$ is even, that is
\[
p[v,f](x)=p[v,f](-x) \qquad \text{ for all }x\in \RR^2,
\]
\item[\textnormal{(iii)}] if $\widetilde{v},\widetilde{f}$ is another pair satisfying \eqref{requirements_on_v_f} and such that $f,\widetilde{f}$ have disjoint supports then
\[
p\left[ v+\widetilde{v},f+\widetilde{f} \right]= p[v,f]+ p\left[ \widetilde{v},\widetilde{f}\right] .
\]
\end{enumerate}
\end{lemma}
\begin{proof}
Property (iii) follows directly from the definition \eqref{def_of_p*}. As for (i) we will show that $p^* [v,f]=p^*[-v,f]$. Substituting  \eqref{def_of_cylindrical_unit_vectors} into \eqref{def_of_us} we obtain 
\eqnb\label{what_is_u1_u2_u3}
\begin{split}
u_1[v,f] (x_1,\rho ,\varphi ) &= v_1 (x_1,\rho ),\\
u_2[v,f] (x_1,\rho , \varphi ) &= v_2 (x_1,\rho ) \cos \varphi - \sqrt{f^2-|v|^2}(x_1, \rho ) \sin \varphi ,\\
u_3[v,f] (x_1,\rho , \varphi ) &= v_2 (x_1,\rho ) \sin \varphi + \sqrt{f^2-|v|^2}(x_1, \rho ) \cos \varphi .
\end{split}
\eqne
Thus since for $\varphi =0$ we have $\p_{2}=\p_\rho$, $\p_{3} = \rho^{-1} \p_\varphi$ (see \eqref{what_is_deriv_in_x_3}, \eqref{what_is_deriv_in_x_2}) and we obtain
\eqnb\label{temp_calculating_diuj}
\begin{array}{lll}
\p_1 u_1[v,f] =\p_{x_1} v_1 ,\quad &\p_2 u_1 [v,f]=\p_\rho v_1,\quad  &\p_3 u_1 [v,f]=0,\\
\p_1 u_2 [v,f]=\p_{x_1} v_2,&\p_2 u_2 [v,f]=\p_{\rho} v_2, &\p_3 u_2[v,f] =-\sqrt{f^2-|v|^2}/\rho ,\\
\p_1 u_3[v,f] =\p_{x_1} \sqrt{f^2-|v|^2} ,&\p_2 u_3 [v,f]=\p_{\rho}  \sqrt{f^2-|v|^2}, &\p_3 u_3 [v,f]=v_2/\rho,
\end{array}
\eqne
from which we immediately see that
\[
\p_i u_j [v,f] \,\p_j u_i [v,f]  =  \p_i u_j [-v,f]\, \p_j u_i [-v,f] 
\]
for any choice of $i,j\in \{ 1,2,3\}$. Summation in $i,j$ gives
\[
 \sum_{i,j=1}^3 \p_i u_j [v,f] \,\p_j u_i [v,f]  =  \sum_{i,j=1}^3 \p_i u_j [-v,f]\, \p_j u_i [-v,f] \qquad \text{ for }\varphi =0,
\]
and the axisymmetry of each of the two sums (see \eqref{scalar_fcns_rotationally_invariant}) gives the equality everywhere in $\RR^3$. Consequently, we obtain
\[
p^* [-v,f] = p^* [v,f],
\]
as required.

As for (ii), we will show that 
\eqnb\label{temp_sum_is_even}
 \left( \sum_{i,j=1}^3 \p_i u_j [v,f] \,\p_j u_i [v,f] \right) (x_1, \rho ) = \left( \sum_{i,j=1}^3 \p_i u_j [v,f] \,\p_j u_i [v,f] \right) (-x_1, \rho ),\qquad x_1\in \RR, \rho >0, 
\eqne
where we skipped the $\varphi$ in the variable (recall that this sum is independent of $\varphi$ due to the axisymmetry \eqref{scalar_fcns_rotationally_invariant}). In other words, the sum 
\[
 \sum_{i,j=1}^3 \p_i u_j [v,f] \,\p_j u_i [v,f]
\]
is an even function (recall that in cylindrical coordinates $\rho = \sqrt{ x_2^2+x_3^2}$ takes the same value for $x$ and $-x$) and so consequently $p^* [v,f]$ is even on $\RR^3$ (by definition, see \eqref{def_of_p*}). Then in particular $p[v,f]$ is even on $\RR^2$, as required. Thus it suffices to show \eqref{temp_sum_is_even}. 

To this end take $(-x_1,\rho ,0)$ as the variable in \eqref{temp_calculating_diuj} to obtain the same expressions as in the case of $(x_1,\rho ,0)$, except for the diagonal expressions, which are now of the opposite sign. This, however, makes no change to the sum
\[
 \sum_{i,j=1}^3 \p_i u_j [v,f] \,\p_j u_i [v,f],
\]
that is
\[
 \left( \sum_{i,j=1}^3 \p_i u_j [v,f] \,\p_j u_i [v,f] \right) (x_1, \rho ,0 ) = \left( \sum_{i,j=1}^3 \p_i u_j [v,f] \,\p_j u_i [v,f] \right) (-x_1, \rho ,0 ),
\]
and thus \eqref{temp_sum_is_even} follows from the axisymmetry.
\end{proof}

Finally we point out that $u[v,f]$ amd $p[v,f]$ enjoy some useful properties regarding continuity with respect to $f$. The point is that given a sequence of $v_k$'s and $f_k$'s one can obtain convergence of
\[
u[v_k,f_k ]\cdot \Delta u [v_k,f_k] - u[v_k,f ]\cdot \Delta u [v_k,f] 
\]
and 
\[
\nabla p[v_k ,f_k] - \nabla p[v_k,f]
\]
given convergence $f_k\to f$ (i.e. no convergence on $v_k$ is requried). The proof of such a result is easy but technical (involving some calculations in cylindrical coordinates) and, since we will only use it (in \eqref{conv_2}, \eqref{cantor_conv_2} and \eqref{conv_1}, \eqref{cantor_conv_1} below, respectively) in a rather specific setting, we discuss it only in Appendix \ref{sec_continuity_of_p,u_wrt_f}.
\subsection{A structure on $U\Subset P$}\label{sec_a_structure}
The definitions in the previous section give rise to a way of defining a smooth, divergence-free velocity field $u$ supported on $R(\overline{U})$, for $U\Subset P$. The following notion of a structure is a part of our simplified approach to the constructions.
\begin{df}\label{def_of_structure}
A \emph{structure on $U\Subset  P$} is a triple $(v,f,\phi )$, where $v \in C_0^\infty (U; \RR^2 )$, $f\in C_0^\infty (P; [0,\infty ))$, $\phi \in C_0^\infty (U; [0,1] )$ are such that $\supp \, f = \overline{U}$,
\[\begin{split}
&\supp \, v \subset \{ \phi =1 \},\qquad \mathrm{div}\, (x_2\, v(x_1,x_2)) =0 \text{ in } U \qquad \text{ and }\\
&f>|v| \quad \text{ in } U\qquad \text{ with } \qquad Lf >0 \text{ in } U \setminus \{ \phi =1 \}.
\end{split} \]
\end{df}
Note that $(av,f,\phi )$ is a structure for any $a\in (-1,1)$ whenever $(v,f,\phi)$ is.

Furthermore, given $(v,f,\phi )$, a structure on $U$, the velocity field $u[v,f]$ is divergence free and is supported in $R(\overline{U} )$. Moreover in $R(\{ \phi < 1 \})$
\eqnb\label{prop_of_structure_1}
u[v,f] \cdot \Delta u [v,f] \geq 0
\eqne
and
\eqnb\label{prop_of_structure_2}
u[v,f]\cdot \nabla q =0
\eqne
 for any rotationally symmetric function $q\colon \RR^3 \to \RR$. This last property is particularly useful when taking $q\coloneqq |u[v,f]|^2+2p[v,f]$ as in this way the left-hand side of \eqref{prop_of_structure_2} is of the same form as one of the terms in the Navier--Stokes inequality \eqref{NSI_u}. In order to see \eqref{prop_of_structure_1}, \eqref{prop_of_structure_2} first note that, due to axisymmetry it is enough to verify that
 \[  u[v,f](x_1,x_2,0) \cdot \Delta u [v,f](x_1,x_2,0) \geq 0
\]
and
\[
u[v,f](x_1,x_2,0)\cdot \nabla q(x_1,x_2,0) =0
\]
for $(x_1,x_2)\in \{ \phi <1 \}$ (recall \eqref{scalar_fcns_rotationally_invariant}). Since $v=0$ in $\{ \phi <1 \}$ we have $u[v,f](x_1,x_2,0) = (0,0,f(x_1,x_2))$ (recall \eqref{def_of_us_previous}), and so obtain the first of the above properties by writing
\eqnb\label{fLf_is_nonnegative}
u[v,f](x_1,x_2,0) \cdot \Delta u [v,f](x_1,x_2,0) = f(x_1,x_2) Lf (x_1,x_2) \geq 0
\eqne
where we used Lemma \ref{properties_of_u[v,f]} (ii). The second property follows in the same way by noting that $\p_{x_3} q (x_1,x_2,0) =0$ (as a property of an axisymmetric function, which can be obtained in the same way as \eqref{d3_of_|u|_vanishes}).

Furthermore, note that given $U$, the $L^\infty $ norm of derivatives of $u[v,f]$ can be bounded above by a constant depending only on $W^{1,\infty}$ norm of $v$ and $f$, that is
\eqnb\label{bound_on_grad_of_u[v,f]}
\| \nabla u[v,f ] \|_{L^\infty} \leq C \left( \| v\|_{W^{1,\infty}}, \| f \|_{W^{1,\infty}} \right),
\eqne
see \eqref{temp_calculating_diuj}. Note also that the constant depends on $U$ only in terms of its distance from the $x_1$ axis.

 \subsection{A recipe for a structure}\label{sec_recipe_for_structure}

In what follows we will only consider functions $v$, $f$ and sets $U\Subset P$ such that for some $\phi$ the triple $(v,f,\phi )$ is a structure on $U$. Moreover, we will only consider sets $U$ in the shape of a rectangle or a ``rectangular ring'', that is $V\setminus \overline{W}$, where $V$, $W$ are open rectangles and $W\Subset V$. One can construct structures on such sets in a generic way, which we now describe.

First construct $v\in C_0^\infty (U;\RR^2)$ satisfying $\mathrm{div}\, (x_2 \,v(x_1,x_2)) =0$ for all $(x_1,x_2)\in U$. For this it is enough to take a mollification of $w$ and divide it by $x_2$, where $w\colon U \to \RR^2$ is a compactly supported and weakly divergence free vector field, that is $\int_P w \cdot \nabla \psi =0$ for every $\psi \in C_0^\infty (P;\RR )$. Indeed, then the mollification of $w$ is divergence-free and thus $\mathrm{div}\, (x_2\, v(x_1,x_2))=0$. As for the construction of $w$ take, for example,
\[
w\coloneqq (x_2-3,x_1) \chi_{1<|x-(0,3)|<2},
\] 
where $\chi $ denotes the indicator function, see Fig.\ \ref{vector_fields_w}.\ Note that $w$ is weakly divergence free due to the fact that $w\cdot n$ vanishes on the boundary of the support of $w$, where $n$ denotes the respective normal vector to the boundary. Alternatively, define $w$ to be a ``regionwise'' constant velocity field
\[
w \coloneqq \begin{cases}
(1,0) \qquad & \text{ in } R_1,\\
(0,1) \qquad & \text{ in } R_2,\\
(-1,0) \qquad & \text{ in } R_3,\\
(0,-1) \qquad & \text{ in } R_4,
\end{cases}
\]
where $R_1,R_2,R_3,R_4$ are arranged as in Fig. \ref{vector_fields_w}.
\begin{figure}[h]
\centering
 \includegraphics[width=0.8\textwidth]{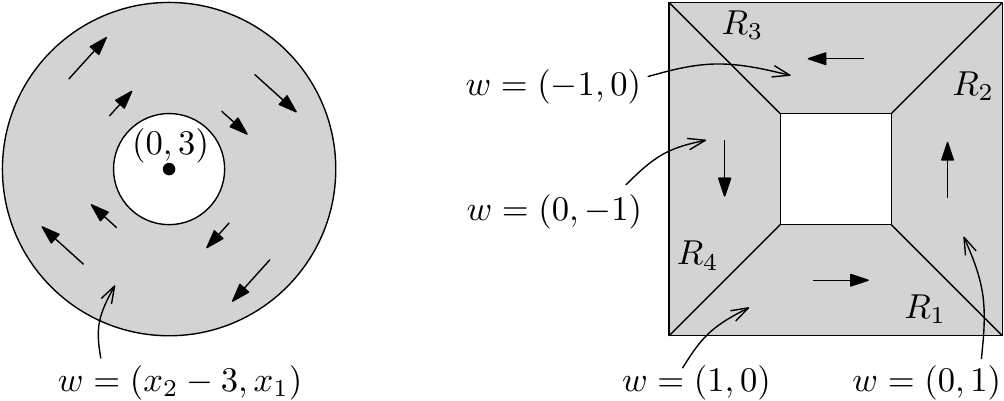}
 \nopagebreak
  \captionof{figure}{Constructing compactly supported, weakly divergence free vector field $w$.}\label{vector_fields_w} 
\end{figure}

An integration by parts and the use of the crucial property of $w\cdot n$ being continuous across the boundary between each pair of neighbouring regions $R_1,R_2,R_3,R_4, P\setminus \bigcup_i R_i$ immediately shows that such a $w$ is weakly divergence free. An advantage of such a definition of $w$ (as compared to the previous one) is that it can be ``stretched geometrically'' in a sense that given $\varepsilon >0$ one can modify $w$ to obtain $w=(1,0)$ in any given strict subset of $P$ and $|w|<\varepsilon$ whenever $v$ has a direction other than $(1,0)$, see Fig. \ref{vector_field_deformed_w}. We will later see an important sharpening of this observation (see Lemma \ref{lemma_existence_of_v2}). \begin{figure}[h]
\centering
 \includegraphics[width=0.6\textwidth]{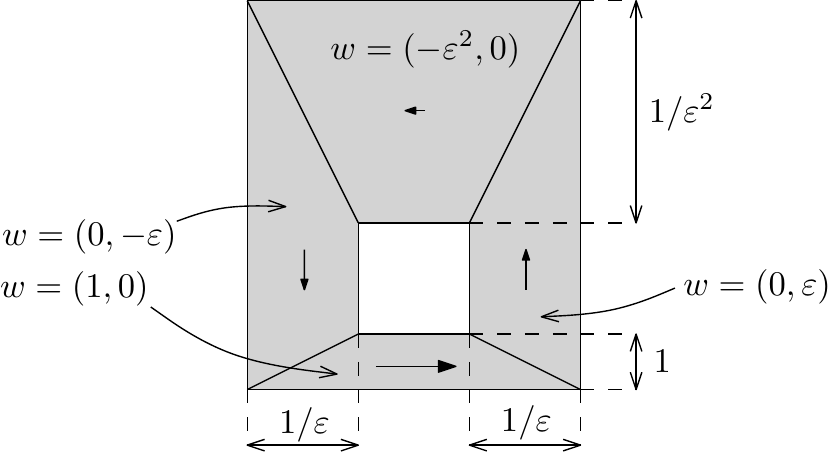}
 \nopagebreak
  \captionof{figure}{Deforming the vector field $w$.}\label{vector_field_deformed_w} 
\end{figure}

Secondly, let $\mu,\eta >0$ be such that $\mu > \| v \|_\infty$ and $\supp \, v \subset U_\eta$, where 
\eqnb\label{def_of_a_eta_subset_of_U}
U_\eta \coloneqq \{ x\in U \colon \mathrm{dist}\,(x,\partial U ) > \eta \}
\eqne
denotes the \emph{$\eta$-subset} of $U$, and let $f\in C_0^\infty (P;[0,\infty ))$ be a certain cut-off function (in $U$) that has a particular behaviour near $\partial U$. Namely, let $f$ be given by the following theorem.
\begin{theorem}\label{thm_existence_of_f_ClaimX}
Let $U\Subset P$ be an open set that is in the shape of a rectangle or $U=V\setminus \overline{W}$ for some open rectangles $V,W \Subset P$ with $W\Subset V$. Given $\eta >0$ there exists $\delta \in (0,\eta )$ and $f\in C_0^\infty (P;[0,1])$ such that
\[
\supp \, f = \overline{U},\quad f>0 \text{ in } U \text{ with } f=1 \text{ on } U_\eta
\]
and
\[
Lf >0 \quad \text{ in } U \setminus U_\delta .
\]
\end{theorem}
The proof of the theorem is elementary in nature, but requires some technicalities, in particular a generalised form of the Mean Value Theorem (see Lemma \ref{lemma_gen_MVT}). We prove the theorem in Appendix \ref{sec_claimX} (see Lemma \ref{lemma_existence_of_f_with_Lf_rectangle} for the case of a rectangle and Lemma \ref{lemma_existence_of_f_with_Lf_rings} for the case of a rectangular ring).

Finally, having defined $v$ and $f$, one can simply take any cut-off function $\phi \in C_0^\infty (U;[0,1])$ such that $\phi=1$ on $U_\delta$. Thus we obtain a structure $(v,f,\phi )$ on $U$. Note that the choice of (sufficiently large) $\mu =\| f\|_\infty $ is arbitrary.
\subsection{The pressure interaction function $F[v,f]$}\label{sec_pressure_int_fcn}
As in the case of the notion of a structure $(v,f,\phi )$ on a set $U\Subset P$, we simplify Scheffer's approach by introducing the notion of a \emph{pressure interaction function corresponding to }$U$, 
 \eqnb\label{def_of_p_interaction_fcn}
 F [v,f] \coloneqq \nabla p[0,f]-\nabla p[v,f],
 \eqne
 where $\nabla$ denotes the two-dimensional gradient. Note that $F[v,f]$ depends on the structure $(v,f,\phi )$ on $U$, and thus a set $U\Subset P$ can possibly have more than one pressure interaction function. It is not clear whether $F [v,f]$ has any physical interpretation, but this is the tool that will form certain interactions between subsets of $P$ (see the comments following Theorem \ref{thm_existence_of_osc_proc}), and we will see later that, in a sense, the strength of this interaction can be adjusted by manipulating the subsets and their corresponding structures (see the comments following \eqref{temp_calc2} and the subsequent Sections \ref{sec_copies_of_U}-\ref{sec_constr_f2_and_rest}). 
 
 We now show that $F[v,f]$ enjoys a number of useful properties, which include estimates of its size at points near the $x_1$ axis.
 \begin{lemma}[Properties of the pressure interaction function $F {[}v,f{]}$]\label{lem_properties_of_F}
Let $(v,f,\phi )$ be a structure on some $U\Subset P$ such that $v_1\not \equiv 0$. Then the pressure interaction function $F \coloneqq F [v,f]$ satisfies 
\begin{enumerate}
\item[\textnormal{(i)}] $F  \in C^\infty (\RR^2 ; \RR^2 )$ and
\[
\lim_{x_1 \to \pm \infty } x_1^4 F_1 (x_1,0) = \frac{\pm 9}{4\pi} \int_{\RR^3 }  v_{1}^2 \left( y_1, \sqrt{y_2^2+y_3^2} \right)  \d y =: \pm D.
\]
\item[\textnormal{(ii)}] $F_1$ restricted to the $x_1$ axis attains a positive maximum, that is there exists $B>0$, $A\in \RR $ such that
\[
B=F_1(A,0)=\max_{x_1\in \RR}\, F_1 (x_1,0).
\]
\item[\textnormal{(iii)}] There exists $C>0$ such that
\[
|F (x) | \leq C/|x|^4 , \qquad |\nabla F (x) | \leq C/ |x|^5 \quad \text{ for } x\in \RR^2.
\]
(Note $C\geq D$.)
\item[\textnormal{(iv)}] $F_2 (x_1,0)=0$ for $x_1 \in \RR$.
\item[\textnormal{(v)}] Let
\eqnb\label{def_of_kappa}
\kappa \coloneqq 10^4 C/D.
\eqne
There exists $N>0$ such that for $n\geq N$
\[
|x_1-n |<\kappa ,\,|x_2|<1 \text{ implies } |F_1 (x_1,x_2) - n^{-4} D | \leq 0.001 n^{-4} D .
\]
\end{enumerate}
\end{lemma}
\begin{proof}
Claim (ii) follows from (i) and the assumption $v_{1} \not \equiv 0$. As for (i), the smoothness of $F $ follows directly from the fact that $(v,f,\phi)$ is a structure on $U$, and the limiting property as $x_1\to \pm \infty$ follows by using \eqref{property_of_p_magic}, from which we obtain
\[
\begin{split}
\lim_{x_1 \to  \infty} x_1^4 F_1 (x_1,0) &= \frac{ 3}{4\pi } \int_{\RR^3 } \left( |u [0,f]|^2  -3(u_1 [0,f])^2  - |u [v,f]|^2  +3 (u_1 [v,f])^2 \right) \d y \\
&=  \frac{ 9}{4\pi } \int_{\RR^3 } v_{1} (R^{-1}(y))  \d y ,
\end{split}
\]
where we also used the facts $|u[v,f](y)|=f(R^{-1}(y))$, $u_1[v,f](y)=v_1(R^{-1}(y))$ (see \eqref{def_of_us}). The case of the limit  $x_1\to -\infty$ is similar.

Claim (iii) follows from the decay properties of the pressure function, see \eqref{decay_of_the_pressure_function}. Claim (iv) follows directly from \eqref{dx2_of_p_is_zero}. 

As for (v), suppose that $|x_1- n | < \kappa$. Then for sufficiently large $n$ (and so also $x_1$) 
\[
|n^4 - x_1^4 |  = |n^2+x_1^2 |\, |n+x_1 | \, |n-x_1 |  \leq \widetilde{C}  |x_1 |^{3}
\]
for some $\widetilde{C} >0$ (depending on $\kappa $). Thus
\[
|n^4 F_1 (x_1,0) - D | \leq |n^4-x_1^4 | \, |F_1 (x_1,0)| + |x_1^4 F_1 (x_1,0) -D|\leq \widetilde{C} C  |x_1 |^{-1} + |x_1^4 F_1 (x_1,0) -D| .
\]
Since taking $n$ large makes $x_1$ large as well, we see from (i) that for sufficiently large $n$ $|n^4 F_1 (x_1 ,0 )-D| \leq 0.0005 D$, that is
\eqnb\label{temp_point(v)}
|F_1 ( x_1,0 ) - n^{-4} D | \leq 0.0005 n^{-4} D .
\eqne 
Moreover, the Mean Value Theorem gives for $|x_2|<1$ and sufficiently large $n$
\[
|F_1 (x_1,x_2) - F_1 (x_1,0)| \leq |x_2 | \, |\nabla F_1 (x_1, \xi )| \leq C |x_1 |^{-5}  \leq 0.0005 n^{-4} D,
\]
where $\xi \in (0,1)$. The claim follows from this and \eqref{temp_point(v)}.
\end{proof}

\section{The setting}\label{sec_the_setting}
In this section we define constants $T>0$, $\nu_0>0$, $\tau \in (0,1)$, $z\in \RR^3$, the set $G$ and the vector field $u$ which were required in the sketch proof in Section \ref{sec_proof_of_thm1}. The definition is based on a certain geometric setting which we formalise here in the notion of the \emph{geometric arrangement}. 
\begin{framed}
By the \emph{geometric arrangement} we mean a pair of open sets $U_1,U_2 \Subset P$ together with the corresponding structures $(v_1,f_1,\phi_1)$, $(v_2,f_2,\phi_2)$ (recall Definition \ref{def_of_structure}) such that $\overline{U_1}\cap \overline{U_2}=\emptyset$ and, for some $T>0$, $\tau \in (0,1)$, $z\in \RR^3 $,
\eqnb\label{fairies_extra_ineq}
f_2^2 +T v_2 \cdot F[v_1,f_1] > |v_2 |^2\quad \text{ in }U_2,
\eqne
\eqnb\label{fairies_scaling}
f_2^2(y) +T v_2(y) \cdot F[v_1,f_1](y) > \tau^{-2} \left(  f_1 (R^{-1}x) +  f_2(R^{-1}x) \right)^2 
\eqne
for all $x\in G\coloneqq R\left( \overline{U_1} \cup \overline{U_2} \right)$, where
\eqnb\label{def_of_y} 
y =  R^{-1} (\Gamma (x))
\eqne
(recall $\Gamma(x)=\tau x + z$).
\end{framed}

Note that \eqref{fairies_scaling} gives in particular that $\Gamma$ maps $G$ into itself: we obtain $R^{-1} (\Gamma (G))\subset G$ and so (taking the rotation $R$ of both sides)
\eqnb\label{gamma_of_G_is_in_G}
\Gamma (G) \subset RR^{-1} (G) \subset R(G)=G.
\eqne

Before defining the remaining constant $\nu_0$ and vector field $u$, we comment on the notion of the geometric arrangement in an informal way.

Recall from Section \ref{sec_proof_of_thm1} that we aim to find a vector field $u$, which is defined on the time interval $[0,T]$, that satisfies the NSI \eqref{NSI_u} as well as admits the gain in magnitude \eqref{u_grows}. 
We want to obtain the gain via the term $u\cdot \nabla p$, which we now discuss. We will construct $u$ in a way that, at time $t=0$
\[
u(0)\approx u[v_1,f_1]+u[v_2,f_2],
\]
and at time $t= T$
\eqnb\label{approx_u(T)}
u(T) \approx u[v_1,f_1]+u\left[ v_2,\sqrt{f_2^2+Tv_2\cdot F[v_1,f_1]} \right].
\eqne
In other words, $u$ is to consist of two disjointly supported (in space) vector fields. The first of them will be supported in $R(\overline{U_1})$ and its absolute value (that is $f_1$) will remain (approximately) constant through the time interval $[0,T]$. The second of them will be supported in $R(\overline{U_2})$ and its absolute value will change in time from $f_2$ to (approximately) $\sqrt{f_2^2   +Tv_2\cdot F[v_1,f_1]}$.

At this point it is clear that the requirement \eqref{fairies_extra_ineq} is necesseary for the right-hand side of \eqref{approx_u(T)} to be well-defined (recall \eqref{requirements_on_v_f}). Furthermore, in light of the property $|u[v,f]|=f$ (valid for any (admissible) $v,f$, recall \eqref{abs_value_of_u_is_f}) we see that the requirement \eqref{fairies_scaling} means simply that
\[
| u(\Gamma (x) ,T)|^2 \gtrsim \tau^{-2} | u (x,0)|^2.
\]
By writing ``approximately'' (or $\approx$, $\gtrsim$) we mean ``very close in the $L^\infty (\RR^3 )$ norm''. Such an approximate sense will be made rigorous below by using continuity arguments as well as the facts that the inequalities in \eqref{fairies_extra_ineq} and \eqref{fairies_scaling} are sharp (``$>$'') and the supports of the functions appearing on their right-hand sides are compact.\\

It remains to ask why the term ``$Tv_2\cdot F[v_1,f_1]$'' is chosen to achieve the gain in magnitude.\\

A rough answer to this question is: because (1) the pressure interaction function has a certain property that allows us to magnify it and because (2) that this one of the very few degrees of freedom allowed by the Navier--Stokes inequality. We have already observed (1)  in Lemma \ref{lem_properties_of_F} (particularly part (ii)), and we will see the full power of it in the construction of the geometric arrangement in Section \ref{sec_geom_arrangement}. As for (2), recall the NSI \eqref{NSI_u},
\[
\p_t |u |^2 \leq -u\cdot \nabla  \left( |u|^2 -2  p \right)  + 2 \nu \, u\cdot \Delta u .
\]
$\mbox{}$

We illustrate the reason for the term ``$Tv_2\cdot F[v_1,f_1]$'' by the following thought experiment. Suppose that 
\eqnb\label{suppose_u_looks_like_this}
 u=u[v_1,f_1]+u[v_2,f_2]
\eqne
and take a close look at the terms appearing on the right-hand side of the NSI above, where we ignore, for a moment, the time dependence. First of all, the pressure function $p$, is given by $p^*[v_1,f_1]+p^*[v_2,f_2]$ (recall Lemma \ref{lem_prop_of_p[v,f]} (iii)). Thus, since both $u$ and $p$ are axisymmetric, so are all the terms on the right-hand side of the NSI (recall \eqref{scalar_fcns_rotationally_invariant}). Thus it is sufficient to look only at points $x$ of the form $(x_1,x_2,0)$, $(x_1,x_2)\in \RR^2 $. At such points the right-hand side of the NSI takes the form
\eqnb\label{rhs_of_nsi}
-\lewy (v_{1}+v_{2})\cdot \nabla \prawy (f_1^2+f_2^2 + 2p[v_1,f_1]+2p[v_2,f_2]) + 2\nu u(\cdot ,0) \cdot \Delta u (\cdot , 0),
\eqne
where $v_1=(v_{11},v_{12})$, $v_2=(v_{21},v_{22})$ and $\nabla=(\p_1,\p_2)$ now denotes the two-dimensional gradient; recall also that $p[v_i,f_i]=p^*[v_i,f_i](\cdot , 0)$ (see \eqref{def_of_ps}). Observe that the $\p_3$ derivative does not appear since both $u$ and $p$ are axisymmetric (and so $\p_3$ is a derivative along a level set, recall \eqref{d3_of_|u|_vanishes} and \eqref{dx3_of_p_is_zero}). \\

The last term in \eqref{rhs_of_nsi} will not play any significant role in our analysis; we will treat it as an error term. In fact, we already know how to deal with this term at points $(x_1,x_2)$ such that $(\phi_1 + \phi_2) (x_1,x_2) <1$ (recall that $\phi_1$, $\phi_2$ play the role of a cutoff function in the structures $(v_1,f_1,\phi_1)$, $(v_2,f_2,\phi_2)$, respectively; see Definition \ref{def_of_structure}). Indeed, at such points $v_1=v_2=0$, and so \eqref{rhs_of_nsi} becomes
\[
2\nu ( f_1 \, Lf_1 + f_2 \, Lf_2 ) \geq 0 
\]
(the inequality being a consequence of \eqref{fLf_is_nonnegative}). This non-negativity will turn out sufficient for the NSI (see \eqref{calculation_case1} below for details), while at points $(x_1,x_2)$ such that $(\phi_1 + \phi_2) (x_1,x_2) =1$ we will use continuity arguments to take $\nu $ sufficiently small (see \eqref{how_small_is_nu0} and \eqref{calculation_case2} below for details).\\

As for the first term in \eqref{rhs_of_nsi}, we will be interested in interactions between $u[v_1,f_1]$ and $u[v_2,f_2]$ (this is the reason why the geometric arrangement consists of two sets $U_1$, $U_2$ and their corresponding structures) and so from the terms in \eqref{rhs_of_nsi} we are concerned with the mixed terms of the form
\[
\text{(something supported in }U_i\text{) (a function ``based on'' }U_j\text{ and its structure)}
\]
for $i,j=1,2$, $i\ne j$, namely with the terms
\eqnb\label{all_the_interaction_terms}
-v_{i}\cdot \nabla f_j^2\qquad \text{ and }\qquad -2v_{i}\cdot \nabla p[v_j,f_j],
\eqne
$i,j=1,2$, $i\ne j$. Note that the first of such terms vanishes since $v_i$ and $f_j$ have disjoint supports. As for the second one, we will be manipulating only the terms with the ``$\p_1$'' derivative since we are only able to control this derivative of the pressure function (which comes, fundamentally, from the property \eqref{property_of_p_magic} and from our choice of picking $Ox_1$ as the axis of symmetry; we have already explored this (to some extent) in Lemma \ref{lem_properties_of_F}). In fact, we aim to construct the geometric arrangement in such a way that 
\eqnb\label{the_pressure_terms_that_will_come_up}
-v_{21} \p_1 (p[v_1,f_1]-p[0,f_1] )=v_{21} F_1[v_1,f_1]\qquad \text{ is large } 
\eqne
in a certain region of $U_2$ that is close to the $Ox_1$ axis (see Section \ref{sec_simplified_geom_arrangement} for a wider discussion of this issue). In other words we will try to, in a sense, magnify the influence of $U_1$ (and its structure) onto $U_2$ (and its structure).\\ %The term $p[0,f_1]$ appears above since we are only able to control (by manipulating $v_1$ and $f_1$) the pressure function of this form (that is of the form $p[v_1,f_1]-p[0,f_1]$, recall Lemma \ref{lem_properties_of_F}).

We now discuss the issue of time dependence, which will explain the appearance of the term ``$p[0,f_1]$'' in \eqref{the_pressure_terms_that_will_come_up} above as well as the mechanism responsible to ``picking'' \eqref{the_pressure_terms_that_will_come_up} from all possible choices at the second term in \eqref{all_the_interaction_terms} (as $i,j$ varies). This will lead us to the term ``$Tv_2\cdot F[v_1,f_1]$'' in the geometric arrangement. In fact, instead of the naive candidate \eqref{suppose_u_looks_like_this}, we will actually consider a time dependent vector field of the form
\[
 u(t)=u[v_{1,t},f_{1,t}]+u[v_{2,t},f_{2,t}],
\]
where $v_{i,t}$, $f_{i,t}$ are certain time dependent extensions of $v_i$, $f_i$, respectively (see \eqref{def_of_u_sol_thm} and \eqref{def_of_q} below for the exact formula), which are chosen so that
\eqnb\label{time_dependent_|u(t)|}\begin{split}
|u(\cdot,0, t)|^2&=f_{1,t}^2+f_{2,t}^2 = f_1^2+f_2^2+ \left( \text{something small, negative and linear in }t\right) \\
& -\int_0^t  (v_{1,s}+v_{2,s}) \cdot \nabla  (f_{1,s}^2+f_{2,s}^2 + 2p[v_{1,s},f_{1,s}]+2p[v_{2,s},f_{2,s}]) \d s,
\end{split}
\eqne
(We write ``$(\cdot , 0,t)$'' to articulate that we restrict ourselves to points $(x,t)$ of the form $(x_1,x_2,0,t)$.)
Note that by taking $\p_t$ we obtain 
\[\begin{split}
\p_t |u(\cdot, 0, t)|^2=&-(\text{something small})\\
& - (v_{1,t}+v_{2,t}) \cdot \nabla  (f_{1,t}^2+f_{2,t}^2 + 2p[v_{1,t},f_{1,t}]+2p[v_{2,t},f_{2,t}]) ,\\
=&-(\text{something small}) - u(\cdot , 0, t)  \cdot \nabla  (|u(\cdot , 0 , t)|^2  + 2p(\cdot , 0 ,t)) .
\end{split}
\]
Here, the small term will be used in the continuity argument to absorb the Laplacian term, $\nu u\cdot \Delta u$ (compare with \eqref{rhs_of_nsi}), see \eqref{calculation_case1} and \eqref{calculation_case2} for details.
In other words, the time dependent extensions $v_{i,t}$, $f_{i,t}$ ($i=1,2$) will be chosen such that, by construction, we will obtain the NSI.

In particular, we will choose  
\[
v_{i,t} = a_i (t) v_i,\qquad i=1,2,
\]
where $a_1,a_2 \in C^\infty (\RR ; [-1,1])$ are certain \emph{oscillatory processes}, which are discussed in detail in Section \ref{sec_oscillatory_process} below. The oscillatory process will have two remarkable features.
The first is that 
\[
\int_0^t a_i (s) v_i \cdot \nabla f_{i,s} \d s \approx 0, \qquad \text{ uniformly in }i=1,2, t\in [0,T],
\]
and it will be a simple consequence of high oscillations of $a_1,a_2$. The second remarkable feature is that they enable us to pick from all the terms 
\[
\int_0^t a_i(s)v_{i} \cdot \nabla p[a_j (s) v_j , f_{j,s}] \d s,\qquad i,j\in \{ 1,2\}
\]
any of the terms
\[
\int_0^t v_i \cdot \nabla p[v_j , f_{j,s} ],
\]
provided we subtract $p[0,f_{j,s}]$. To be more precise for any choice of indices $i_0,j_0\in \{ 1,2\} $ there exist oscillatory processes $a_1,a_2\in C^\infty (\RR ; [-1,1])$ such that 
\[
\sum_{i,j=1}^2\int_0^t a_i(s)v_{i} \cdot \nabla p[a_j (s) v_j , f_{j,s}] \d s \approx \int_0^t v_{i_0} \cdot\nabla  ( p[v_{j_0} , f_{j_0,s} ]- p[0 , f_{j_0,s} ] )\d s\quad \text{ for all }t\in [0,T].
\]
Therefore, choosing $(i_0,j_0)\coloneqq (2,1)$ (since we are interested in the influence of $U_1$ onto $U_2$)
we obtain that the integral on the right-hand side of \eqref{time_dependent_|u(t)|} is approximately
\[
\int_0^t v_{2} \cdot\nabla  ( p[v_{1} , f_{1,s} ]- p[0 , f_{1,s} ] )\d s = -\int_0^t v_{2} \cdot  F [v_1,f_{1,s}] \d s \quad \text{ for all }t\in [0,T]
\]
On the other hand, we will make a choice of $f_{1,s}$ that is, roughly speaking, very slowly depending on $s$, so that the last integral is approximately
\[
t\, v_2 \cdot  F[v_1,f_1].
\]
In other words, we will choose the oscillatory processes $a_1, a_2$ and the time-dependent extensions of $f_1$, $f_2$ such that, except for the expression of $|u(t)|$ given (approximately) by \eqref{time_dependent_|u(t)|}, we will obtain, at the same time, another one:
\eqnb\label{another_repr_of_|u|}
|u(\cdot , 0 ,t)|^2 \approx f_1^2+f_2^2 +t\, v_2\cdot   F[v_1,f_1] \quad t\in [0,T].
\eqne
This explains (by taking $t=T$) the appearance of the term $T v_2\cdot F[v_1,f_1]$ in the geometric arrangement. \\

To sum up the above heuristic discussion, based on any disjoint sets $U_1$, $U_2$ and their corresponding structures $(v_1,f_1,\phi_1)$, $(v_2,f_2,\phi_2)$ we can find a way of prescribing the time dependence (on any time interval) such that the NSI is satisfied (by prescribing behaviour in time, in particular by the oscillatory processes) and that $|u(t)|$ is approximately as in \eqref{another_repr_of_|u|}, which  in turn we are able to magnify (at least in some region of the support) by arranging $U_1$, $U_2$ (and the corresponding structures) and defining $T>0$ appropriately; namely by constructing the geometric arrangement. 

The construction of the geometric arrangement (which is sketched in Fig.\ \ref{the_arrangement_sketchy_sketch}) is a nontrivial matter and it is in fact the heart of the proof of Theorem \ref{point_blowup_thm}. We present it in Section \ref{sec_geom_arrangement}.
\begin{figure}[h]
\centering
 \includegraphics[width=\textwidth]{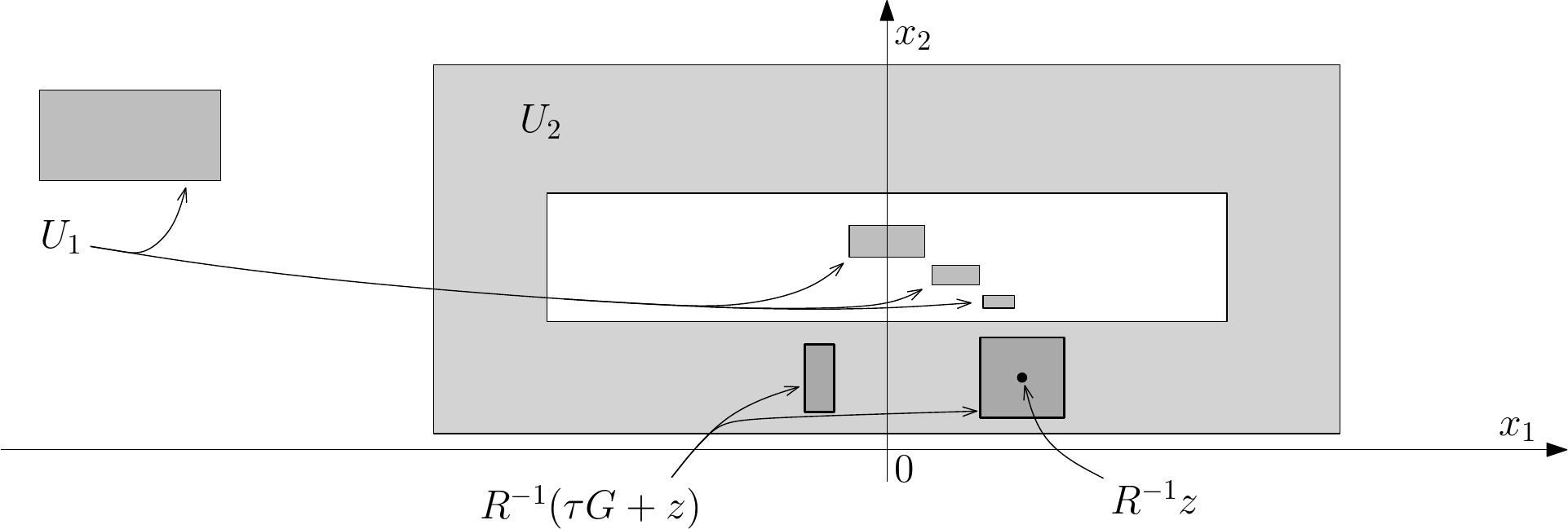}
 \nopagebreak
 \captionsetup{width=0.9\textwidth} 
  \captionof{figure}{A sketch of the geometric arrangement, see Section \ref{sec_geom_arrangement} for details. Note that the inclusion $R^{-1}(\tau G+z)\subset U_2 \subset R^{-1} (G)$, which is illustrated on this sketch, implies that $\Gamma (G)\subset G$ (recall \eqref{gamma_of_G_is_in_G}). Proportions are not conserved on this sketch.}\label{the_arrangement_sketchy_sketch} 
\end{figure}

In the remainder of this section, we assume that the geometric arrangement is given and we apply the strategy outlined in the heuristic discussion above, but in a rigorous way. Namely we obtain $\nu_0$ and $u$ (the remaining constants $T>0$, $\tau \in (0,1)$, $z\in \RR^3$ and the compact set $G\subset \RR^3$, which were required in the sketch argument in Section \ref{sec_sketch_of_pf_of_thm2}, are given by the geometric arrangement). 

We note that, except for the need of rigorous presentation (in the remainder of this section as well as in Section \ref{sec_geom_arrangement}, where we construct the geometric arrangement), it is also rather pleasing to observe all components of the construction fit together.

Furthermore, we will not be using the notation $f_{i,t}$ (to denote the time extension of $f_i$, $i=1,2$), but rather $h_{i,t}$ (the time extension of $f_i$) and $q_{i,t}^k$ (an approximation of $h_{i,t}$, where $k$ is large).\vspace{1cm}\\

Let $\theta >0$ be sufficiently small such that
\eqnb\label{fairies_scaling1}
f_2^2(y) +T v_2 (y) \cdot F[v_1,f_1](y) > \tau^{-2} \left(  f_1 (R^{-1}x) +  f_2(R^{-1}x) \right)^2  +2\theta 
\eqne
for $x\in G$. (Recall $y=R^{-1}(\Gamma x))$.) Such a choice is possible by continuity since the inequality in \eqref{fairies_scaling} is strict and $G$ is compact.

Let $h \colon P \times [0,T] \to [0,\infty )$ be defined by 
\eqnb\label{def_of_ht}
h_t = h_{1,t}+ h_{2,t}
\eqne
(recall we use the convention $h_t(\cdot ) \equiv h(\cdot , t)$), 
where
\begin{eqnarray}
&& \hspace{-1cm}h_{1,t}^2 \coloneqq f_1 ^2 - 2t \delta \phi_1  , \label{def_h1} \\
&& \hspace{-1cm}h_{2,t}^2 \coloneqq f_2^2 -2t\delta \phi_2 + \int_0^t v_2  \cdot  F[v_1,h_{1,s}] \,\d s . \label{def_h2}
\end{eqnarray}
Thus $h_{i,t}$ is a time dependent modification of $f_i$, $i=1,2$, such that $h_{i,t}=f_i$ outside $\supp \, \phi_i$ (recall $\supp \, v_2 \subset \supp \, \phi_2$, see Definition \ref{def_of_structure}). Here $\delta >0$ is a fixed, small number given by the following lemma.
\begin{lemma}[properties of functions $h_{1,t}$, $h_{2,t}$]\label{prop_of_h1h2}
There exists $\delta >0$ (sufficiently small) such that $h_1,h_2 \in C^\infty ( P\times (-\delta , T+\delta ); [0,\infty ))$,
\eqnb\label{h_i_give_structure}
(v_i,h_{i,t},\phi_i ) \text{ is a structure on } U_i \qquad \text{ for } t\in (-\delta , T+\delta ), i=1,2,
\eqne 
and
\eqnb\label{3.14}
h_{2,T}^2 (y)    > \tau^{-2}  \left(  f_1 (R^{-1}x) +  f_2(R^{-1}x) \right)^2  + \theta  \quad \text{ for }x\in R\left( \overline{U_1} \cup \overline{U_2} \right).
\eqne
\end{lemma}
(Recall $y=R^{-1}(\Gamma (x))$.)
\begin{proof} For $h_{1,t}$ note that since $f_1>0 $ in $U_1$ we can take $\delta \in (0,1)$ such that
\[\delta < \min_{\supp \, \phi_1} \left| f_1^2 - |v_1|^2 \right| / 2(T+2)
\]
to obtain 
\[
h_{1,t} > |v_1| \qquad \text{ in } \supp \, \phi_1 \text{ for }t\in [-1,T+1].
\] 
Thus, since $h_{1,t}=f_1$ outside $\supp \, \phi_1$,
\[
h_{1,t} > |v_1| \geq 0 \qquad \text{ in } U_1\text{ for } t\in (-\delta,T+\delta ).
\] 
Hence, since both $f_1$ and $\phi_1$ are smooth on $P$ we immediately obtain the required smoothness of $h_1$ and that $(v_1,h_{1,t}, \phi_1)$ is a structure on $U_1$ for all $t\in (-\delta , T+\delta )$.

As for $h_{2,t}$, suppose for the moment that $\delta =0$. Then $h_{1,t} = f_1$ and so
\eqnb\label{what_is_h2_if_delta_is_0}
h_{2,t}^2 = f_2^2+t \, v_2  \cdot  F[v_1,f_1].
\eqne
This means that
\eqnb\label{h2_when_delta_is_0} h_{2,0}^2 = f_2^2 \quad \text{ and }\quad  h_{2,T}^2 = f_2^2 + T v_2 \cdot  F[v_1,f_1]\quad \text{ if } \delta =0.
\eqne
Using the fact that $(v_2,f_2,\phi_2)$ is a structure on $U_2$ and \eqref{fairies_extra_ineq}, we see that both of the above functions are greater than $|v_2|^2$ in $U_2$. In particular they are greater than $|v_2|^2$ on the compact set $\supp\, \phi_2$. Since $h_{2,t}$ in \eqref{what_is_h2_if_delta_is_0} depends linearly on $t$ we thus obtain 
\[
h_{2,t} > |v_2| \quad \text{ in }  \supp \, \phi_2 , \, \text{ for } t\in [0,T] \quad \text{ if } \delta =0.
\]
Therefore, since $h_2$ depends continuously on $\delta $, we obtain
\[
h_{2,t} > |v_2| \quad \text{ in }  \supp\, \phi_2, \, \text{ for } t\in [0,T] \quad \text{ if } \delta >0 \text{ is sufficiently small.}
\]
Thus, by continuity in time, this property holds also for $t$ belonging to an open interval containing $[0,T]$. Taking $\delta $ smaller we can take this open interval to be $(-\delta , T+\delta )$. Thus, recalling that $h_{2,t}=f_2$ outside $\supp \, \phi_2$ we obtain
\[
h_{2,t} > |v_2|\geq 0 \quad \text{ in }  U_2, \, \text{ for } t\in (-\delta ,T+\delta ) \quad \text{ if } \delta >0 \text{ is sufficiently small.}
\]
As in the case of $h_{1,t}$ this immediately gives the required regularity of $h_2$ and that $(v_2,h_{2,t},\phi_2)$ is a structure on $U_2$.

As for \eqref{3.14} note that \eqref{h2_when_delta_is_0} gives in particular
\[
h_{2,T}^2 = f_2^2 + T v_2 \cdot  F[v_1,f_1] \quad \text{ in } \supp \, \phi_2 \qquad \text{ if } \delta =0,
\]
and so for sufficiently small $\delta >0$
\[
 h_{2,T}^2 \geq  { f_2^2 + T v_2 \cdot  F[v_1,f_1] } - \theta \qquad \text{ in } \supp \, \phi_2.
\]
Since $h_{2,T}=f_2$ and $v_2=0$ outside $\supp \, \phi_2$, we trivially obtain the above inequality outside $\supp \, \phi_2$, and so \eqref{3.14} follows from this and \eqref{fairies_scaling1}.\end{proof}

We have now fixed $\delta >0$. Note that \eqref{h_i_give_structure} gives in particular that
\eqnb\label{av_and_h_give_structure}
(av_i,h_{i,t},\phi_i ) \text{ is a structure on } U_i \qquad \text{ for } t\in (-\delta , T+\delta ), i=1,2,
\eqne
for any $a\in [-1,1]$ ($t\in (-\delta , T+\delta )$, $i=1,2$).

At this point we fix $\nu_0 \in (0,1)0$ sufficiently small such that
\eqnb\label{how_small_is_nu0}
\nu_0 \left| u[a v_i , h_{i,t} ] \cdot \Delta u[a v_i, h_{i,t} ] \right| <\delta /8 \quad \text{ in } \RR^3 
\eqne
for $a\in [-1,1]$, $i=1,2$, $t\in [0,T]$.

Having constructed the time dependent functions $h_{1,t}, h_{2,t}$ and having fixed $\nu_0$, we now construct $u$. 
\begin{proposition}\label{prop_existence_of_u}
There exist $\eta \in (0,\delta )$ and $u\in C^\infty (\RR^3 \times (-\eta , T+\eta ); \RR^3 )$ such that
\begin{enumerate}
\item[\textnormal{(i)}] $\supp \, u(t) = G$ and $\mathrm{div} \, u(t) =0$ for $t\in (-\eta , T+\eta )$,
\item[\textnormal{(ii)}]$|u(x,0)|=h_0 (R^{-1} x)$ and $\left| \,|u(x,t)|^2 - h_t (R^{-1} x)^2 \right| < \theta$ for all $x\in \RR^3$, $t\in [0,T]$,
\item[\textnormal{(iii)}] the Navier--Stokes inequality
\[
\p_t |u|^2 \leq -u\cdot \nabla \left( |u|^2 +2p \right) + 2\nu\, u \cdot \Delta u
\]
holds in $\RR^3\times [0,T]$ for all $\nu \in [0,\nu_0]$ where $p$ is the pressure function corresponding to $u$.
\end{enumerate}
\end{proposition}
Note that $u$ given by the proposition satisfies all the properties required in Section \ref{sec_proof_of_thm1}. Among those only \eqref{u_grows} is nontrivial; this follows from (ii) and \eqref{3.14} by writing
\[\begin{split}
\left| u(\Gamma (x), T) \right|^2 &\geq h_T (R^{-1} (\Gamma (x))^2 - \theta \\
&\geq h_{2,T} (R^{-1} (\Gamma (x))^2 - \theta \\
&> \tau^{-2} \left(  f_1 (R^{-1}x) +  f_2(R^{-1}x) \right)^2 \\
&= \tau^{-2} h_0 (R^{-1} (x))^2 \\
&=  \tau^{-2} \left| u(x,0) \right|^2
\end{split}
\]
for $x\in G$ (the case $x\not \in G$ is trivial). The rest of the properties follow directly from (i), (iii). It remains to prove Proposition \ref{prop_existence_of_u}. The proof is separated into three steps, which we present in Sections \ref{sec_strategy_for_u}-\ref{sec_oscillatory_process} below. 

\subsection{The construction of $u$}\label{sec_strategy_for_u}
We will find $u$ (a solution to Proposition \ref{prop_existence_of_u}) that is axisymmetric (see \eqref{rot_invariant_velocity}). For such a vector field (ii) is equivalent to
\eqnb\label{equiv_to_ii}
\left| u (x,0,0) \right| = h_0 (x), \quad \text{ and } \quad \left| \,|u(x,0,t)|^2 - h_t (x)^2 \right| < \theta \quad \text{ for } x\in {P}, t\in[0,T]
\eqne
and (iii) is equivalent to 
\eqnb\label{equiv_to_iii}
\p_t |u(x,0,t)|^2 \leq -u(x,0,t) \cdot \nabla \left( |u(x,0,t)|^2 +2p(x,0,t) \right) + 2\nu \,u(x,0,t) \cdot \Delta u(x,0,t)
\eqne
being satisfied for all $x\in {P}$, $t\in [0,T]$, $\nu \in [0,\nu_0 ]$.

We will consider functions $q_1^k$, $q_2^k$ defined by
\eqnb\label{def_of_q}
\left( q_{i,t}^k \right)^2 \coloneqq f_i^2-2t\delta \phi_i - \int_0^t a_i^k (s) v_i\cdot \left( \nabla h_{i,s}^2+2\nabla p[a_1^k(s) v_1,h_{1,s}]+2\nabla p[a_2^k(s) v_2,h_{2,s}] \right)\d s ,
\eqne
$i=1,2$, $k\in \NN$, for some functions $a_1^k, a_2^k \in C^\infty (\RR ; [-1,1])$ (which we shall call \emph{oscillatory processes} and which we discuss below). Recall that we use the convention $q_{i,t}^k (\cdot ) \equiv q_i^k (\cdot , t)$ (see \eqref{the_convention_with_t}). We will show that, given a particular choice of the oscillatory processes $a_1^k, a_2^k $, the vector field 
\eqnb\label{def_of_u_sol_thm}
u(x,t) \coloneqq u[a_1^k (t) v_1 , q_{1,t}^k](x)+ u[a_2^k (t) v_2 , q_{2,t}^k](x),
\eqne
is a solution to Proposition \ref{prop_existence_of_u} for sufficiently large $k$. Note that such $u$ is axisymmetric (recall Section \ref{u[v,f]_section}). Before proceeding to the proof, we comment on this strategy in an informal way.\\

Forget, for the moment, about the functions $q_1^k$, $q_2^k$, and let us try to attack Proposition \ref{prop_existence_of_u} directly. We observe that part (ii) and the facts that $h_{1,t}$, $h_{2,t}$ have disjoint supports $\overline{U_1}$, $\overline{U_2}$ (respectively) and that $(v_1, h_{1,t},\phi_1 )$, $(v_2, h_{2,t},\phi_2 )$ are structures on $U_1$, $U_2$ (respectively) suggest looking at the velocity field of the form 
\eqnb\label{def_of_u_tilde}
\widetilde{u}(x,t) \coloneqq u [v_1, h_{1,t} ](x) + u[v_2, h_{2,t} ] (x).
\eqne 
In other words we have
\[
\left| \widetilde{u}(x,0,t) \right|^2 = h_{1,t}^2 (x) + h_{2,t}^2 (x) \qquad x\in P,t\in [0,T],
\]
so that claim (ii) is satisfied in an {exact sense} (rather than in an approximate sense with error $\theta$). This might look promising, but, recalling the definition of $h_1$, $h_2$ (see \eqref{def_h1}, \eqref{def_h2}) we see that
\[
\p_t \left| \widetilde{u}(x,0,t) \right|^2 = - 2\delta (\phi_1+\phi_2)(x) + v_2(x) \cdot F[v_1,h_{1,t}](x),
\]
and at this point is is not clear how to relate the right-hand side to the terms
\[
-u \cdot \nabla \left( |u|^2 +2p \right) + \nu u \cdot \Delta u,
\]
which are required by the Navier--Stokes inequality \eqref{equiv_to_iii} (that is by (iii)). Thus the velocity field $\widetilde{u}$ seems unlikely to be a solution of Proposition \ref{prop_existence_of_u}. In order to proceed one needs to make use of two degrees of freedom available in the construction of $\widetilde{u}$. The first of them is the fact that claim (ii) of Proposition \ref{prop_existence_of_u} only requires $|u(x,t)|$ to ``{keep close}'' to $h_{t}(R^{-1} x)$ as $t$ varies between $0$ and $T$ (rather than to be equal to it), which we have already pointed out above. The second one is that $|\widetilde{u}(x,0,t)|$ is expressed only in terms of $h_{1,t}$, $h_{2,t}$. Thus a velocity field of the form
\[
 \overline{u} (x,t) \coloneqq u [ a_1(t) v_2, h_{1,t} ] (x) + u [a_2(t) v_2,h_{2,t}](x) 
\]
has the same absolute value $\left| \overline{u}  \right|$ as $\left|  \widetilde{{u}}  \right|$ for any choice of $a_1,a_2 \colon \RR \to [-1,1]$. Recall also that since $|a_1|,|a_2|\leq 1$,
\[
\left( a_i(t) v_i, h_{i,t} , \phi_i \right) \quad \text{ is a structure on } U_i, i=1,2,
\]
(recall \eqref{av_and_h_give_structure}) and so $\overline{u}$ is well-defined. By introducing the functions $q_{1}^k$, $q_2^k$ (in \eqref{def_of_q}) we make use of these two degrees of freedom. \\

We now proceed to a discussion of some elementary properties of these functions, and we show in Section \ref{sec_pf_of_claims_of_thm} that considering them is a good idea; namely that \eqref{def_of_u_sol_thm} is a solution of the proposition for sufficiently large $k$.

First note that, as in the case of $h_{i,t}$, $q_{i,t}^k$ differs from $f_i$ only on the compact set $\supp \, \phi_i$, $i=1,2$. Secondly, 
\eqnb\label{deriv_of_q}
\p_t \left( q_{i,t}^k \right)^2 = -2\delta \phi_i - a_i^k (t) v_i \cdot \left(\nabla h_{i,t}^2 + 2 \nabla p [a_1^k (t) v_1 , h_{1,t}]+ 2 \nabla p [a_2^k (t) v_2 , h_{2,t}] \right).
\eqne
(Compare the terms appearing on the right-hand side to \eqref{all_the_interaction_terms}.)

Finally, we will show in Section \ref{sec_oscillatory_process} that, given a particular choice of the oscillatory processes $a_1^k,a_2^k \in C^\infty (\RR , [-1,1])$ (which are a part of the definition of $q_1^k$, $q_2^k$, recall \eqref{def_of_q}),
\eqnb\label{conv_of_qs}
\begin{cases}
q_{i,t}^k \to h_{i,t}  \\
\text{ and }\\
D^l q_{i,t}^k \to D^l h_{i,t}\quad 
\end{cases} \qquad \text{ uniformly in } P \times [0,T], i=1,2, \text{ for each }l\geq 1
\eqne
as $k\to \infty$. Recalling properties of $h_{1,t}$, $h_{2,t}$ (see Lemma \ref{prop_of_h1h2}), we see that this convergence gives in particular that for sufficiently large $k$
\[
q_{i,t}^k > |v_i| \qquad \text{ in } \supp \, \phi_i \qquad \text{ for } t\in [0,T], i=1,2, 
\]
and so by continuity (as in the proof of Lemma \ref{prop_of_h1h2})
\eqnb\label{q_larger_than_v}
q_{i,t}^k >|v_i| \geq 0 \qquad \text{ in } U_i \qquad\text{ for } t\in (-\delta_k , T+\delta_k),
\eqne
for some $\delta_k\in (0,\delta)$. Thus for sufficiently large $k$
\[
(v_i,q_{i,t}^k, \phi_i) \text{ is a structure on } U_i\qquad \text{ for } t\in (-\delta_k,T+\delta_k), i=1,2, 
\]
and thus, since $|a_1^k|, |a_2^k|\in [-1,1]$, 
\eqnb\label{qs_and_as_give_structure}
(a_i^k(t)v_i,q_{i,t}^k, \phi_i) \text{ is a structure on } U_i\qquad \text{ for } t\in (-\delta_k,T+\delta_k), i=1,2. 
\eqne
Moreover, \eqref{q_larger_than_v} and the fact that all terms on the right-hand side of \eqref{def_of_q} are smooth (recall \eqref{pressure_p[v,f]_is_smooth} for the smoothness of the pressure) give 
\[
q_i^k \in C^\infty (P\times (-\delta_k ,T+\delta_k);[0,\infty )) ,\qquad i=1,2.
\]
\subsection{The proof of the claims of the proposition}\label{sec_pf_of_claims_of_thm}
Using the above properties of the functions $q_1^k$, $q_2^k$, we now show that for $k$ sufficiently large the vector field $u$ given by \eqref{def_of_u_sol_thm},\[
u(x,t) \coloneqq u[a_1^k (t) v_1 , q_{1,t}^k](x)+ u[a_2^k (t) v_2 , q_{2,t}^k](x),\]
with $\eta \coloneqq \delta_k$ satisfies the claims of Proposition \ref{prop_existence_of_u}. 

Claim (i) and the smoothness of $u$ on $\RR^3 \times (-\eta,T+\eta)$ follow directly from \eqref{qs_and_as_give_structure}, the smoothness of the oscillatory processes $a_1^k,a_2^k$ on $\RR$ (which we are about to construct in the next section) and from the smoothness of $q_1^k, q_2^k$ stated above.

Claim (ii) is equivalent to \eqref{equiv_to_ii} (due to axisymmetry of $u$), and thus its first part follows by writing
\[
|u(x,0,0)|=q_{1,0}^k(x)+q_{2,0}^k(x)=f_1(x)+f_2(x)=h_0 (x).
\]
The second part follows directly from the convergence \eqref{conv_of_qs} by taking $k$ sufficiently large such that
\[
\left| (q_{1,t}^k+q_{2,t}^k)^2 - h_t^2 \right| < \theta \qquad \text{ in } P, \quad \text{ for }t\in [0,T].
\]
For such $k$ we obtain
\[
\left| \,|u(x,0,t)|^2 - h_t (x)^2 \right| =\left| (q_{1,t}^k(x)+q_{2,t}^k(x))^2 - h_t(x)^2 \right| < \theta,
\]
as required.

As for claim (iii), first recall that $p(t)$, the pressure function corresponding to $u(t)$, is (due to \eqref{def_of_p*}) given by
\[
p(t) = p^* [a_1^k(t) v_1,q_{1,t}^k ]+p^* [a_2^k(t) v_2,q_{2,t}^k ],
\]
and so in particular
\[
p(x,0,t) = p [a_1^k(t) v_1,q_{1,t}^k ](x)+p [a_2^k(t) v_2,q_{2,t}^k ](x).
\]
Recalling that claim (iii) is equivalent to \eqref{equiv_to_iii}, that is the Navier--Stokes inequality restricted to $P$,
 \[
\p_t |u(x,0,t)|^2 \leq -u(x,0,t) \cdot \nabla \left( |u(x,0,t)|^2 +2p(x,0,t) \right) + 2\nu \, u(x,0,t) \cdot \Delta u(x,0,t),
\]
where $\nu \in [0,\nu_0]$ (recall \eqref{how_small_is_nu0} for the choice of $\nu_0$), we fix $x\in P$, $t\in [0,T]$ and we consider two cases.\vspace{0.4cm}\\
\emph{Case 1.} $\phi_1 (x) +\phi_2 (x) <1$. 

For such $x$ we have $v_1(x)=v_2(x)=0$ (from the elementary properties of structures, recall Definition \ref{def_of_structure}) and the Navier--Stokes inequality follows trivially for all $k$ by writing 
\eqnb\label{calculation_case1}
\begin{split}
\p_t |u(x,0,t)|^2 &= \p_t q_{1,t}^k (x)^2 + \p_t q_{2,t}^k  (x)^2 \\
&= -2\delta (\phi_1 (x) + \phi_2 (x) )\\
&\leq 0\\
&\leq - u(x,0,t)\cdot \nabla \left( |u(x,0,t)|^2 +2 p (x,0,t) \right) + 2 \nu\, u(x,0,t) \cdot \Delta u(x,0,t) ,
\end{split}
\eqne
where we used \eqref{prop_of_structure_1} and \eqref{prop_of_structure_2} in the last step.\vspace{0.4cm}\\
\emph{Case 2.} $\phi_1(x)+\phi_2 (x) =1$. 

In this case we need to use the convergence \eqref{conv_of_qs} to take $k$ sufficiently large such that
\eqnb\label{conv_1}
|v_i | \left( \left| \nabla (q_{i,t}^k)^2 - \nabla h_{i,t}^2 \right| +2 \sum_{j=1,2}\left| \nabla p[a_j^k (t) v_j , q_{j,t}^k ]-\nabla p [a_j^k (t) v_j,h_{j,t} ] \right|  \right)\leq \delta /2
\eqne
in $P$ (see Lemma \ref{lem_continuity_of_the_pressure_fcns} for a verification that \eqref{conv_of_qs} is sufficient for the convergence of the pressure functions) and
\eqnb\label{conv_2}
\begin{split}
&\nu_0 \left| u[a_i^k(t) v_i , q_{i,t}^k ] \cdot \Delta u [a_i^k(t) v_i , q_{i,t}^k ] \right| \\
&\leq \nu_0 \left| u[a_i^k(t) v_i , h_{i,t} ] \cdot \Delta u [a_i^k(t) v_i , h_{i,t} ] \right| + \delta/8 \leq \delta/4 
\end{split}
\eqne
in $\RR^3$, for $t\in [0,T]$, $i=1,2$ (see Lemma \ref{lem_continuity_of_u_dot_Delta_u} for a verification that \eqref{conv_of_qs} is sufficient for the first inequality; the last inequality follows from the definition of $\nu_0$, see \eqref{how_small_is_nu0}). \\
Recall that $\delta >0$ was fixed in Lemma \ref{prop_of_h1h2} (i.e. when we were defining $h_1$, $h_2$) and it also appears in the definition of $q_1^k$, $q_2^k$ (recall \eqref{def_of_q}).

 Using \eqref{deriv_of_q} we obtain 
\eqnb\label{calculation_case2}
\begin{split}
\p_t |u(x,0,t)|^2 &= \p_t q_{1,t}^k (x)^2 + \p_t q_{2,t}^k  (x)^2   \\
&= -2\delta - \left( a_1^k (t) v_1(x) +a_2^k (t)v_2 (x) \right) \cdot \nabla \left( h_{1,t} (x)^2 + h_{2,t} (x)^2  \right. \\
&\hspace{4cm}+ \left. 2 p[a_1^k(t) v_1,h_{1,t}](x) +2 p[a_2^k(t) v_2,h_{2,t}](x) \right) \\
&\leq -\delta - \left( a_1^k (t) v_1(x) +a_2^k (t)v_2 (x) \right) \cdot \nabla \left( q_{1,t}^k (x)^2 + q_{2,t}^k (x)^2  \right. \\
&\hspace{4cm}+ \left. 2 p[a_1^k(t) v_1,q_{1,t}^k](x) +2 p[a_2^k(t) v_2,q_{2,t}^k](x) \right) \\
&= -\delta - u_1(x,0,t) \p_{x_1} \left( |u(x,0,t)|^2 + 2p(x,0,t) \right)\\
&\hspace{1cm}- u_2(x,0,t) \p_{x_2} \left( |u(x,0,t)|^2 + 2p(x,0,t) \right),
\end{split}
\eqne
and so, recalling that $\p_{x_3} |u(x,0,t)|^2 = \p_{x_3} p(x,0,t) =0$ (as a property of axisymmetric functions, see \eqref{d3_of_|u|_vanishes} and \eqref{dx3_of_p_is_zero}),
\[\begin{split}
\p_t |u(x,0,t)|^2  &\leq - \delta - u (x,0,t) \cdot \nabla \left( |u(x,0,t)|^2 +2 p (x,0,t) \right) \\
&\leq 2\nu \, u(x,0,t) \cdot \Delta u (x,0,t)- u (x,0,t) \cdot \nabla \left( |u(x,0,t)|^2 +2 p (x,0,t) \right)
\end{split}
\]
for all $\nu \in [0,\nu_0]$, where we used \eqref{conv_2} in the last step.\\

Thus we have shown that for sufficiently large $k$ the Navier--Stokes inequality \eqref{equiv_to_iii} holds for all $x\in P$, $t\in [0,T]$ and $\nu \in [0,\nu_0]$, which gives (iii), as required.

\subsection{The oscillatory processes}\label{sec_oscillatory_process}

Here we construct the oscillatory processes $a_1^k,a_2^k \in C^\infty (\RR , [-1,1])$, $k\geq 1$, such that the functions $q_1^k$, $q_2^k$ (given by \eqref{def_of_q}) converge to $h_1$, $h_2$ (respectively) as in \eqref{conv_of_qs}. As outlined in Section \ref{sec_strategy_for_u} this completes the proof of Theorem \ref{point_blowup_thm} given the \emph{geometric arrangement} (which we construct in Section \ref{sec_geom_arrangement}).

As for the strategy for choosing $a_1^k$, $a_2^k$ we will divide $[0,T]$ into $4k$ subintervals and on each those subintervals we will let each of $a_1^k$, $a_2^k$ equal $1$, $-1$ or $0$ (except for a set of times of measure less than $1/k$) in a particular configuration. The configuration is such that the resulting $q_{1,t}^k$, $q_{2,t}^k$ oscillate near $h_{1,t}$, $h_{2,t}$ as $t$ varies between $0$ and $T$, and such that the oscillations grow in frequency (that is the number of subintervals increases with $k$) and decrease in magnitude (that is we obtain convergence \eqref{conv_of_qs}), see Fig. \ref{strategy_for_aik} for a sketch. 
\begin{figure}[h]
\centering
 \includegraphics[width=\textwidth]{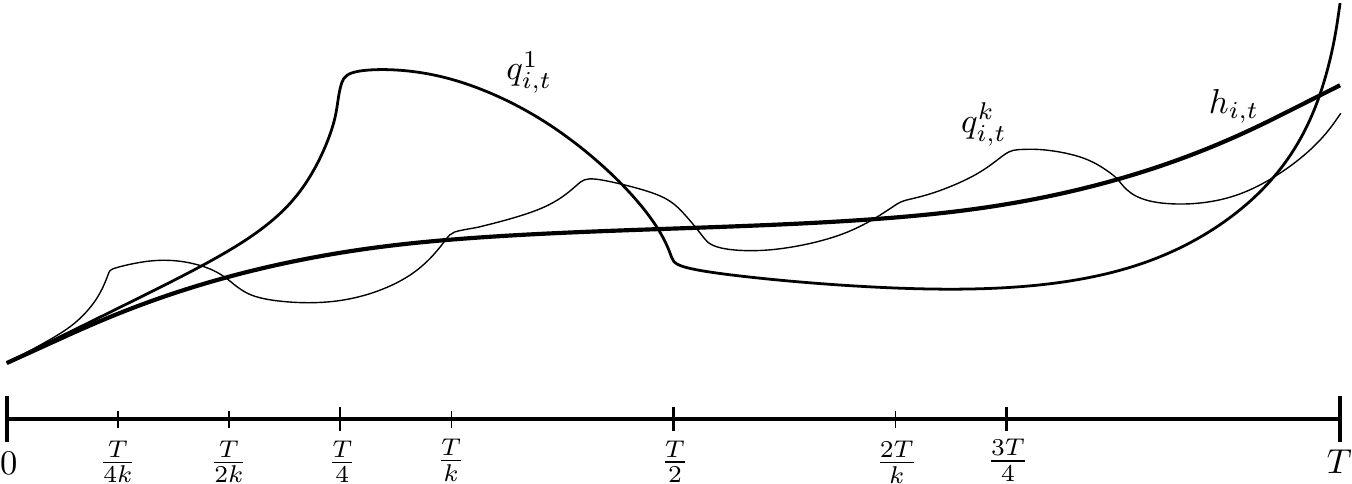}
 \nopagebreak
 \captionsetup{width=0.9\textwidth} 
  \captionof{figure}{The strategy for the choice of $a_{i}^k$, $i=1,2$. This sketch illustrates how the choice of $a_i^k$'s causes $q_{i,t}^k$'s to ``oscillate around $h_{i,t}$'' as $t$ varies between $0$ to $T$. Here $k=3$.}\label{strategy_for_aik} 
\end{figure}

We employ this strategy in the proof of the theorem below.
\begin{theorem}[existence of the oscillatory processes]\label{thm_existence_of_osc_proc}
For each $k\geq 1$ there exist a pair of functions $a_i^k \in C^\infty (\RR ; [-1,1])$, $i=1,2$, such that
\eqnb\label{osc_process_conv}\begin{split}
&\int_0^t a_i^k (s) \left( G_i (x,s) + F_{i,1} \left( x,s,a_1^k(s) \right) + F_{i,2} \left( x,s,a_2^k (s) \right)  \right) \d s\\
&\hspace{2cm}\stackrel{k\to \infty}{\longrightarrow } \begin{cases} \frac{1}{2}\int_0^t  \left( F_{2,1} (x,s,1)-F_{2,1} (x,s,0) \right) \d s \quad &i=2, \\
0 &i=1
\end{cases}
\end{split}
\eqne
uniformly in $(x,t) \in P\times [0,T ]$ for any bounded and uniformly continuous functions 
\[
G_i \colon P\times [0,T] \to \RR , \qquad F_{i,l} \colon P\times [0,T] \times [-1,1] \to \RR,
\]
$i,l=1,2$, satisfying 
\[
F_{i,l} (x,t,-1) = F_{i,l} (x,t,1) \qquad \text{for }x\in P, t\in [0,T], i,l=1,2.
\]
\end{theorem}
Note that this theorem gives \eqref{conv_of_qs} simply by taking 
\[
\begin{split}
G_i(x,t) &\coloneqq v_i (x) \cdot \nabla h_i(x,t)^2,\\
F_{i,l} (x,t,a) &\coloneqq 2v_i (x) \cdot \nabla p[av_l,h_{l,t}](x)
\end{split}
\]
(recall $p[v,f]=p[-v,f]$ by Lemma \ref{lem_prop_of_p[v,f]} (i)), and so such $F_{i,l}$'s satisfy the requirement $F_{i,l} (x,t,-1) = F_{i,l} (x,t,1)$ above) and by taking 
\[
\begin{split}
G_i(x,t) &\coloneqq D^\alpha \left( v_i (x) \cdot \nabla h_i(x,t)^2 \right),\\
F_{i,l} (x,t,a) &\coloneqq D^\alpha \left( 2v_i (x) \cdot \nabla p[av_l,h_{l,t}](x) \right)
\end{split}
\]
for any given multiindex $\alpha=(\alpha_1,\alpha_2)$. 

Before proceeding to the proof of Theorem \ref{thm_existence_of_osc_proc} we pause for a moment to comment on the meaning of the theorem and the convergence \eqref{conv_of_qs} in an informal manner. Recall that \eqref{def_of_q} includes terms of the form
\[
2 \int_0^t  a_i^k (s) v_i \cdot \nabla p [a_l^k(s) v_l ,h_{l,s}] \d s ,\quad  i,l\in \{1,2\}.
\]
Note that each of such terms represent, in a sense, an influence of the set $U_l$ (together with the structure $(a_l^k(s) v_l ,h_{l,s},\phi_l)$) on the set $U_i$; namely it vanishes outside $U_i$ and it uses the nonlocal character of the pressure function $p[\cdot , \cdot ]$ (that is the fact that the pressure function $p [a_l^k(s) v_l ,h_{l,s}]$ does not vanish on $U_i$). Thus we see from \eqref{osc_process_conv} that the role of the oscillatory processes $a_1^k , a_2^k $ is to ``select'' only the influence of $U_1$ on $U_2$ as $k\to \infty$ (except for this the oscillatory behaviour of the processes makes the terms $\int_0^t a_i^k (s) v_i\cdot \nabla h_{i,s}^2$, $i=1,2$, vanish as $k\to \infty$). Note this is the desired behaviour since we want to show the convergence \eqref{conv_of_qs} and of the two functions $h_1$, $h_2$ only $h_2$ includes an influence from $U_1$ (recall \eqref{def_h1}, \eqref{def_h2}).
The construction of such oscillatory processes is clear from the following auxiliary considerations, in which we forget, for a moment, about the smoothness requirement.

Let $f\colon [-1,1]\to \RR$ be such that $f(-1)=f(1)$ and let functions $b_1,b_2 \colon [0,T]\to [-1,1]$ be such that
\eqnb\label{def_of_b1_b2}
b_1(t) = \begin{cases}
1 \quad &t\in (0,T/4),\\
-1 \quad &t\in (T/4,T/2),\\
0 \quad &t\in (T/2,T),
\end{cases}\qquad b_2 (t) = \begin{cases}
1 \quad &t\in (0,T/2),\\
-1 \quad &t\in (T/2,T).
\end{cases}
\eqne
Then
\eqnb\label{basic_processes_magic}
\int_0^T b_i (s) f\left( b_l(s)\right) \d s =\begin{cases}
\frac{T}{2} (f(1)-f(0))  &(i,l) = (2,1),\\
0 \quad &(i,l) \ne (2,1),
\end{cases} 
\eqne
that is the choice of $b_1,b_2$ is such that they ``{pick}'' the value ${T} (f(1)-f(0))/2$ only for the choice of indices $(i,l)=(2,1)$. Clearly, given $(i_0,l_0)\in \{1,2\}^2$ one could choose $b_1$, $b_2$ that pick this value only for the choice of indices $(i,l)=(i_0,l_0)$.

More generally, let $f$ be also a function of time, $f\colon [0,T] \times [-1,1] \to \RR $ with $f(t,-1)=f(t,1)$ for all $t$ such that $f$ is almost constant with respect to the first variable, i.e. for some $\epsilon >0$
\[
\sup_{t\in [0,T]}  f(t,a) - \inf_{t\in [0,T]} f(t,a) < \epsilon, \qquad a\in [-1,1].
\]
Then
\[
\int_0^T b_i (t) f\left( s, b_l(s)\right) \d s =\begin{cases}
\frac{1}{2}\int_0^T \left(  f(s,1) -  f(s,0) \right) \d s +T\, O(\epsilon ) &(i,l) = (2,1),\\
T\, O(\epsilon ) \quad &(i,l) \ne (2,1).
\end{cases} 
\] 

These observations are helpful in finding $b_1^k, b_2^k \colon [0,T] \to [-1,1]$ such that for every continuous $f$
\eqnb\label{limiting_prop_of_b1_b2}
\int_0^T b_i^k (s) f\left( s, b_l^k (s)\right) \d s \to \begin{cases}
\frac{1}{2}\int_0^T ( f(s,1) - f(s,0) ) \d s &(i,l) = (2,1),\\
0 \quad &(i,l) \ne (2,1)
\end{cases} 
\eqne
as $k\to \infty$. Indeed one can take $b_1^k$, $b_2^k$ to be oscillations of the form \eqref{def_of_b1_b2}, but of higher frequency,
\eqnb\label{def_of_b1k_b2k}
b_1^k (t) \coloneqq b_1 (kt),\qquad b_2^k (t) \coloneqq b_2(kt),
\eqne
where we extended $b_1, b_2$ $T$-periodically to the whole line, see Fig. \ref{b1_b2}.
\begin{figure}[h]
\centering
 \includegraphics[width=0.7\textwidth]{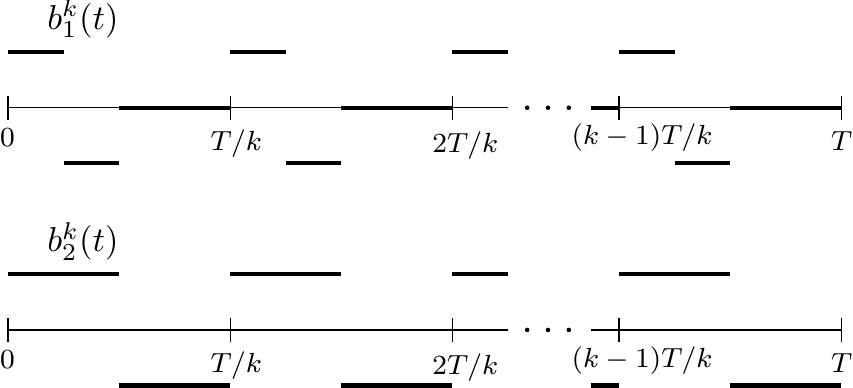}
 \nopagebreak
  \captionof{figure}{The functions $b_1^k$, $b_2^k$.}\label{b1_b2} 
\end{figure}

In order to see that such a choice gives the convergence in \eqref{limiting_prop_of_b1_b2} note that continuity of $f$ implies that $\varepsilon_k \to 0$ as $k\to \infty$, where $\varepsilon_k>0$ is the smallest positive number such that
\eqnb\label{contin_prop_f}
|f(t,a)-f(s,a)|\leq \varepsilon_k
\eqne
whenever $a\in [-1,1]$ and $s,t\in [0,T]$ are such that $|t-s|\leq T/k$. Thus, if $(i,l)=(2,1)$ we write
\eqnb\label{conv_computation}
\begin{split}
\int_0^T b_2^k (s) &f\left( s, b_1^k (s)\right) \d s = \sum_{p=0}^{k-1} \int_{pT/k}^{(p+1)T/k} b_2^k (s) f\left( s, b_1^k (s)\right) \d s \\
&= \sum_{p=0}^{k-1} \left( \int_{pT/k}^{(p+1/2)T/k}  f\left( s ,1\right) \d s - \int_{(p+1/2)T/k}^{(p+1)T/k}  f\left( s ,0\right) \d s  \right) \\
&= \sum_{p=0}^{k-1} \left( \frac{1}{2} \int_{pT/k}^{(p+1)T/k}  f\left( s ,1\right) \d s  - \frac{1}{2} \int_{pT/k}^{(p+1)T/k}  f\left( s ,0\right) \d s +  \frac{T}{k}O(\varepsilon_k ) \right)\\
&= \frac{1}{2}\int_0^T ( f(s,1) - f(s,0) ) \d s + T\, O(\varepsilon_k).
\end{split}
\eqne
Thus
\[
\int_0^T b_2^k (s) f\left( s, b_1^k (s)\right) \d s \to \frac{1}{2}\int_0^T ( f(s,1) - f(s,0) ) \d s  \qquad \text{ as }k\to \infty,
\]
and in the same way one can show that
\[
\int_0^T b_i^k (s) f\left( s, b_l^k (s)\right) \d s \to 0 \qquad \text{ as }k\to \infty
\]
if $(i,l)\ne (2,1)$. Therefore we obtain \eqref{limiting_prop_of_b1_b2}.

In a similar way one can show that for such choice of $b_1^k$, $b_2^k$, the upper limit of the integrals in \eqref{limiting_prop_of_b1_b2} can be replaced by any $t\in [0,T]$, that is
\eqnb\label{limiting_prop_any_t}
\int_0^t b_i^k (s) f\left( s, b_j^k (s)\right) \d s \to \begin{cases}
\frac{1}{2}\int_0^t ( f(s,1) - f(s,0) ) \d s &(i,l) = (2,1),\\
0 \quad &(i,l) \ne (2,1)
\end{cases} 
\eqne
as $k\to \infty$, uniformly in $t\in [0,T]$.
To this end, given $t\in [0,T]$ let $q\in \{ 0,\ldots , k-1\}$ be such that $t\in [qT/k,(q+1)T/k )$ and write the left-hand side of \eqref{limiting_prop_any_t} above as
\[ \sum_{p=0}^{q-1}  \int_{pT/k}^{(p+1)T/k} b_2^k (s) f( s, b_1^k (s)) \d s+\int_{qT/k}^t  b_2^k (s) f( s, b_1^k (s)) \d s.
\]
The sum from $p=0$ to $q-1$ can be treated in the same way as the sum over all $p$'s in the calculation \eqref{conv_computation} above to give
\[
\frac{1}{2}\int_0^{qT/k} ( f(s,1) - f(s,0) ) \d s + \frac{qT}{k} O(\varepsilon_k ) .
\]
The remaining term can be treated using boundedness of $f$ (note $|f|\leq N$ for some $N>0$ due to continuity of $f$ and to the fact that its domain $[0,T]\times [-1,1]$ is compact) by writing
\[
\left| \int_{qT/k}^t  b_2^k (s) f( s, b_1^k (s)) \d s - \frac{1}{2}\int_{qT/k}^t \left(  f( s,1)- f(s,0) \right)  \d s \right| \leq 2 N \left| t- qT/k \right| \leq 2NT/k ,
\]
and thus we obtain \eqref{limiting_prop_any_t} in the case $(i,l)=(2,1)$. The case $(i,l)\ne (2,1)$ follows similarly.

Moreover, due to the oscillatory behaviour of $b_1^k$, $b_2^k$ as $k$ increases we also see that each of $b_1^k,b_2^k $ converges to $0$ in a weak sense, that is
\eqnb\label{limiting_prop_for_g}
\int^t_0 b_i^k (s) g(s) \, \d s \to 0 \qquad \text{ as } k\to \infty, i=1,2,\text{ uniformly in } t\in [0,T]
\eqne
for any continuous $g\colon [0,T]\to \RR$.

The above ideas are a basis of the proof of Theorem \ref{thm_existence_of_osc_proc}, in which $x$ plays no role and the processes $a_1^k$, $a_2^k$ are obtained by a smooth approximation of $b_1^k$, $b_2^k$, respectively.
\begin{proof}[\nopunct Proof of Theorem \ref{thm_existence_of_osc_proc}]
Let $b_1^k, b_2^k\colon [0,T] \to [-1,1]$ be defined by \eqref{def_of_b1k_b2k} above. 
Given $k\geq 0$ let $\varepsilon_k >0$ be the smallest number such that
\eqnb\label{unif_cont_of_Fij_Gi}
|F_{i,l} (x,t,a) - F_{i,l} (x,s,a) |, |G_i (x,t)- G_i(x,s) |\leq \varepsilon_k,\qquad i,l=1,2
\eqne
whenever $x\in P$, $a\in [-1,1]$ and $t,s\in [0,T]$ are such that $|t-s|\leq T/k$.
Due to the uniform continuity of $F_{i,l}$'s and $G_i$'s we obtain $\varepsilon_k\to 0 $ as $k\to \infty$. Moreover, from boundedness we obtain $N>0$ such that $|F_{i,l}|, |G_i|\leq N$ for $i,j=1,2$. Thus applying \eqref{limiting_prop_any_t}, with $f(t,a)\coloneqq F_{i,l} (x,t,a)$ (for every $x$) and with the continuity property \eqref{contin_prop_f} replaced by the uniform continuity of $F_{i,j}$'s \eqref{unif_cont_of_Fij_Gi} and by the boundedness $|F_{i,l}|\leq N$ we obtain
\[
\int_0^t b_i^k (s) F_{i,l}\left(x, s, b_l^k (s)\right) \d s \to \begin{cases}
\frac{1}{2}\int_0^t ( F_{2,1}(x,s,1) - F_{2,1}(x,s,0) ) \d s &(i,l) = (2,1),\\
0 \quad &(i,l) \ne (2,1)
\end{cases} 
\] 
as $k\to \infty$ uniformly in $x\in P$, $t\in [0,T]$. Similarly applying \eqref{limiting_prop_for_g} with $g(t)\coloneqq G_i (x,t)$ we obtain
\[
\int_0^t b_i^k (s) G_i (x,s) \, \d s \to 0\qquad \text{ as } k\to \infty
\]
uniformly in $x\in P$, $t\in [0,T]$, $i=1,2$. Thus, altogether
\eqnb\label{osc_process_conv_b}\begin{split}
\int_0^t b_i^k (s) &\left( G_i (x,s) + F_{i,1} \left( x,s,b_1^k(s) \right) + F_{i,2} \left( x,s,b_2^k (s) \right)  \right) \d s\\
&\stackrel{k\to \infty}{\longrightarrow } \begin{cases} \frac{1}{2}\int_0^t  \left( F_{2,1} (x,s,1)-F_{2,1} (x,s,0) \right) \d s \quad &i=2, \\
0 &i=1
\end{cases}
\end{split}
\eqne
uniformly in $(x,t)\in P \times [0,T]$. Thus the oscillatory processes $b_1^k,b_2^k$ (defined by \eqref{def_of_b1k_b2k}) satisfy all the claims of the theorem, except for the $C^\infty$ regularity. To this end let $a_1^k, a_2^k\in C^\infty (\RR ; [-1,1])$ be such that
\[
\left| \lewy t \in [0,T] \colon a_i^k (t) \ne b_i^k (t) \prawy \right| \leq \frac{1}{k} ,\qquad i=1,2.
\]
Such $a_1^k$, $a_2^k$ can be obtained by extending $b_1^k$, $b_2^k$ to the whole line by zero and mollifying.
Clearly, such definition of the processes $a_1^k$, $a_2^k$ and the boundedness $|F_{i,l}|,|G_i|\leq N$ gives that the difference between the left-hand sides of \eqref{osc_process_conv} and \eqref{osc_process_conv_b} is bounded by
\[
6N/k \to 0 \qquad \text{ as }k\to \infty ,
\]
 which shows that these left-hand sides converge to the same limit
\[
\begin{cases} \frac{1}{2}\int_0^t  \left( F_{2,1} (x,s,1)-F_{2,1} (x,s,0) \right) \d s \quad &i=2, \\
0 &i=1
\end{cases}
\]
uniformly in $(x,t)\in P\times [0,T]$, as required.
 \end{proof}

\section{The geometric arrangement}\label{sec_geom_arrangement}
In this section we construct the \emph{geometric arrangement}, that is $T>0$, $\tau \in (0,1)$, $z\in \RR^3 $, sets $U_1,U_2\Subset P$ with disjoint closures and the respective structures $(v_1,f_1,\phi_1)$, $(v_2,f_2,\phi_2)$ such that 
\[
f_2^2 +T v_2 \cdot F[v_1,f_1] > |v_2 |^2\quad \text{ in }U_2,
\]
\[
f_2^2(y) +T v_2(y) \cdot F[v_1,f_1](y) > \tau^{-2} \left(  f_1 (R^{-1}x) +  f_2(R^{-1}x) \right)^2 
\]
for all $x\in G = R(\overline{U_1} \cup \overline{U_2})$, where $y =  R^{-1} (\Gamma(x))$.
According to the considerations of Section \ref{sec_the_setting}, this construction concludes the proof of Theorem \ref{point_blowup_thm}.\\

Let 
\[ U\coloneqq (-1,1) \times (1/8 , 7 /8 ), \]
and let $v\in C_0^\infty (U; \RR^2 )$ be any vector field satisfying 
\[ \begin{cases}
v_1 (-x_1, x_2) = v_1(x_1,x_2) ,\\
v_2(-x_1,x_2) = -v_2 (x_1,x_2) 
\end{cases}
\]
and
\[
\mathrm{div} (x_2 \, v(x_1,x_2)=0,\qquad (x_1,x_2)\in P.
\]
One can take for instance 
\[
v(x_1,x_2 ) \coloneqq x_2^{-1} J \left( (-(x_2-1/2),x_1 ) \chi_{\{ 1/16 < | (x_1,x_2-1/2) |< 1/8\}}  \right),
\]
where $J$ denotes a sufficiently fine mollification (and $\chi$ denotes the indicator function), as in the recipe for a structure presented in Section \ref{sec_recipe_for_structure}. Following the recipe, let $f \in C_0^\infty (P; [0,\infty ))$ be such that $\supp \, f = \overline{U}$, $f> |v|$ in $U$ and $Lf >0$ at points of $U$ of sufficiently small distance from $\partial U$. Furthermore, construct $f$ in a way that
\[
f(-x_1,x_2)=f(x_1,x_2).
\]
We show existence of such $f$ in Lemma \ref{lemma_existence_of_f_with_Lf_rectangle}. Let $\phi \in C_0^\infty (U; [0,1])$ be a cutoff function such that $\supp \, v \subset \{ \phi =1 \} $ and $Lf>0$ in $U\setminus \{\phi =1 \}$. Thus we obtained a structure $(v,f,\phi )$ on $U$. Consider the pressure interaction function $F\coloneqq F[v,f]$ (recall \eqref{def_of_p_interaction_fcn}) and let $A\in \RR$, $B,C,D,N>0$ and $\kappa = 10^4 C/D$ be the constants given by Lemma \ref{lem_properties_of_F}.

Since the structure $(v,f,\phi )$ satisfies the condition of Lemma \ref{lem_prop_of_p[v,f]} (ii), we see that the first component of $F[v,f]$ is odd when restricted to the $x_1$ axis, that is
\[
F_1(-x_1,0)=-F_1 (x_1,0) , \quad x_1 \in \RR.
\]
Thus, in the view of Lemma \ref{lem_properties_of_F} (ii), we observe that $A\ne 0$ and 
\[
-B=F_1(-A,0)= \min_{x_1\in \RR} F_1(x_1,0).
\] 

\subsection{A simplified geometric arrangement}\label{sec_simplified_geom_arrangement}
At this point we pause for the moment to present a certain simplified geometric arrangement. Although the simplified arrangement has the unfortunate property of being impossible, it offers a good perspective on the main difficulty. We also explain the strategy for overcoming this difficulty. The reader who is not interested in the simplified arrangement is referred to the next section (Section \ref{sec_copies_of_U}), where we proceed with the presentation of the geometric arrangement proper.

From Lemma \ref{lem_properties_of_F} (ii) we see that there exists a rectangle $U_2\Subset P$ such that $F_1[v,f] \geq B/2$ in $U_2$. Let $v_2=(v_{21},v_{22}) \in C_0^\infty (U_2;\RR^2)$ be such that $\mathrm{div} \, (x_2 v_2(x_1,x_2))=0$ for $(x_1,x_2)\in P$,
\[v_{22}=0,\,v_{21}\geq 0, \text{ and }v_{2}=(1,0) \text{ in some closed rectangle }K\subset U_2.
\]
\begin{warning}\label{warning}
Such $v_2$ does not exist!\\
Indeed, take $w\coloneqq x_2 v_2$ and let $K'$ be a rectangle such that its left edge is the left edge of $K$ and its right edge lies on $\partial U_2$. Integrating $\mathrm{div}\, w$ over $K$ we obtain
\[
0=\int_K \mathrm{div}\, w =\int_{\partial K'} w \cdot n = \int_{\partial_{L} K'} w_1 = \int_{\partial_{L} K'} x_2  >0,
\]
where $\partial_{L} K'$ denotes the left edge of $K'$.
\end{warning} 
Let $(z_1,z_2)$ be an interior point of $K$, $z\coloneqq (z_1,z_2,0)\in \RR^3$, $U_1\coloneqq U$, $v_1\coloneqq v$, $f_1\coloneqq f$, $\phi_1 \coloneqq \phi$ (note then $F=F[v_1,f_1]$) and let $\tau \in (0,1)$ be sufficiently small such that
\eqnb\label{prototype_scaling}
R^{-1} \left( \tau  R(\overline{U_1} \cup \overline{U_2})+z \right)\subset K,
\eqne
see Fig. \ref{prototype_geometrical_fit}.
\begin{figure}[h]
\centering
 \includegraphics[width=0.5\textwidth]{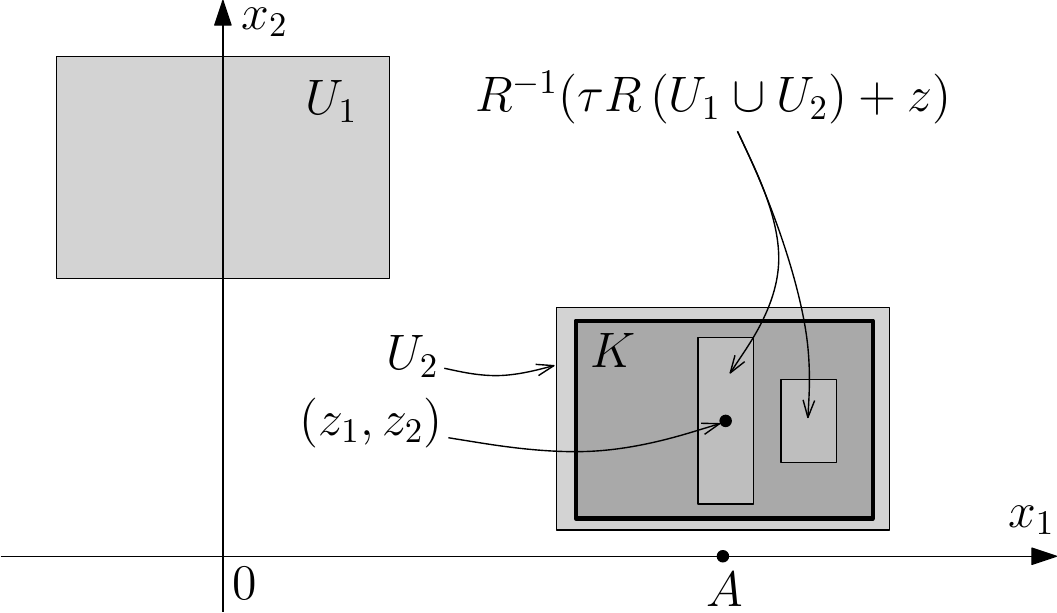}
 \nopagebreak
 \captionsetup{width=0.9\textwidth} 
  \captionof{figure}{The simplified arrangement. Note that it is not quite correct, see Warning \ref{warning}.}\label{prototype_geometrical_fit} 
\end{figure}
Let $f_2$, $\phi_2$ be any functions such that $(v_2,f_2,\phi_2)$ is a structure on $U_2$ (that is define $f_2$, $\phi$ as described in the recipe in Section \ref{sec_recipe_for_structure}). Then \eqref{fairies_extra_ineq} follows trivially for every $T>0$ by noting that
\eqnb\label{temp_calc1_convenient_prop}
v_2 \cdot F = v_{21} F_1 \geq 0 \quad \text{ in }U_2,
\eqne
and so
\eqnb\label{temp_calc1}
f_2^2 +T v_2 \cdot F = f_2^2 + Tv_{21} F_1 \geq  f_2^2 > |v_2 |^2 \qquad \text{ in } U_2.
\eqne
Moreover, \eqref{fairies_scaling} follows provided we choose $T>2\tau^{-2} \| f_1+f_2 \|_\infty^2/B $. Indeed, then we obtain
\eqnb\label{temp_calc2_convenient_prop}
T F_1  \geq \tau^{-2} \| f_1+f_2 \|_\infty^2 \quad \text{ in } U_2,
\eqne
and so letting $x\in R(\overline{U_1}\cup \overline{U_2})$ and $y\coloneqq R^{-1} (\tau x+z)$ we see that \eqref{prototype_scaling} gives $y\in K$ and thus
\eqnb\label{temp_calc2}
f_2^2(y) +T v_2(y) \cdot F(y)= f_2^2 (y) + T F_1(y) \geq \tau^{-2} \| f_1+f_2 \|_\infty^2,
\eqne
as required.

This concludes the simplified geometric arrangement. Note, however, it does not exist due to Warning \ref{warning}. In fact, it is clear that $v_2$ cannot have $(1,0)$ as the only direction, which is, roughly speaking, a consequence of the fact that any weakly divergence-free vector field in $\RR^2$ must ``run in a loop'', cf. Fig \ref{vector_fields_w}. Thus, given any of the quantities 
\[ F_1, F_2,-F_1,-F_2\]
there exists a region in $P$ such that at least one of the ingredients of the inner product \[ v_2\cdot F = v_{21}F_1 + v_{22}F_2
\]
gives the given quantity multiplied by $v_{21}$ or $v_{22}$ (the size of which obviously depending on the choice of $v_2$). 
Thus the calculations \eqref{temp_calc1}, \eqref{temp_calc2}, in which we used the very convenient properties \eqref{temp_calc1_convenient_prop}, \eqref{temp_calc2_convenient_prop} immediately become useless and at this point it is not clear how to estimate $v_2 \cdot F$ to obtain the required relations \eqref{fairies_extra_ineq}, \eqref{fairies_scaling}.\\

In the remainder of this section we sketch a more elaborate construction of sets $U_1$ and $U_2$ as well as their structures that solve this difficulty. In particular we point out the relations that will replace \eqref{temp_calc1_convenient_prop}, \eqref{temp_calc2_convenient_prop} in showing the required relations \eqref{fairies_extra_ineq}, \eqref{fairies_scaling}. The construction is then presented in detail in the following Sections \ref{sec_copies_of_U}-\ref{sec_constr_f2_and_rest}. \\

First of all, we will consider the rescaling of the set $U$ and its structure $(v,f,\phi)$, that is for $\alpha\in \RR, \rho >0$ and $\sigma >0$ we will consider a set $U^{\alpha , \rho}$ and a structure $(v^{\alpha, \rho , \sigma},f^{\alpha, \rho , \sigma},\phi^{\alpha, \rho })$ on $U^{\alpha , \rho}$. Here $\alpha$ corresponds to a translation in the $x_1$ direction, $\rho $ scales the size of $U$ and $\sigma$ scales the magnitude of $v$ and $f$. We will observe that manipulating the values of $\alpha, \rho, \sigma$ gives us certain amount of freedom in the manipulation of the shape of the pressure interaction function
\[F^{\alpha , \rho , \sigma}\coloneqq F[v^{\alpha, \rho , \sigma},f^{\alpha, \rho , \sigma}],
\]
and so we will consider a disjoint union of $U$ together with its two rescalings,
\[
U \cup U^{a',r'} \cup U^{a'',r''},
\]
along with the corresponding structure 
\[(v,f,\phi )+\left( v^{a',r',s'},f^{a',r',s'},\phi^{a',r'}\right)+\left( v^{a'',r'',s''},f^{a'',r'',s''},\phi^{a'',r''} \right) ,\]
where the sum is understood in an entry-wise sense. Here, the values of $a',a'',r',r'',s',s''$ will be chosen in a particular way, roughly speaking such that the (joint) pressure interaction function 
\[H \coloneqq F+F^{a',r',s'}+F^{a'',r'',s''}\] enjoys a similar decay to $F$ (recall Lemma \ref{lem_properties_of_F} (iii)) and, when restricted to the $x_1$ axis, its first component $H_1$ admits maximum $7B$ at $A$ and minimum greater than or equal to $-1.005B$ (rather than maximum $B$ and minimum $-B$, which is the case for $F_1$), see Fig. \ref{the_arrangement_sketch}.
\begin{figure}[h]
\centering
 \includegraphics[width=\textwidth]{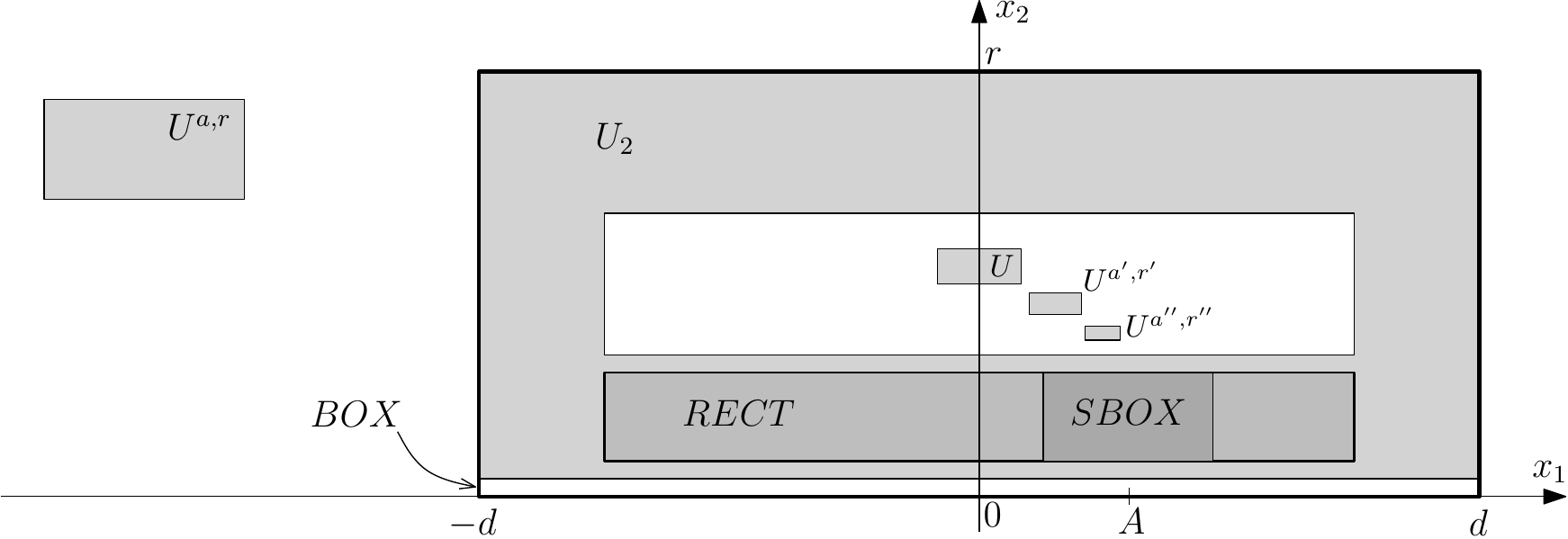}
 \nopagebreak
 \captionsetup{width=0.9\textwidth} 
  \captionof{figure}{A sketch of the geometric arrangement (see Fig. \ref{the_arrangement} for a more detailed sketch). Some proportions are not conserved on the sketch.}\label{the_arrangement_sketch} 
\end{figure}
Then, given a small parameter $\varepsilon >0$, we  will find numbers $d,r=O(1/\varepsilon )$ with $d\gg r$, $U_2\Subset P$ and $v_2 \in C_0^\infty (U_2;\RR^2)$ such that
\[
U_2 \subset \mbox{\emph{BOX}} \coloneqq [-d,d]\times [0,r],
\]
\[
U_2 \text{ is a rectangular ring encompassing } U \cup U^{a',r'} \cup U^{a'',r''},
\]
namely $U_2=V\setminus \overline{W}$ where $V,W\Subset P$ are rectangles such that 
\[U \cup U^{a',r'} \cup U^{a'',r''} \Subset W\Subset V ,\]
see Fig. \ref{the_arrangement_sketch}, and 
\[
v_2 = (1,0) \quad \text{ in }\mbox{\emph{RECT}}\subset U_2 ,
\]
where $\mbox{\emph{RECT}}$ will be a carefully chosen rectangle located sufficiently close to the $x_1$ axis so that 
\eqnb\label{temp_values_of_H}\begin{split}
&H_1\geq -1.01B \text{ in } \mbox{\emph{RECT}},\\
&H_1\geq 6.99B \text{ in some rectangle }\mbox{\emph{SBOX}} \subset \mbox{\emph{RECT}} .
\end{split}
\eqne
We will then choose $\tau$, $z$ such that
\eqnb\label{temp_box_fits_into_sbox}
R^{-1} ( \tau R (\mbox{\emph{BOX}} ) + z) \subset \mbox{\emph{SBOX}} ,
\eqne
see Fig. \ref{the_arrangement}, and we will define a pair of numbers $a=O(-\varepsilon^{-2})$, $s=O(\varepsilon^{-5/2})$ such that the rescaling $U^{a,r}$ of $U$ together with the rescaled structure $(v^{a,r,s},f^{a,r,s},\phi^{a,r})$ satisfies 
\eqnb\label{temp_Uar_fits_into_rect}
R^{-1} \left( \tau R \left( \overline{U^{a,r}} \right) + z\right) \subset \mbox{\emph{RECT}} ,
\eqne
see Fig. \ref{the_arrangement}, and that the pressure interaction function $F^{a,r,s}=F[v^{a,r,s},f^{a,r,s}]$ is of particular size when restricted to $\mbox{\emph{BOX}}$, that is $F_2^{a,r,s}$ is small (in some sense) and
\eqnb\label{temp_values_of_Fars}
1.03B \leq F_1^{a,r,s} \leq 1.05 B \qquad \text{ in } \mbox{\emph{BOX}}.
\eqne
For this we will crucially need the last property in Lemma \ref{lem_properties_of_F}, which, roughly speaking, quantifies the decay (in $x_1$) of the pressure interaction function. We will then set 
\[U_1 \coloneqq U \cup U^{a',r'}\cup U^{a'',r''} \cup U^{a,r} \]
together with the structure 
\[\begin{split} (v_1,f_1,\phi_1)\coloneqq (v,f,\phi )&+\left( v^{a',r',s'},f^{a',r',s'},\phi^{a',r'}\right)\\
&\hspace{1.5cm}+\left( v^{a'',r'',s''},f^{a'',r'',s''},\phi^{a'',r''} \right) +\left( v^{a,r,s},f^{a,r,s},\phi^{a,r} \right) ,
\end{split}
\] so that the (total) pressure interaction function is
\[
F^*\coloneqq F[v_1,f_1]= H+F^{a,r,s} .
\]
$\mbox{}$

Observe that \eqref{temp_box_fits_into_sbox}, \eqref{temp_Uar_fits_into_rect} give in particular
\[
R^{-1}  \left( \tau R \left( \overline{U_1}\cup \overline{U_2} \right) + z\right) \subset \mbox{\emph{RECT}},
\]
that is, as in the simplified setting (see \eqref{prototype_scaling}), the cylindrical projection $R^{-1}$ maps $\Gamma( G)$ (recall $G=R\left( \overline{U_1}\cup \overline{U_2} \right) $) into the region in $P$ in which $v_2=(1,0)$. Moreover, \eqref{temp_values_of_H} and \eqref{temp_values_of_Fars} immediately give
\eqnb\label{temp_values_of_F*}
\begin{split}
F^*_1 &>0.01B \quad \text{ in } \mbox{\emph{RECT}},\\
F^*_1 &>8B \qquad \text{ in } \mbox{\emph{SBOX}}.
\end{split}
\eqne
Furthermore, it can be shown (using the properties of the choice of $\varepsilon, d, r, a, s, v_2$ and the decay of $H$) that
\eqnb\label{temp_bound_from_below_on_v2_F}
v_2\cdot F^* \geq -1.1\varepsilon B \quad \text{ in }\supp \, v_2 .
\eqne

Finally, we will make a particular choice of $f_2$, $\phi_2$ and $T>0$ such that $(v_2,f_2,\phi_2)$ is a structure on $U_2$ and the properties \eqref{fairies_extra_ineq}, \eqref{fairies_scaling} hold. The proof of \eqref{fairies_extra_ineq} will be in essence similar to the calculation \eqref{temp_calc1}, but with the inequality \eqref{temp_calc1_convenient_prop} replaced by \eqref{temp_bound_from_below_on_v2_F} and a property of the choice of $T$. The proof of \eqref{fairies_scaling} is, in a sense, a more elaborate version of the calculation \eqref{temp_calc2}. Namely, rather than taking any $x\in R \left( \overline{U_1}\cup \overline{U_2} \right)$ we will consider two cases, which correspond to different means of substituting the use of the inequality \eqref{temp_calc2_convenient_prop}:\\

\emph{Case 1.} $x\in R \left( \overline{U^{a,r}} \right)$. Then $y\in \mbox{\emph{RECT}}$ by \eqref{temp_Uar_fits_into_rect} and we will replace \eqref{temp_calc2_convenient_prop} by the first inequality in \eqref{temp_values_of_F*} and the properties of $f_2$ and $T$.\\

\emph{Case 2.} $x\in R\left( \overline{U} \cup \overline{U^{a',r'}} \cup \overline{U^{a'',r''}} \cup \overline{U_2} \right)\subset R(\mbox{\emph{BOX}})$. Then $y \in \mbox{\emph{SBOX}}$ by \eqref{temp_box_fits_into_sbox} and we will replace \eqref{temp_calc2_convenient_prop} by the second inequality in \eqref{temp_values_of_F*} and the properties of $f_2$ and $T$.\vspace{0.5cm}\\
We now present the rigorous version of this explanation.

\subsection{The copies of $U$ and its structure}\label{sec_copies_of_U}

Let us consider disjoint ``copies'' of $U$ and its structure $(v,f,\phi )$ and arranging these copies into a favourable composition. Namely, for $\alpha \in \RR$, $\rho >0$, $\sigma >0$ let
\eqnb\label{defs_of_copying}
\begin{split}
U^{\alpha , \rho } &\coloneqq \lewy (x_1,x_2) \in \RR^2  \colon \left( \frac{x_1-\alpha }{\rho } , \frac{x_2}{\rho }\right)  \in U \prawy, \\
v^{\alpha , \rho , \sigma } (x_1,x_2) &\coloneqq \sigma \, v \left( \frac{x_1-\alpha }{\rho } , \frac{x_2}{\rho }\right),\\
f^{\alpha , \rho , \sigma } (x_1,x_2) &\coloneqq \sigma \, f \left( \frac{x_1-\alpha }{\rho } , \frac{x_2}{\rho }\right),\\
\phi^{\alpha , \rho  } (x_1,x_2) &\coloneqq  \phi \left( \frac{x_1-\alpha }{\rho } , \frac{x_2}{\rho }\right),\\
F^{\alpha , \rho , \sigma } (x_1,x_2) &\coloneqq \frac{\sigma^2 }{\rho } F \left( \frac{x_1-\alpha }{\rho } , \frac{x_2}{\rho }\right).
\end{split}
\eqne
(Recall $F=F[v,f]$ is the pressure interaction function.)
\begin{figure}[h]
\centering
 \includegraphics[width=\textwidth]{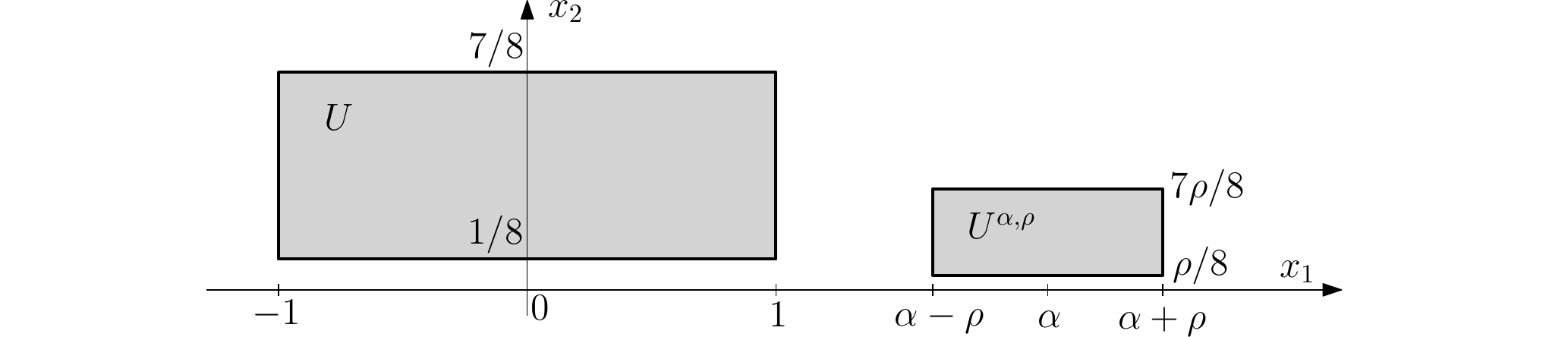}
 \nopagebreak
  \captionof{figure}{The set $U^{\alpha, \rho }$, where $\rho <1$.}\label{sets_u_and_ualpharho} 
\end{figure}

Here $\alpha \in \RR$ denotes the translation in $x_1$ direction of $U$ and its structure and $\rho$ denotes the scaling of the variables, see Fig. \ref{sets_u_and_ualpharho}. Also, $\sigma $ denotes the scaling in magnitude of $v$ and $f$. A direct consequence of the definitions above is that $U^{\alpha , \rho } \Subset P$, $(v^{\alpha , \rho , \sigma }, f^{\alpha , \rho , \sigma }, \phi^{\alpha, \rho} )$ is a structure on $U^{\alpha , \rho }$ and $F^{\alpha , \rho, \sigma }$ is a pressure interaction function corresponding to $U^{\alpha, \sigma }$, namely
\[
F^{\alpha , \rho, \sigma } = F [v^{\alpha , \rho, \sigma } , f^{\alpha , \rho , \sigma }] ,
\]
for each choice of $\alpha \in \RR$, $\rho, \sigma >0$. Now let 
$a', a'' \in \RR$, $r', r'', s', s'' >0$ be such that the sets $U$, $U^{a',r'}$, $U^{a'',r''}$ have disjoint closures and the function
\eqnb\label{def_of_H}
H \coloneqq  F + F^{a',r',s'} + F^{a'',r'',s''}
\eqne
satisfies
\begin{enumerate}
\item[(i)] $H_1(A,0) = 7B$,
\item[(ii)] $H_1(x_1,0) \geq -1.005 B$,
\item[(iii)] $|H(x) |\leq 2C / |x|^4$ for $|x|>2|A|$.
\end{enumerate}
Such a choice is possible due to the following simple geometric argument (which is sketched in Fig. \ref{choice_of_primes_doubleprimes}). Let $s'$, $r'$ satisfy $(s')^2/r'=2$ (so that we have $\max F^{a',r',s'}_1 (\cdot , 0) =2B = - \min  F^{a',r',s'}_1 (\cdot , 0)$) and take $r'>0$ so small that $|F_1^{0,r',s'} (x_1,0)| <0.001 B$ for $x_1$ such that $F_1(A+x_1,0)<0.999B$. Then choose $a'$ such that the maxima of both $F_1 (x_1,0)$ and $F_1^{a',r',s'}(x_1,0)$ coincide (at $x_1=A$).
Then, similarly, choose $s''$, $r''$ so that $(s'')^2/r''=4$ and $r''>0$ is small enough so that $|F_1^{0,r'',s''} (x_1,0)| <0.001 B$ for $x_1$ such that $F_1^{a',r',s'}(A+x_1,0)<0.999 \cdot (2B)$, and choose $a''$ so that the maximum of $F_1^{a'',r'',s''}(x_1,0)$ occurs at $x_1=A$.
This way we obtain (i) and (ii) by construction, while (iii) follows given $r'$ and $r''$ were chosen small enough. Furthermore, taking $r'$ and $r''$ small ensures that the sets $U$, $U^{a',r'}$, $U^{a'',r''}$ have disjoint closures ($r'<1/8$ and $r''<r'/8$ suffices, cf. Fig. \ref{the_arrangement_sketch}).

\begin{figure}[h]
\centering
 \includegraphics[width=\textwidth]{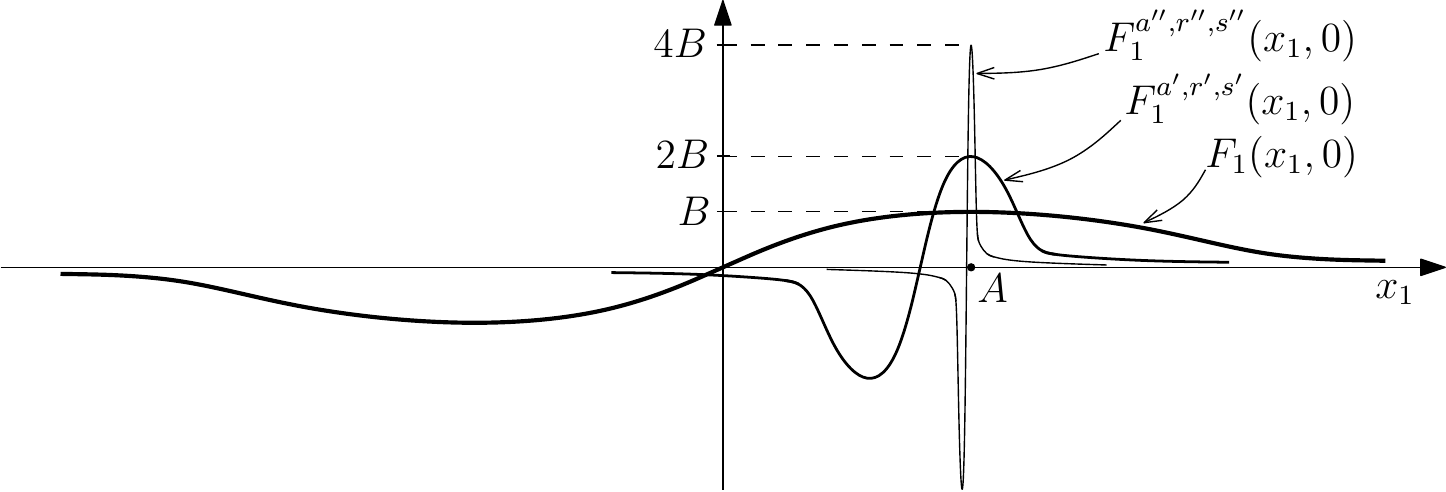}
 \nopagebreak
  \captionof{figure}{The choice of $a',a'',r',r'',s',s''$.}\label{choice_of_primes_doubleprimes} 
\end{figure}

Thus by specifying $a', a'',  r', r'', s', s''$ we added to $U$ two disjoint copies of it such that the total pressure interaction function $H$ has a specific behaviour on the $x_1$ axis. We now want to specify the behaviour of $H$ on a strip in $P$ near the $x_1$ axis. That is, by continuity, we see that there exists $E>0$ (sufficiently small) such that
\begin{enumerate}
\item[(iv)] the strip $\{ 0 < x_2 <E \} \subset P$ is disjoint  from $U$, $U^{a',r'}$, $U^{a'',r''}$,
\item[(v)] $H(x) \geq -1.01 B$ in the strip $\{ 0<x_2 < E\}$,
\item[(vi)] $H(x) \geq 6.99 B$ for $x\in P$ such that $|x_1 - A | < \kappa E$, $0<x_2<E$. 
\end{enumerate}
Here claim (v) also uses the decay property (iii) of $H$. 

\subsection{Construction of $v_2$ and $U_2$}\label{sec_construct_U2_v2}
Now let $\varepsilon >0$ be a small parameter (whose value we fix below) and let $d, r>0$ be defined by
\[
r\coloneqq E/\varepsilon ,\quad d \coloneqq \kappa r .
\] 
Note that by taking $\varepsilon$ small, both $r$ and $d$ become large, and since 
\eqnb\label{kappa_is_large}
\kappa =10^4 C/D \geq 10^4 \quad \text{ we have } \quad d\geq 10^4 r.
\eqne
In fact, $\varepsilon$ is the main parameter of the construction and in what follows we will use certain algebraic inequalities, all of which rely on $\varepsilon$ being sufficiently small. We gather all these properties here in order to demonstrate that the argument is not circular. Namely, let $\varepsilon >0$ be sufficiently small that
\eqnb\label{how_small_is_eps}
\begin{split}
\varepsilon &< 1/10, \quad d-r> 2(|A|+\kappa E), \quad r>10, \quad r>20|A|, \\
d &>2 \,\mathrm{diam}\left( U\cup U^{a',r'} \cup U^{a'',r'' }\right), \quad 
\varepsilon <  \kappa /N,  \quad \varepsilon^2 < \frac{B E^4 }{2\cdot 10^6 C}.
\end{split}
\eqne
We now construct $v_2 $ by sharpening the observation from Fig. \ref{vector_field_deformed_w}. Namely we let $v_2$ be given by the following lemma.
\begin{lemma}\label{lemma_existence_of_v2}
Given $d,r,\varepsilon >0$ such that $d>r$, $\varepsilon < 1/10$ there exists $v_2 = (v_{21},v_{22})\in C_0^\infty (P; \RR^2 )$ such that
\begin{enumerate}
\item[(i)] $\mathrm{div}\, (x_2 \, v_2 (x_1,x_2)) =0$,
\item[(ii)] $\supp \, v_2 \subset (-d,d) \times (0.005\varepsilon r,r) \setminus [-(d-r),d-r ] \times [\varepsilon r , r/10 ]$,
\item[(iii)] $|v_{22} |< \varepsilon /2 $, $-\varepsilon^2 \leq v_{21} \leq 1$ with
\[
v_{21}\geq 0, v_{22}=0\quad \text{ in } \quad [-(d-r),d-r] \times (0 ,  \varepsilon r ),
\]
\item[(iv)] $v_2 = (1,0)$ in $[-(d-r),d-r] \times [0.02 \varepsilon r , 0.98 \varepsilon r ]$.
\end{enumerate}
\end{lemma}
Before proving the lemma, we note that the construction of such a vector field $v_2$ is one of the central ideas of the proof of Theorem \ref{point_blowup_thm}. We will shortly see that it is thanks to $v_2$ that we can overcome the difficulty posed by Warning \ref{warning}. Indeed, we can already see (in part (iv) above) that $v_2$ keeps constant direction and magnitude in a rectangular-shaped subset of $P$ which is located near the $Ox_1$ axis, and that $v_2=O(\varepsilon )$ whenever its direction is different (which we will see in the proof below). 
\begin{proof}
Let $w\colon P \to \RR^2$ be defined by
\[
w(x_1,x_2) = \begin{cases}
(x_2 , 0 ) \qquad &\text{ in } R_1,\\
\frac{\varepsilon}{2} (d-x_1 , x_2 ) & \text{ in } R_2,\\
-\varepsilon^2 (x_2 , 0) & \text{ in } R_3,\\
\frac{\varepsilon }{2} (x_1 + d , - x_2 ) & \text{ in } R_4 , \\
0 & \text{ in } P\setminus (R_1\cup R_2 \cup R_3 \cup R_4 ),
\end{cases}
\]
where regions $R_1$, $R_2$, $R_3$, $R_4$ are as indicated in Fig. \ref{construction_of_w}. Observe that these regions, and the form of $w$ inside each of them, is defined in the way that $w$ is diveregence-free inside each region and $w\cdot n$ is continuous across the boundary between any pair of neighbouring regions, where $n$ denotes the unit normal vector of the boundary. Recall (from a recipe for a structure, Section \ref{sec_recipe_for_structure}) that this is sufficient for $w$ to be weakly divergence-free on $\RR^2$.
\begin{figure}[h]
\centering
 \includegraphics[width=\textwidth]{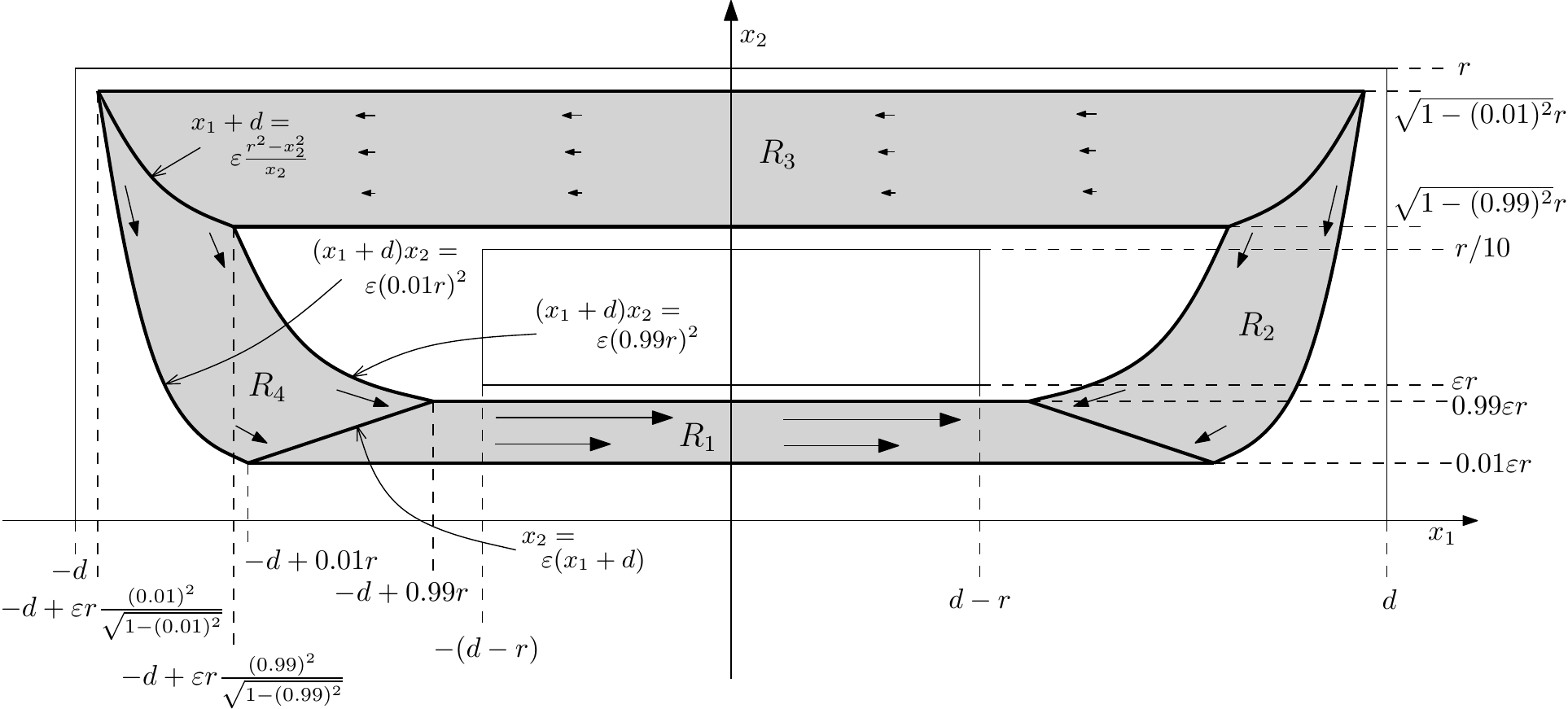}
 \nopagebreak
 \captionsetup{width=0.9\textwidth} 
  \captionof{figure}{The construction of $w$. The configuration of the curves in $\{ x_1 >0 \}$ is defined symmetrically with respect to the $x_2$ axis. The arrows (inside the grey region) indicate the direction and magnitude of $w$. Note that some proportions are not conserved on this sketch.}\label{construction_of_w} 
\end{figure}
Therefore (as in the recipe for a structure, see Section \ref{sec_recipe_for_structure}) $J w $ is divergence free, smooth and compactly supported vector field on $P$, where $J$ denotes any mollification operator. Thus letting 
\[
v_2 = Jw /x_2
\]
we see that, for sufficiently fine mollification $J$, $v_2$ satisfies all the required properties. In particular $v_2=(1,0)$ in $[-(d-r),d-r] \times [0.02 \varepsilon r , 0.98 \varepsilon r ]$ since affine functions are invariant under mollifications.\end{proof}
Now let
\eqnb\label{def_of_tau_z}
\tau \coloneqq 0.48 \varepsilon ,\qquad z\coloneqq (A,\varepsilon r/2,0).
\eqne
We see that 
\eqnb\label{tau_d_less_than}
\tau d = \tau \kappa E /\varepsilon < \kappa E.
\eqne
Let 
\eqnb\label{defs_subsets_of_P}
\begin{split}
U_2 &\coloneqq (-d,d)\times (0.005 \varepsilon r , r) \setminus [-(d-r),d-r]\times [\varepsilon r , r/10 ], \\
\mbox{\emph{BOX}} &\coloneqq [-d,d] \times [0,r], \\
\mbox{\emph{SBOX}} &\coloneqq [A-\kappa E, A+ \kappa E] \times [0.02\varepsilon r , 0.98 \varepsilon r] ,\\
\mbox{\emph{RECT}} &\coloneqq [-(d-r),d-r] \times [0.02 \varepsilon r , 0.98 \varepsilon r ] ,
\end{split}
\eqne
see Fig. \ref{subsets_of_p}.\\
\begin{figure}[h]
\centering
 \includegraphics[width=\textwidth]{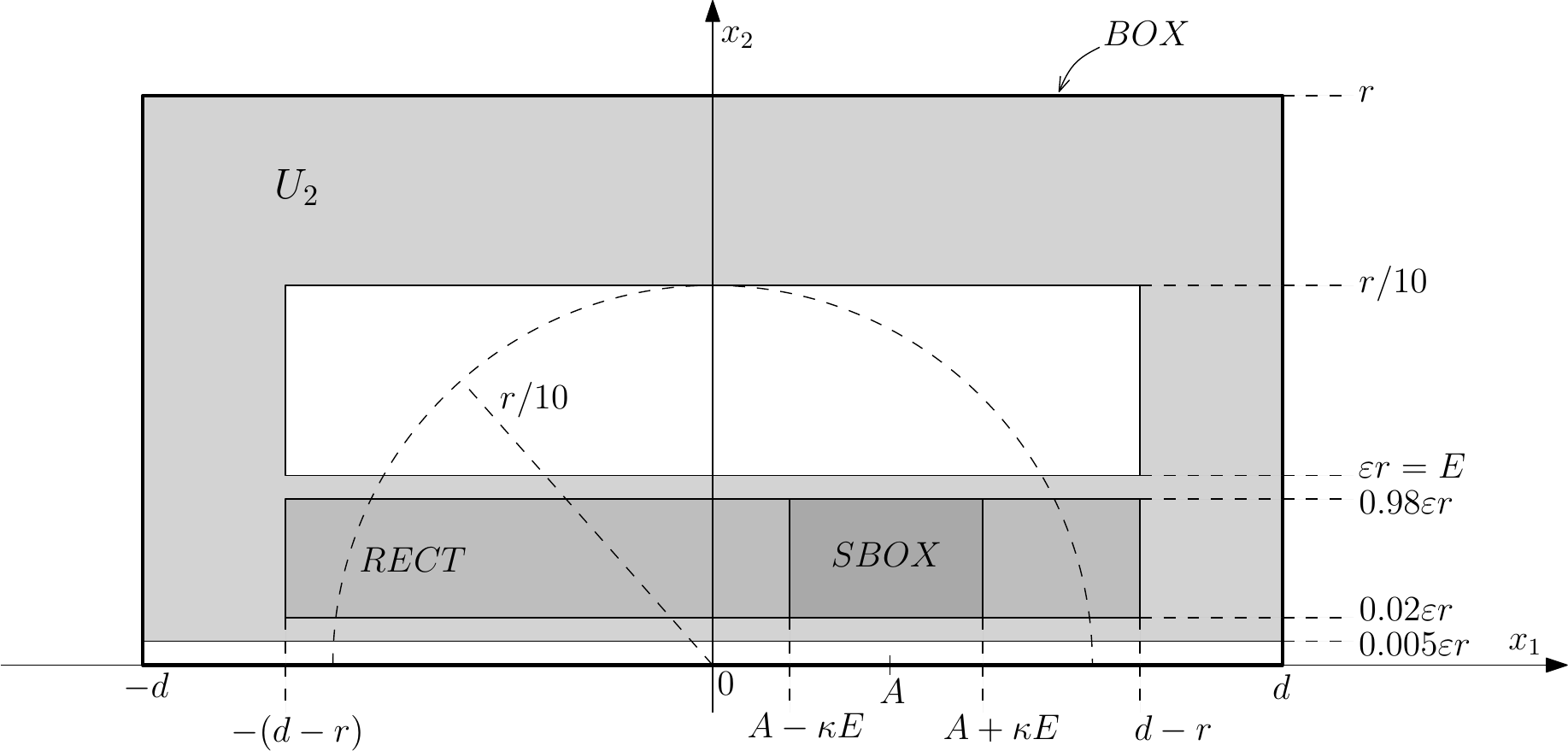}
 \nopagebreak
 \captionsetup{width=0.9\textwidth} 
  \captionof{figure}{The sets $U_2$, $\mbox{\emph{BOX}}$, $\mbox{\emph{RECT}}$ and $\mbox{\emph{SBOX}}$. Note that proportions are not conserved on this sketch.}\label{subsets_of_p} 
\end{figure}

Note that $\supp\, v_2 \subset U_2$ by construction (see Lemma \ref{lemma_existence_of_v2} (ii)) and that $\mbox{\emph{SBOX}} \subset \mbox{\emph{RECT}}$ by the second inequality in \eqref{how_small_is_eps}. Moreover, 
\eqnb\label{box_fits_into_sbox}
R^{-1} ( \tau R (\mbox{\emph{BOX}} ) + z) \subset \mbox{\emph{SBOX}} .
\eqne
Indeed, since $\tau r < \varepsilon r/2$ we observe that the set on the left-hand side is simply 
\[ [A-\tau d,A+\tau d] \times [\varepsilon r/2 - \tau r , \varepsilon r/2 + \tau r ] =[A-\tau d,A+\tau d] \times [0.02 \varepsilon r , 0.98 \varepsilon r ] \subset \mbox{\emph{SBOX}} ,
\]
where the inclusion follows from \eqref{tau_d_less_than}. What is more, the sets $U$, $U^{a',r'}$, $U^{a'',r''}$ are ``encompassed'' by $U_2$, that is
\eqnb\label{three_copies_are_encompassed}
U \cup U^{a',r'}\cup U^{a'',r''} \subset (-(d-r),d-r)\times (\varepsilon r, r/10),
\eqne
see Fig. \ref{the_arrangement}. This property is clear from the identity $\varepsilon r=E$ and property (iv) of the choice of $E$ (so that the strip $\{ 0<x_2<\varepsilon r \}$ is ``below'' these sets), the third inequality in \eqref{how_small_is_eps} (so that the half-plane $\{x_2>r/10\}$ is above $U$), and the fifth inequality in \eqref{how_small_is_eps}, which gives 
\[d-r>\mathrm{diam}\left( U\cup U^{a',r'} \cup U^{a'',r'' }\right) \]
(so that the length in the $x_1$ direction of the set $U \cup U^{a',r'}\cup U^{a'',r''}$ is less than $d-r$; recall also $d>2r$ by \eqref{kappa_is_large}). Furthermore, properties (v) and (vi) of $H$ (and the trivial inequality $0.98\varepsilon r \leq E$) immediately give that
\eqnb\label{H_in_RECT}
\begin{cases}
H_1 (x ) \geq -1.01 B \quad &\text{ in } (-(d-r),d-r)\times (0,\varepsilon r)\supset \mbox{\emph{RECT}},\\
 H_1 (x) \geq 6.99 B &\text{ in } \mbox{\emph{SBOX}}.
 \end{cases}
\eqne

\subsection{Construction of $U_1$ and its structure}\label{sec_constr_U1_and_structure}
We will add one more copy of $U$ (and its structure) to the collection $U$, $U^{a',r'}$, $U^{a'',r''}$ (and the corresponding collection of structures). Namely let 
\eqnb\label{def_of_as}
a\coloneqq -\kappa r /\varepsilon, \quad \frac{s^2}{r} \coloneqq 1.04 \left(- \frac{a}{r} \right)^4 B/D ,
\eqne
and consider $U^{a,r}$ with structure $(v^{a,r,s},f^{a,r,s},\phi^{a,r})$. In this way, the pressure interaction function 
\[
F^{a,r,s} = F[v^{a,r,s},f^{a,r,s}]
\]
is of particular size in the whole of $\mbox{\emph{BOX}}$, which we make precise in the following lemma.
\begin{lemma}\label{lemma_Fars}
\[
1.03 B \leq F_1^{a,r,s} \leq 1.05B \quad \text{ and }\quad  |F_2^{a,r,s} | \leq 0.01 \varepsilon B \quad \text{ in } \mbox{\emph{BOX}}.
\]
\end{lemma}
\begin{proof}
As for $F_1^{a,r,s}$ let $n\coloneqq -a/r = \kappa /\varepsilon $ and observe that the sixth inequality in \eqref{how_small_is_eps} gives $n\geq N$. Thus, since $|x_2 |/r \leq 1$ and
\[\left| \frac{x_1 - a}{r} -n\right| = \frac{|x_1|}{r}\leq \frac{d}{r}=\kappa\] Lemma \ref{lem_properties_of_F} (v) gives
\[
 \left| F_1  \left( \frac{x_1-a}{r},\frac{x_2}{r} \right) - n^{-4} D \right| \leq 0.001 n^{-4} D .
\]
(Recall (from the paragraph preceeding Section \ref{sec_simplified_geom_arrangement}) that $F=(F_1,F_2)$ denotes the pressure interaction function corresponding to $U$ and structure $(v,f,\phi )$, that is $F=F[v,f]$.) Therefore, since
\[
F_1 \left( \frac{x_1-a}{r},\frac{x_2}{r} \right) = \frac{r}{s^2} F_1^{a,r,s} (x_1,x_2)  = \frac{n^{-4}D}{1.04 B} F_1^{a,r,s} (x_1,x_2)  
\]
(recall \eqref{defs_of_copying} and \eqref{def_of_as}), we can multiply the last inequality by $1.04B/(n^{-4}D)$ to obtain
\[
\left|  F_1^{a,r,s} (x) - 1.04B \right| \leq 0.001 (1.04 B) < 0.01B \quad \text{ for } x\in \mbox{\emph{BOX}} .
\]
As for $F_2^{a,r,s}$ let $(x_1,x_2) \in \mbox{\emph{BOX}}$ and use Lemma \ref{lem_properties_of_F} (iv), the Mean Value Theorem and Lemma \ref{lem_properties_of_F} (iii) to write
\[
\begin{split}
\frac{r}{s^2}\left| F_2^{a,r,s} (x_1,x_2) \right| &= \left| F_2 \left( \frac{x_1-a}{r},\frac{x_2}{r} \right) - F_2 \left( \frac{x_1-a}{r},0 \right) \right| \\
&\leq \left| \nabla F_2 \left( \frac{x_1-a}{r},\xi \right) \right| \, \left| \frac{x_2}{r} \right| \leq C \left| \frac{x_1-a}{r} \right|^{-5},
\end{split}
\]
where $\xi \in (0,1)$. Thus, since the triangle inequality and the fact $\varepsilon <1/2$ give 
\[
 \frac{|x_1-a|}{r} \geq \frac{|a|}{r} - \frac{|x_1|}{r} \geq \frac{|a|}{r} - \frac{d}{r} = \kappa \left( \frac{1}{\varepsilon } -1 \right) \geq \frac{\kappa }{2\varepsilon } ,
\]
we obtain (recalling \eqref{def_of_as} and that $\kappa=10^4 C/D$, see \eqref{kappa_is_large})
\[
\left| F_2^{a,r,s} (x_1,x_2) \right| \leq  \left( 2^5 \frac{1.04 C}{D\kappa }\right) \varepsilon B =  \frac{32\cdot 1.04}{10^4} \varepsilon B < 0.01 \varepsilon B .\qedhere
\]
\end{proof}
Thus letting 
\[
\begin{split}
U_1 &\coloneqq U\cup U^{a',r'} \cup U^{a'',r''} \cup U^{a,r} ,\\
f_1 &\coloneqq f + f^{a',r',s'}+ f^{a'',r'',s''}+ f^{a,r,s},\\
v_1 &\coloneqq v + v^{a',r',s'}+ v^{a'',r'',s''}+ v^{a,r,s},\\
\phi_1 &\coloneqq \phi + \phi^{a',r'}+ \phi^{a'',r''}+ \phi^{a,r}\\
\end{split}
\]
we obtain a structure $(v_1,f_1,\phi_1)$ on $U_1$, and denoting by $F^*$ the total pressure interaction function,
\[
F^* \coloneqq F[v_1,f_1]= F + F^{a',r',s'}+ F^{a'',r'',s''}+ F^{a,r,s} = H+F^{a,r,s} ,
\]
we see that the above lemma and \eqref{H_in_RECT} give 
\eqnb\label{Fstar_in_sets}
\begin{cases}
F^*_1 \geq 0.01 B \quad & \text{ in } (-(d-r),d-r)\times (0,\varepsilon r) \supset \mbox{\emph{RECT}}, \\
F^*_1 \geq 8B \quad &\text{ in } \mbox{\emph{SBOX}} .
\end{cases}
\eqne
Moreover, the properties of $H$ (the ``joint'' pressure interaction function of $U$, $U^{a',r'}$ and $U^{a'',r''}$, recall \eqref{def_of_H}), $v_2$, the smallness of $\varepsilon $ (recall \eqref{how_small_is_eps}) and the lemma above give 
\eqnb\label{v_times_F_loweR_bound}
v_2 \cdot F^* \geq -1.1 \varepsilon B \qquad \text{ in } \mbox{\emph{BOX}},
\eqne
which we now verify. The claim for $x\in \mbox{\emph{BOX}} \setminus \supp \, v_2 $ follows trivially. For $x\in \supp \, v_2$ consider two cases.\vspace{0.3cm}\\
\emph{Case 1.} $|x|< r/10$. In this case observe that since $d\geq 10^4 r$ (see \eqref{kappa_is_large}) we have $d-r>r/10$ and so $x\in (-(d-r),d-r)\times (0,\varepsilon r)$ (cf. Fig. \ref{subsets_of_p}). Thus, $v_{21}(x)\geq 0$, $v_{22}(x)=0$ by construction of $v_2$ (see Lemma \ref{lemma_existence_of_v2} (iii)), and so \eqref{Fstar_in_sets} gives
\[
v_2 (x) \cdot F^*(x) = v_{21} (x) F_1^* (x) \geq 0.01B v_{21} (x) \geq 0 > -1.1\varepsilon B.
\]
\emph{Case 2.} $|x| \geq r/10$. Since $r>20|A|$ (see the fourth inequality in \eqref{how_small_is_eps}), in this case $|x| \geq 2|A|$, and so property (iii) of $H$ and the last inequality in \eqref{how_small_is_eps} give
\eqnb\label{the_use_of_decay_of_H}
|H(x)| \leq 2C/|x|^4 \leq 2C \left( \frac{10}{r} \right)^4 = 2\cdot 10^4 C \varepsilon^4 / E^4 < 0.01 \varepsilon^2 B .
\eqne
This, the properties $-\varepsilon^2 \leq v_{21} \leq 1$, $|v_{22}|<\varepsilon/2$ (see Lemma \ref{lemma_existence_of_v2} (iii)) and Lemma \ref{lemma_Fars} give
\[
\begin{split}
v_2(x) \cdot F^* (x) &= v_2 (x) \cdot H(x) + v_{21} (x) F_1^{a,r,s} (x) + v_{22} (x) F_2^{a,r,s} (x) \\
&\geq - 2 (0.01 \varepsilon^2 B) - \varepsilon^2 (1.05 B) - \frac{\varepsilon  }{2} (0.01 B\varepsilon ) \\
&= - \varepsilon^2 B (0.02 +1.05+0.005) \geq -1.1 \varepsilon^2 B.
\end{split}
\]
Thus we obtain \eqref{v_times_F_loweR_bound}, as required.\\

Moreover, since $U^{a,r}=(a-r,a+r)\times (r/8,7r/8)$, we see that 
\eqnb\label{Uar_is_to_the_left_of_box}
U^{a,r}\text{ is located ``to the left'' of }\mbox{\emph{BOX}},
\eqne
that is $a+r<-d$ (see Fig. \ref{the_arrangement}), which can be verified as follows. Since $\varepsilon <1/10$ (recall \eqref{how_small_is_eps}) and $\kappa >1$ (recall \eqref{kappa_is_large}) we trivially obtain
\[ \kappa \left( \frac{1}{\varepsilon } -1 \right) >1 ,
\]
which, multiplied by $r$, gives
\[
-\frac{\kappa r}{\varepsilon } + r < -\kappa r,
\]
that is $a+r<-d$, as required. 
Thus, taking into account \eqref{three_copies_are_encompassed} we see that $\overline{U_1}$ and $\overline{U_2}$ are disjoint (see Fig. \ref{the_arrangement}), which is one of the requirements of the {geometric arrangement}. 
\begin{figure}[h]
\centering
 \includegraphics[width=\textwidth]{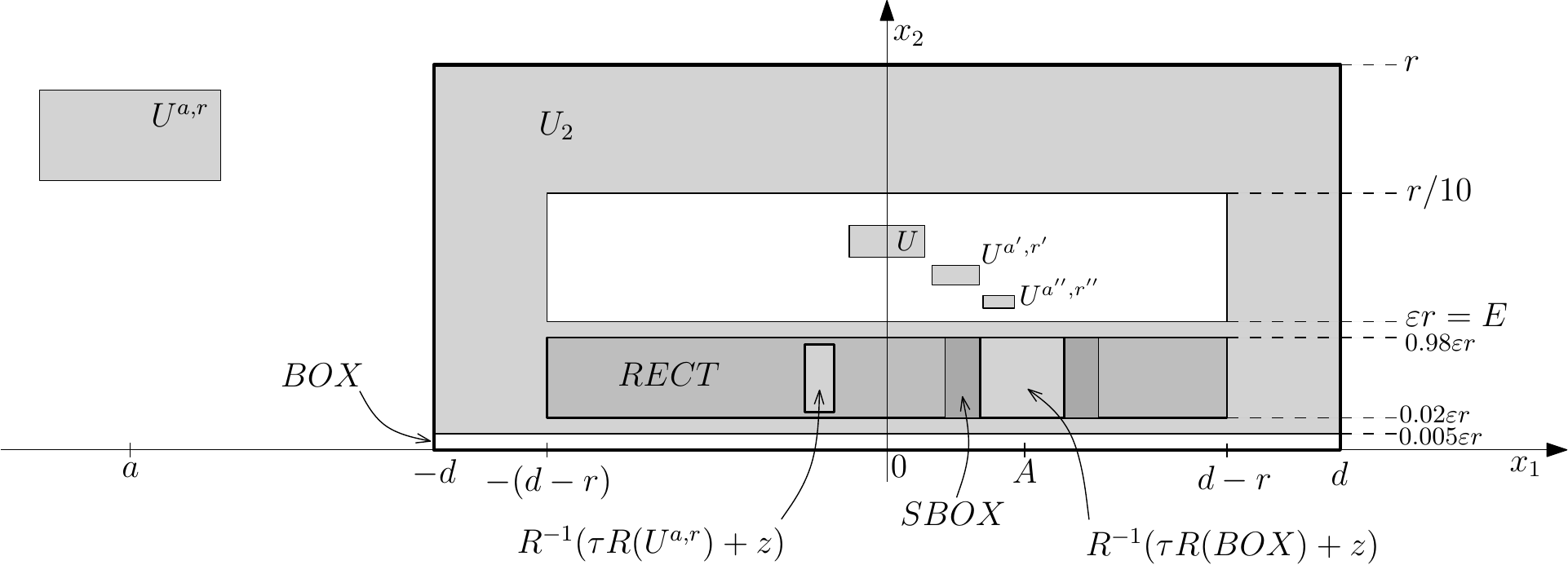}
 \nopagebreak
 \captionsetup{width=0.9\textwidth} 
  \captionof{figure}{The geometric arrangement (cf. Fig. \ref{the_arrangement_sketchy_sketch}).}\label{the_arrangement} 
\end{figure}

Furthermore note that
\eqnb\label{Uar_fits_into_rect}
R^{-1} \left( \tau R \left( \overline{U^{a,r} } \right) + z \right) \subset \mbox{\emph{RECT}} ,
\eqne
see Fig. \ref{the_arrangement}. Indeed, since $U^{a,r}= (a-r,a+r)\times (r/8,7r/8)$ (recall \eqref{defs_of_copying}) we see that the set on the left-hand side is simply
\[
[\tau(a-r)+A,\tau (a+r)+A] \times [\varepsilon r/2 -7\tau r /8, \varepsilon r/2 + 7 \tau r/8 ] ,
\]
where we recalled that $z=(A,\varepsilon r/2 ,0)$ (see \eqref{def_of_tau_z}).
The second of these intervals is contained in $[\varepsilon r/2 - \tau r , \varepsilon r/2 +\tau r]=[0.02\varepsilon r , 0.98 \varepsilon r]$, where we recalled that $\tau=0.48\varepsilon$ (see \eqref{def_of_tau_z}). Thus \eqref{Uar_fits_into_rect} follows if the first of the intervals is contained in $[-(d-r),d-r]$, that is if
\[
|\tau a +A | \leq d-r-\tau r.
\]
This last inequality follows from the fourth inequality in \eqref{how_small_is_eps} and the facts that $\kappa \geq 10^4$ (recall \eqref{kappa_is_large}) and $\tau <1$, by writing
\[
|\tau a + A | \leq \tau |a| + |A| = 0.48 \kappa r + |A| \leq 0.48 \kappa r + 0.05 r < \kappa r /2 < (\kappa -2)r = d-2r < d-r-\tau r .
\]
The inclusions \eqref{box_fits_into_sbox} and \eqref{Uar_fits_into_rect} combine to give 
\[
R^{-1} (\Gamma (G)) \subset \mbox{\emph{RECT}} \subset \overline{U_2},
\]
and thus 
\[
\Gamma (G) \subset R(\overline{U_2}) \subset G
\]
(recall $G= R(\overline{U_1} \cup \overline{U_2})$ and $\Gamma (x) = \tau x +z$), as required by the geometric arrangement, which is sketched in Fig. \ref{the_arrangement}.

\subsection{Construction of $f_2$, $\phi_2$, $T$ and conclusion of the arrangement}\label{sec_constr_f2_and_rest}
It remains to construct $f_2$, $\phi_2$, $T$ such that $(v_2,f_2,\phi_2 )$ is a structure on $U_2$ and properties \eqref{fairies_extra_ineq}, \eqref{fairies_scaling} hold, that is
\eqnb\label{fairies_extra_ineq_rewritten}
f_2^2 + T v_2 \cdot F^*  > |v_2|^2 \qquad \text{ in } U_2 ,
\eqne
and
\eqnb\label{fairies_scaling_rewritten}
f_2^2(y) +T v_2(y) \cdot F^* (y) > \tau^{-2} \left(  f_1 (R^{-1}x) +  f_2(R^{-1}x) \right)^2 
\eqne
for 
\[
y =  R^{-1} (\tau x + z),  \quad x\in R\left( \overline{U_1} \cup \overline{U_2} \right) ,
\]
respecitvely.\\

To this end, note that since $U_2$ is a rectangular ring, we can (as in Section \ref{sec_recipe_for_structure}) use Theorem \ref{thm_existence_of_f_ClaimX} to obtain $f_2\in C_0^\infty (P; [0,1])$ such that $\supp \, f_2 =\overline{U_2}$, $f_2>0$ in $U_2$, $f_2=\mu $ on $\supp \, v_2$, $Lf_2 >0$ at points of $U_2$ of sufficiently small distance to $\partial U_2$. Here we choose $\mu >100$  sufficiently large such that
\eqnb\label{mu_is_large}
\mu \geq 100 \| f_1 \|_\infty .
\eqne
Following the recipe for a structure (Section \ref{sec_recipe_for_structure}) we let $\phi_2 \in C_0^\infty (U_2;[0,1])$ be a cut off function such that $\phi_2 =1$ in $\supp \, v_2$ and $Lf_2 >0$ in $U_2 \setminus \{ \phi_2 =1 \}$. Thus $(v_2,f_2,\phi_2)$ is a structure on $U_2$. We now let 
\eqnb\label{def_of_T}
T\coloneqq \frac{\mu^2 - 5}{1.1 \varepsilon^2 B } \geq \frac{0.9 \mu^2 }{1.1\varepsilon^2 B } 
\eqne
and we verify \eqref{fairies_extra_ineq_rewritten} and \eqref{fairies_scaling_rewritten}. \\

Using \eqref{v_times_F_loweR_bound} and the fact that $|v_2 |\leq 2$ (recall Lemma \ref{lemma_existence_of_v2} (iii)) we immediately obtain \eqref{fairies_extra_ineq_rewritten} by writing
\[
f_2^2 + T v_2 \cdot F^* \geq \mu^2 - 1.1 \varepsilon^2 BT = 5 > |v_2|^2 \qquad \text{ in } \supp \, v_2 ,
\]
and the claim in $U_2\setminus \supp\, v_2$ follows trivially from positivity of $f_2$ in $U_2$. \\

As for \eqref{fairies_scaling_rewritten}, we need to show
\[
f_2^2(y) +T v_2(y) \cdot F^* (y) > \tau^{-2} \left(  f_1 (R^{-1}x) +  f_2(R^{-1}x) \right)^2 
\]
for 
\[
y =  R^{-1} (\tau x + z),  \quad x\in R\left( \overline{U_1} \cup \overline{U_2} \right) .
\]
To this end fix $x\in R\left( \overline{U_1} \cup \overline{U_2} \right)$. Since
\[
\overline{U_1} \cup \overline{U_2} = \left( \overline{U_2} \cup \overline{U}\cup \overline{U^{a',r'}}\cup \overline{U^{a'',r''}} \right) \cup   \overline{U^{a,r}}
\] 
we consider two cases.\vspace{0.3cm}\\
\emph{Case 1.} $x\in R \left( \overline{U_2} \cup \overline{U}\cup \overline{U^{a',r'}}\cup \overline{U^{a'',r''}} \right)$. \\Then $R^{-1}x \in \mbox{\emph{BOX}}$ and hence $y\in \mbox{\emph{SBOX}}$ by \eqref{box_fits_into_sbox}.
Thus $v_2(y)=(1,0)$ and $F_1^*(y) \geq 8B$ in $\mbox{\emph{SBOX}}$ (see Lemma \ref{lemma_existence_of_v2} (iv) and \eqref{Fstar_in_sets}) and, using \eqref{def_of_T} and \eqref{mu_is_large},
\[
\begin{split}
f_2^2(y) +T v_2(y) \cdot F^* (y) & \geq 8TB \geq \frac{7.2}{1.1} \left( \frac{\mu }{\varepsilon } \right)^2  > \left( \frac{1.01}{0.48} \right)^2 \left( \frac{\mu }{\varepsilon } \right)^2  \\
&=\tau^{-2} (1.01\mu)^2 \geq \tau^{-2} \left( \| f_2 \|_\infty +  \| f_1 \|_\infty \right)^2 .
\end{split}
\]
\emph{Case 2.} $x\in R \left( \overline{ U^{a,r} } \right)$. Then $f_2(R^{-1} x)=0$ (since $R^{-1} x \in \overline{U^{a,r}}\subset \overline{U_1}$ and $\overline{U_1}$, $\overline{U_2}$ are disjoint) and $y\in \mbox{\emph{RECT}}$ (see \eqref{Uar_fits_into_rect}). Therefore, $v_2(y)=(1,0)$, $F_1^* (y) \geq 0.01 B$ (by Lemma \ref{lemma_existence_of_v2} (iv) and \eqref{Fstar_in_sets}) and so, using \eqref{def_of_T} and \eqref{mu_is_large},
\[
\begin{split}
f_2^2(y) +T v_2(y) \cdot F^* (y) &\geq 0.01TB \geq \frac{0.009}{1.1} \left( \frac{\mu }{\varepsilon } \right)^2  > \left( \frac{0.01}{0.48} \right)^2 \left( \frac{\mu }{\varepsilon } \right)^2 \\
&= \tau^{-2} (0.01 \mu )^2 \geq \tau^{-2} \| f_1 \|_\infty^2 \geq \tau^{-2} \left(  f_1 (R^{-1}x) +  f_2(R^{-1}x) \right)^2 .
\end{split}
\]
Hence we obtain \eqref{fairies_scaling_rewritten}. This concludes the construction of geometric arrangement, and so also the proof of Theorem \ref{point_blowup_thm}.

\subsection[The norm inflation result]{The norm inflation result, Theorem \ref{mainthm_large_gain}}\label{sec_norm_inflation}
Here we prove Theorem \ref{mainthm_large_gain}, which is a corollary of Theorem \ref{point_blowup_thm}.

Namely, given $\mathcal{N}>0$ and $\vartheta >0$ we will construct $T>0$, $\nu_0>0$, $\eta>0$ and divergence-free $u \in C^\infty (\RR^3 \times (-\eta ,T+\eta );\RR^3)$ such that
\eqnb\label{NSI_with_almost_equality}
-\frac{2\vartheta}{T^2 \nu_0 } \leq \p_t |u|^2 -2\nu \, u\cdot \Delta u+ u\cdot \nabla  (|u|^2 +2 p ) \leq 0\qquad \text{ in } \RR^3 \times [0,T]
\eqne
for all $\nu \in [0,\nu_0]$, $\mathrm{supp}\, u(t) \subset G$ for all $t\in [0,T]$ (where $G\subset \RR^3$ is compact), and
\eqnb\label{arb_large_gain_nts}
\| u(T) \|_{L^\infty } \geq \mathcal{N} \| u(0) \|_{L^\infty }
\eqne
Theorem \ref{mainthm_large_gain} (which corresponds to the case $T=\nu_0=1$) then follows by a simple rescaling; namely 
\[
\mathfrak{u} (x,t) \coloneqq \sqrt{ \nu_0 T} u\left( \sqrt{{T}/{\nu_0}} x, T t  \right)
\]
satisfies the claim of Theorem \ref{mainthm_large_gain}.\\ 

Let $T>0$, $\nu_0\in (0,1)$, $\eta >0$ and $u$ be as in the proof of Theorem \ref{point_blowup_thm} (that is recall Proposition \ref{prop_existence_of_u}, \eqref{how_small_is_nu0}, \eqref{def_of_T}). Note that this already gives that $\mathrm{div}\,u(t)=0$, $\mathrm{supp}\,u(t) \subset G$ for all $t\in [0,T]$, smoothness of $u$ and the right-most inequality in \eqref{NSI_with_almost_equality}. We now verify that such a choice satisfies the remaining two claims (i.e. the left-most inequality in \eqref{NSI_with_almost_equality} and \eqref{arb_large_gain_nts}) given we take $\varepsilon$ (the ``sharpness'' of the geometric arrangement, recall \eqref{how_small_is_eps}) and $\delta $ (a part of the definition of $h$, recall Lemma \ref{prop_of_h1h2}) sufficiently small.

Since
\[
\| u(0) \|_{\infty } = \| h_0 \|_\infty = \| f_1 +f_2 \|_\infty  ,
\]
and, from Proposition \ref{prop_existence_of_u} (ii) and \eqref{3.14},
\[
\|u(T) \|_\infty^2  \geq \| h_T \|_\infty^2 - \theta \geq \tau^{-2} \| f_1 + f_2 \|_\infty^2,
\]
we obtain 
\[
\| u(T) \|_\infty \geq \mathcal{N} \| u(0) \|_\infty 
\]
given $\tau^{-1} \geq \mathcal{N}$; that is provided $\varepsilon >0$ is sufficiently small such that 
\[\varepsilon \leq (0.48 \mathcal{N} )^{-1},
\]
in addition to the smallness requirements of the geometric arrangement \eqref{how_small_is_eps}.
Note that making the value of $\varepsilon$ smaller we also make $T$ larger.  

In order to obtain the left-most inequality in \eqref{NSI_with_almost_equality} we perform similar calculation as in cases 1 \& 2 in Section \ref{sec_pf_of_claims_of_thm} given $\delta >0$ is small as in Lemma \ref{prop_of_h1h2} and additionally
\[\delta < \vartheta / 2 T^2  .\]
Indeed, since \eqref{conv_2} gives
\[
2\nu_0 \left| u(x,0,t) \cdot \Delta u(x,0,t) \right|  \leq \delta 
\]
and since $\nu_0 <1$ (recall \eqref{how_small_is_nu0}) we write in the case $\phi_1 (x) +\phi_2(x)<1$ (that is Case 1 in Section \ref{sec_pf_of_claims_of_thm})
\[
\begin{split}
\p_t |u(x,0,t)|^2 &= \p_t q_{1,t}^k (x)^2 + \p_t q_{2,t}^k  (x)^2 \\
&= -2\delta (\phi_1 (x) + \phi_2 (x) )\\
&>-2\delta \\
&\geq - 2\vartheta / T^2 \nu_0 - u(x,0,t)\cdot \nabla \left( |u(x,0,t)|^2 +2 p (x,0,t) \right) + 2 \nu \,u(x,0,t) \cdot \Delta u(x,0,t) 
\end{split}
\]
for all $\nu\in [0,\nu_0]$, $t\in [0,T]$, where we also used \eqref{prop_of_structure_2}. In the case $\phi_1(x) + \phi_2 (x) =1$ (that is Case 2 in Section \ref{sec_pf_of_claims_of_thm}) we use \eqref{conv_1} (in the same way as before) to obtain
\[
\begin{split}
\p_t |u(x,0,t)|^2 &= \p_t q_{1,t}^k (x)^2 + \p_t q_{2,t}^k  (x)^2   \\
&= -2\delta - \left( a_1^k (t) v_1(x) +a_2^k (t)v_2 (x) \right) \cdot \nabla \left( h_{1,t} (x)^2 + h_{2,t} (x)^2  \right. \\
&\hspace{5cm}+ \left. 2 p[a_1^k(t) v_1,h_{1,t}](x) +2 p[a_2^k(t) v_2,h_{2,t}](x) \right) \\
&\geq -3\delta - \left( a_1^k (t) v_1(x) +a_2^k (t)v_2 (x) \right) \cdot \nabla \left( q_{1,t}^k (x)^2 + q_{2,t}^k (x)^2  \right. \\
&\hspace{5cm}+ \left. 2 p[a_1^k(t) v_1,q_{1,t}^k](x) +2 p[a_2^k(t) v_2,q_{2,t}^k](x) \right) \\
&= -3\delta -  u (x,0,t) \cdot \nabla \left( |u(x,0,t)|^2 +2 p (x,0,t) \right)\\
&\geq -2\vartheta / T^2 \nu_0 -   u (x,0,t) \cdot \nabla \left( |u(x,0,t)|^2 +2 p (x,0,t) \right) + 2\nu  \,u(x,0,t) \cdot \Delta u(x,0,t)   
\end{split}
\]
for all $\nu \in [0,\nu_0 ]$, $t\in [0,T]$, where (as before) we also used \eqref{d3_of_|u|_vanishes} and \eqref{dx3_of_p_is_zero} in the fourth step.

\section[Blow-up on a Cantor set]{Weak solution to the Navier--Stokes inequality with a blow-up on a Cantor set}\label{sec_pf_of_thm2}
In this section we construct Scheffer's counterexample (Theorem \ref{1D_blowup_thm_intro}), that is a weak solution of the Navier--Stokes inequality that blows up on a Cantor set $S\times \{ T_0 \}$ with $d_H (S)\geq \xi$ for any preassigned $\xi \in (0,1)$. Such a solution was first constructed by \cite{scheffer_nearly}, and we now restate the result (Theorem \ref{1D_blowup_thm_intro}) for the reader's convenience. 
\begin{theorem}[Nearly one-dimensional singular set]\label{1D_blowup_thm}
Given any $\xi \in (0,1)$ there exists $\nu_0>0$, a compact set $G\Subset \RR^3$ and a function $\mathfrak{u}\colon \RR^3 \times [0,\infty ) \to \RR^3$ that is a weak solution to the Navier--Stokes inequality for any $\nu \in [0,\nu_0 ]$ such that $u(t) \in C^\infty$, $\supp \,\mathfrak{u}(t)\subset G$ for all $t$, and 
\[
\xi \leq d_H (S) \leq 1 ,
\]
where
\[
S\coloneqq \{ (x,t) \in \RR^3 \times (0,\infty ) \, : \, \mathfrak{u}(x,t) \text{ is unbounded in any neighbourhood of } (x,t) \}.
\]
\end{theorem}

Recall that the difference between this theorem and the result of the previous sections (Theorem \ref{point_blowup_thm}) is the size of the singular set. In the case of Theorem \ref{point_blowup_thm} the singular set is a point $\{ x_0 \} \times \{ T_0 \}$ and in the case of Theorem \ref{1D_blowup_thm} it is a set $S\times \{ T_0 \}$, where $S$ is a Cantor set with $d_H(S) \in [\xi ,1]$ for given $\xi \in (0,1)$. We will show how Theorem \ref{1D_blowup_thm} can be obtained by sharpening the proof of Theorem \ref{point_blowup_thm} (which we shall refer to by writing ``previously'') as intuitively sketched on Fig. \ref{supp_sketch_difference}. 
\begin{figure}[h]
\centering
 \includegraphics[width=0.9\textwidth]{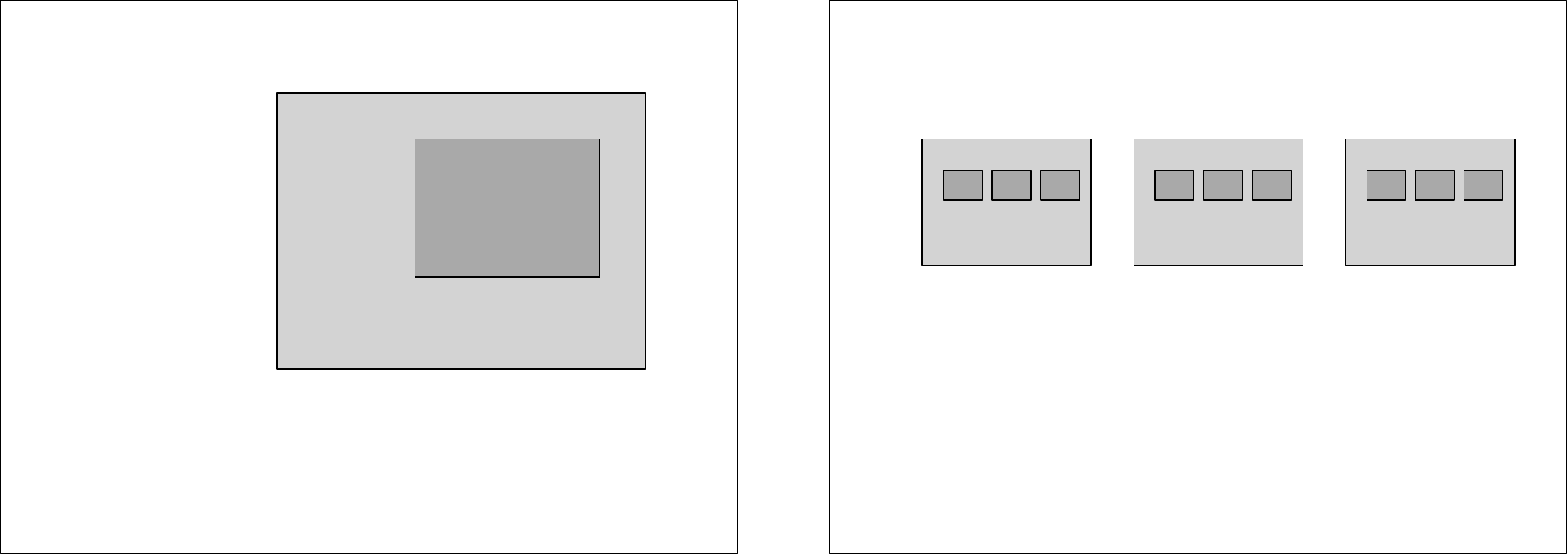}
 \nopagebreak
 \captionsetup{width=0.9\textwidth} 
  \captionof{figure}{The (intuitive) difference in constructing solutions to Theorem \ref{point_blowup_thm} (left; cf. Fig. \ref{switching_process_figure_with_supp}) and Theorem \ref{1D_blowup_thm} (right).}\label{supp_sketch_difference} 
\end{figure}

In other words, the solution $\mathfrak{u}$ (of Theorem \ref{1D_blowup_thm}) is obtained by a similar switching procedure as in Section \ref{sec_proof_of_thm1}, except that at every switching the support of $\mathfrak{u}$ shrinks (by a fixed factor) to form $M$ copies of itself ($M\in \NN$) and thus form a Cantor set $S$ at the limit $t\to T_0^-$. It is remarkable that such approach allows enough freedom to ensure that $d_H(S)$ is arbitrarily close to $1$ (from below). Before proceeding to the proof we briefly comment on the construction of such a Cantor set and we introduce some handy notation. We then prove Theorem \ref{1D_blowup_thm} in Section \ref{sec_sketch_of_pf_of_thm2}. 

\subsection{Constructing a Cantor set}\label{sec_CANTOR_set}

The problem of constructing Cantor sets is usually demonstrated in a one-dimensional setting using intervals, as in the following proposition. 

\begin{proposition}
Let $I\subset \RR$ be an interval and let $\tau \in (0,1)$, $M\in \NN$ be such that $\tau M<1$. Let $C_0\coloneqq I$ and consider the iteration in which in the $j$-th step ($j\geq 1$) the set $C_j$ is obtained by replacing each interval $J$ contained in the set $C_{j-1}$ by~$M$~equidistant copies of $\tau J$ contained in $J$, see for example Fig. \ref{cantor1_new}. Then the limiting object 
\[ C\coloneqq \bigcap_{j\geq 0} C_j \]
is a Cantor set whose Hausdorff dimension equals $- \log M /\log \tau$.
\end{proposition}
See Example 4.5 in \cite{falconer} for a proof. Thus if $\tau\in (0,1)$, $M\in \NN$ satisfy
\[
\tau^\xi M \geq 1 \quad \text{ for some } \xi \in (0,1),
\] 
we obtain a Cantor set $C$ with 
\eqnb\label{haus_dim_of_C_is_at_least_xi}
d_H (C) \geq \xi.
\eqne
Note that both the above inequality and the constraint $\tau M <1$ (which is necessary for the iteration described in the proposition above, see also Fig. \ref{cantor1_new}) can be satisfied only for $\xi <1$.
In the remainder of this section we extend the result from the proposition above to the three-dimensional setting.

Let $G\subset \RR^3$ be a compact set. We will later take $G\coloneqq R(\overline{U}_1 \cup \overline{U}_2)$ (as in the case of Theorem \ref{point_blowup_thm}), and so for convenience suppose further that $G=G_1 \cup G_2$ for some disjoint compact sets $G_1,G_2\subset \RR^3$, and such that $G_2 = R(\overline{U_2})$ for some open and connected $U_2 \Subset P$. Let $\tau \in (0,1)$, $M\in \NN$, $z=(z_1,z_2,0)\in G_2$, $X>0$ be such that
\eqnb\label{CANTOR_tau_xi_M}
\tau^\xi M \geq 1,\qquad \tau M < 1
\eqne
and 
\eqnb\label{CANTOR_gamma_n_maps_G_inside}
\{\Gamma_n (G)\}_{n=1,\ldots , M} \text{ is a family of pairwise disjoint subsets of } G_2,
\eqne
where
\[
\Gamma_n (x) \coloneqq \tau x + z + (n-1) (X,0,0).
\]
Equivalently, 
 \eqnb\label{equiv_def_of_Gamma_n}
 \Gamma_n (x_1,x_2,x_3) = ( \beta_n (x_1), \gamma (x_2), \tau x_3 ),
 \eqne
 where
 \[
 \begin{cases}
 \beta_n (x) \coloneqq \tau x + z_1+(n-1)X ,\\
 \gamma (x) \coloneqq \tau x + z_2,
 \end{cases}\qquad x\in \RR, n=1,\ldots , M .
 \]
 Now for $j\geq 1$ let
\[
M(j) \coloneqq \lewy m=(m_1,\ldots , m_j ) \colon m_1,\ldots , m_j \in \{ 1, \ldots , M \} \prawy 
\] 
denote the set of multi-indices $m$. Note that in particular $M(1)=\{ 1, \ldots , M\}$.  Casually speaking, each multiindex $m\in M(j)$ plays the role of a ``coordinate'' which let us identify any component of the set obtained in the $j$-th step of the construction of the Cantor set. Namely, letting 
 \[
 \pi_m  \coloneqq  \beta_{m_1} \circ \ldots \circ \beta_{m_j},\qquad m\in M(j),
 \]
 that is
 \eqnb\label{def_of_pi_m}
 \pi_m (x) = \tau^j x + z_1 \frac{1-\tau^j}{1-\tau} + X \sum_{k=1}^j \tau^{k-1} (m_k-1), \quad x\in \RR,
 \eqne
 we see that the set $C_j$ obtained in the $j$-th step of the construction of the Cantor set $C$ (from the proposition above) can be expressed simply as
 \[
 C_j \coloneqq \bigcup_{m\in M(j)} \pi_m (I),
 \]
 see Fig. \ref{cantor1_new}. Moreover, each $\pi_m (I)$ can be identified by, roughly speaking, first choosing $m_1$-th subinterval, then $m_2$-th subinterval, ... , up to $m_j$-th interval, where $m=(m_1,\ldots , m_j)$. This is demonstrated in Fig. \ref{cantor1_new} in the case when $m=(1,2)\in M(2)$. 

 \begin{figure}[h]
\centering
 \includegraphics[width=\textwidth]{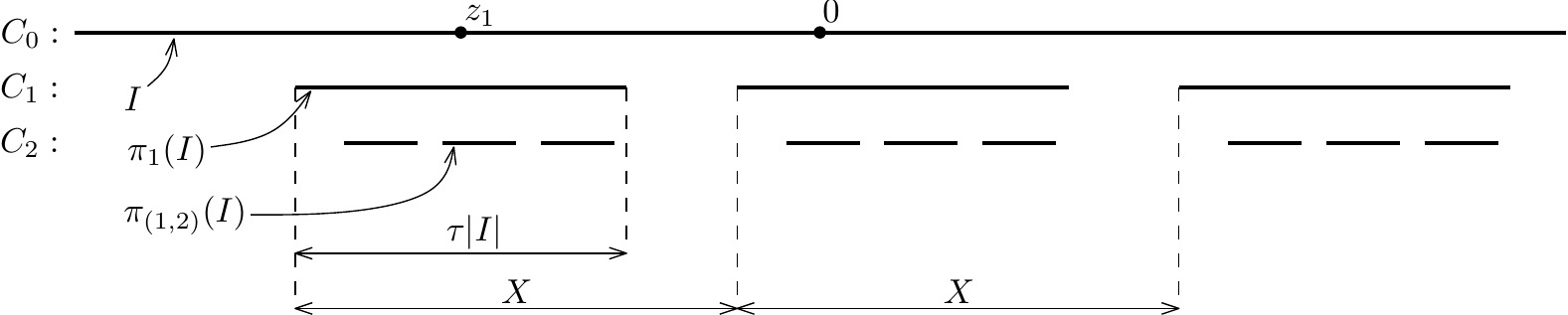}
 \nopagebreak
 \captionsetup{width=0.9\textwidth} 
  \captionof{figure}{A construction of a Cantor set $C$ on a line (here $M=3$, $j=0,1,2$).}\label{cantor1_new} 
\end{figure}
 
 In order to proceed with our construction of a Cantor set in three dimensions let
\[
\Gamma_m (x_1,x_2,x_3) \coloneqq \left( \pi_m (x_1), \gamma^j (x_2) , \tau^j x_3 \right).
\] 
Note that such a definition reduces to \eqref{equiv_def_of_Gamma_n} in the case $j=1$. If $j=0$ then let $M(0)$ consist of only one element $m_0$ and let $\pi_{m_0}\coloneqq\mathrm{id}$. Moreover, if $m\in M(j)$ and $\overline{m}\in M(j-1)$ is its sub-multiindex, that is $\overline{m} = (m_1,\ldots , m_{j-1})$ ($\overline{m}=m_0$ if $j=1$), then \eqref{CANTOR_gamma_n_maps_G_inside} gives
\eqnb\label{CANTOR_Gamma_m_decreases}
\Gamma_m (G)= \Gamma_{\overline{m}} ( \Gamma_{m_j}(G))\subset \Gamma_{\overline{m}} (G_2),
\eqne
which is a three-dimensional equivalent of the relation $\pi_m (I)\subset \pi_{\overline{m}}(I)$ (see Fig. \ref{cantor1_new}).
The above inclusion and \eqref{CANTOR_gamma_n_maps_G_inside} give that
\eqnb\label{different_mutliindices_disjoint_subsets}
\Gamma_m (G) \cap \Gamma_{\widetilde{m}} (G) = \emptyset\qquad \text{ for } m,\widetilde{m}\in M(j), \,j\geq 1, \text{ with }m\ne \widetilde{m}.
\eqne
  Another consequence of \eqref{CANTOR_Gamma_m_decreases} is that the family of sets
 \eqnb\label{family_of_sets_decreases}
 \lewy \bigcup_{m\in M(j)} \Gamma_m (G) \prawy_j\,\, \text{ decreases as }j \text{ increases.}
 \eqne
Moreover, given $j$, each of the sets $\Gamma_m (G)$, $m\in M(j)$, is separated from the rest by at least $\tau^{j-1}\zeta$, where $\zeta >0$ is the distance between $\Gamma_n (G)$ and $\Gamma_{n+1} (G)$, $n=1,\ldots , M-1$ (recall  \eqref{CANTOR_gamma_n_maps_G_inside}).
  
  Taking the intersection in $j$ we obtain 
\eqnb\label{S_intersection}
S\coloneqq  \bigcap_{j\geq 0} \bigcup_{m\in M(j)} \Gamma_m (G),
\eqne
  and we now show that 
  \eqnb\label{cantor_bounds_on_dH_of_S} \xi \leq d_H (S) \leq 1.\eqne Noting that $S$ is a subset of a line, the upper bound is trivial. As for the lower bound note that
  \[
  S\supset \bigcap_{j\geq 0} \bigcup_{m\in M(j)} \Gamma_m (G_2)=:S' .
  \]
  Thus, letting $I\subset \RR$ be the orthogonal projection of $G_2$ onto the $x_1$ axis, we see that $I$ is an interval (since $U_2$ is connected). Thus the orthogonal projection of $S'$ onto the $x_1$ axis is 
 \[
 \bigcap_{j\geq 0} \bigcup_{m\in M(j)} \pi_m (I)=C,
 \] 
  where $C$ is as in the proposition above. Thus, since the orthogonal projection onto the $x_1$ axis is a Lipschitz map, we obtain $d_H(S')\geq d_H (C)$ (as a property of Hausdorff dimension, see, for example, Proposition 3.3 in \cite{falconer}). Consequently
  \[
  d_H (S) \geq d_H (S') \geq d_H(C) \geq \xi,
  \]
  as required (recall \eqref{haus_dim_of_C_is_at_least_xi} for the last inequality).

\subsection[Sketch of the proof]{Sketch of the proof of Theorem \ref{1D_blowup_thm}}\label{sec_sketch_of_pf_of_thm2}

As in the proof of Theorem \ref{point_blowup_thm}, the proof is based on a geometric arrangement. Here we will need a certain sharper geometric arrangement as follows.
\begin{framed}
By the \emph{geometric arrangement} (for Theorem \ref{1D_blowup_thm}) we mean a pair open sets $U_1,U_2 \Subset P$ together with the corresponding structures $(v_1,f_1,\phi_1)$, $(v_2,f_2,\phi_2)$ such that $\overline{U_1}\cap \overline{U_2}=\emptyset$ and, for some $T>0$, $X>0$, $\tau \in (0,1)$, $z=(z_1,z_2,0) \in \RR^3 $, $M\in \NN$, \eqref{CANTOR_tau_xi_M} and \eqref{CANTOR_gamma_n_maps_G_inside} hold with
\[
G=G_1\cup G_2 \coloneqq R\left( \overline{U_1} \cup \overline{U_2} \right),
\]
\eqnb\label{cantor_fairies_extra_ineq}
f_2^2 +T v_2 \cdot F[v_1,f_1] > |v_2 |^2\quad \text{ in }U_2
\eqne
and
\eqnb\label{cantor_fairies_scaling}
f_2^2(y_n) +T v_2(y_n) \cdot F[v_1,f_1](y_n) > \tau^{-2} \left(  f_1 (R^{-1}x) +  f_2(R^{-1}x) \right)^2 
\eqne
for all $x\in G$ and $n=1,\ldots, M$, where
\eqnb\label{def_of_y_n} 
y_n =  R^{-1} (\Gamma_n (x))
\eqne
and $\Gamma_n$ is as in \eqref{equiv_def_of_Gamma_n}.
\end{framed}
The difference, as compared to the previous geometric arrangement (see Section \ref{sec_the_setting}) is \eqref{CANTOR_gamma_n_maps_G_inside} and \eqref{cantor_fairies_scaling}, which we now require for all $n=1,\ldots , M$ (rather than only for $n=1$, which was the case previously). This reflects the fact that at each switching time we expect the support of $\mathfrak{u}$ to form $M$ copies of itself (rather than one copy, which was the case in Theorem \ref{point_blowup_thm}). Except for this, the inequalities \eqref{CANTOR_tau_xi_M} specify the relation between $\tau$ and $M$ that needs to be satisfied in order to obtain blow-up on a Cantor set with Hausdorff dimension at least $\xi$.
In fact, the previous geometric arrangement is recovered if one takes $\xi=0$, $M=1$. We now show how Theorem \ref{1D_blowup_thm} follows (given the geometric arrangement) in a similar way as discussed in Section \ref{sec_the_setting}, except for a subtle change in the construction of the vector field $u^{(j)}$ (recall the previous construction \eqref{uj_rescaling_def}).

To this end, as in Section \ref{sec_the_setting}, let $\theta >0$ be sufficiently small such that
\eqnb\label{cantor_def_of_theta}
f_2^2(y_n) +T v_2(y_n) \cdot F[v_1,f_1](y_n) > \tau^{-2} \left(  f_1 (R^{-1}x) +  f_2(R^{-1}x) \right)^2 +2\theta
\eqne
for $x\in G$, $n=1,\ldots , M$ (by \eqref{cantor_fairies_scaling}), and set
\eqnb\label{cantor_def_of_h}
h_t \coloneqq h_{1,t}+ h_{2,t},
\eqne
where $h_1$, $h_2$ are given by \eqref{def_h1}, \eqref{def_h2}, that is
\eqnb\label{cantor_def_of_h1h2}
\begin{split}
h_{1,t}^2 &\coloneqq f_1 ^2 - 2t \delta \phi_1  ,\\
h_{2,t}^2 &\coloneqq f_2^2 -2t\delta \phi_2 + \int_0^t v_2  \cdot  F[v_1,h_{1,s}] \,\d s .
\end{split}
\eqne
As in Lemma \ref{prop_of_h1h2}, let $\delta >0$ be sufficiently small so that $h_1,h_2 \in C^\infty ( P\times (-\delta , T+\delta ); [0,\infty ))$,
\[
(v_i,h_{i,t},\phi_i ) \text{ is a structure on } U_i \qquad \text{ for } t\in (-\delta , T+\delta ), i=1,2,
\]
and
\eqnb\label{CANTOR_3.14}
h_{2,T}^2 (y_n)    > \tau^{-2}  \left(  f_1 (R^{-1}x) +  f_2(R^{-1}x) \right)^2  + \theta 
\eqne
for $x\in G$, $n=1,\ldots, M$. Here only the last inequality differs from the corresponding property \eqref{3.14}; note however that this is a consequence of \eqref{cantor_def_of_theta}, as previously \eqref{3.14} was a consequence of \eqref{fairies_scaling1}.

Let $\nu_0>0$ be as in \eqref{how_small_is_nu0}. As in Section \ref{sec_proof_of_thm1}, in order to obtain a solution $\mathfrak{u}$ we want to find $\eta_j>0$ and a velocity field $u^{(j)}$. However, in contrast to the arguments from Section \ref{sec_proof_of_thm1}, the velocity field $u^{(j)}$ will not be obtained by rescaling a single vector field $u$ (recall \eqref{uj_rescaling_def}), which we have pointed out above. In fact, for each $j$ we expect $u^{(j)}$ to consist of $M^j$ disjointly supported vector fields (recall the comments preceding Section \ref{sec_CANTOR_set}). A naive idea of constructing $u^{(j)}$ would be to consider $M^j$ rescaled copies of $u$, that is the vector field
\[
\widetilde{u}^{(j)} (x,t) \coloneqq \tau^{-j}\sum_{m\in M(j)} u(\Gamma_m^{-1} (x), \tau^{-2j} (t-t_j)), \qquad j\geq 0.
\]
 For such vector field 
 \[
 \supp\, \widetilde{u}^{(j)} (t) = \bigcup_{m\in M(j)} \Gamma_m (G),\qquad t\in [t_j ,t_{j+1}], j\geq 0,
 \]
which shrinks to the Cantor set $S$ as $j\to \infty$ (recall \eqref{S_intersection}), as expected. However, the observation that the pressure function does not have a local character (that is the pressure function corresponding to a compactly supported vector field does not have compact support, recall \eqref{def_of_the_corresponding_pressure}) suggests that $\widetilde{u}$ has little chance to satisfy the local energy inequality \eqref{local_energy_inequality}. Instead, one needs to make use of the following proposition.
\begin{proposition}\label{prop_CANTOR_existence_of_vj}
Let $j\geq 0$ and 
\eqnb\label{rescaling_h^(j)}
h_t^{(j)} (x_1,x_2) \coloneqq \sum_{m\in M(j)} h_t (\pi_m^{-1} (\tau^j x_1),x_2) ,
\eqne
where $h_t$ is given by \eqref{cantor_def_of_h}, $t\in (-\delta , T+\delta )$. Then there exists a vector field $v^{(j)}  \in C^\infty \left( \RR^3 \times [0, T]; \RR^3 \right)$ such that
\begin{enumerate}
\item[(i)] $\mathrm{div}\, v^{(j)} (t)=0$ and $ \supp \, v^{(j)} (t) = R\left( \supp \, h_t^{(j)}  \right)$, $ t\in [0, T]$,
\item[(ii)] for all $x\in \RR^3$, $t\in [0,T]$
\[ \left| v^{(j)} (x,0)  \right| = h_0^{(j)} \left( R^{-1} x \right),\qquad  \left| \left| v^{(j)} (x,t) \right|^2 - h_t^{(j)} \left( R^{-1} x \right)^2 \right| < \theta ,\]
\item[(iii)] the Navier--Stokes inequality
\[
\p_t \left| v^{(j)} \right|^2 \leq -v^{(j)}\cdot \nabla \left( \left| v^{(j)} \right|^2 +2\overline{p}^{(j)} \right) + 2\nu \, v^{(j)}\cdot \Delta v^{(j)} 
\]
is satisfied in $\RR^3 \times [0,T]$ for every $\nu\in [0,\nu_0 ]$, and
\item[(iv)] $\| v^{(j)} (t) \|_{L^2}\leq \mathcal{C}$ for $t\in [0,T]$ and
 \[
 \int_0^T \| \nabla v^{(j)} (t) \|_{L^2}^2\d t\leq \mathcal{C},  
 \]
 for some constant $\mathcal{C}>0$ which is independent of $j$, where $\overline{p}^{(j)}$ is the pressure function corresponding to $v^{(j)}$.
\end{enumerate}
\end{proposition}
We prove the proposition in in Section \ref{sec_cantor_pf_of_thm_uj} below.
Proposition \ref{prop_CANTOR_existence_of_vj} is a generalisation of the previous result (Proposition \ref{prop_existence_of_u}, which is recovered by taking $j=0$) In the previous construction, $u^{(j)}$ (i.e. the solution for times between $t_j$ and $t_{j+1}$, recall Figure \ref{switching_process_figure_with_supp}) was given by 
\eqnb\label{uj_rescaling_def_recalling}
u^{(j)} (x,t) \coloneqq \tau^{-j} v^{(0)} \left( \Gamma^{-j} (x) , \tau^{-2j} (t-t_j) \right), \qquad j\geq 1
\eqne
(recall \eqref{uj_rescaling_def}), where $v^{(0)}$ is from the proposition above (or equivalently $u$ from Proposition \ref{prop_existence_of_u}). Here given $j\geq 0$ we set
\eqnb\label{cantor_what_is_uj}
u^{(j)}(x_1,x_2,x_3,t) \coloneqq \tau^{-j} v^{(j)} (\tau^{-j} x_1 ,\gamma^{-j} (x_2), \tau^{-j}x_3, \tau^{-2j}(t-t_j))
\eqne
(where $t_0\coloneqq 0$ and $t_j \coloneqq T\sum_{k=0}^{j-1}\tau^{2k}$, as previously), which (at each time $t$) consists of $M^j$ disjointly supported (in rescaled copies of $G= R(\overline{U_1} \cup \overline{U_2} )$) vector fields and we ensure (by constructing $v^{(j)}$ in Proposition \ref{prop_CANTOR_existence_of_vj}) that such a ``composite vector field'' still satisfies the Navier--Stokes inequality despite the nonlocal character of the pressure function. Indeed, (as in Section \ref{sec_the_setting} (cf. Proposition~\ref{prop_existence_of_u})) it follows from claims (i), (iii) above that $u^{(j)} \in C^\infty (\RR^3 \times [t_j,t_{j+1}] ; \RR^3)$ is divergence free and satisfies the Navier--Stokes inequality
\eqnb\label{claim_for_u(j)_NSI}
\p_t \left| u^{(j)} \right|^2 \leq -u^{(j)}\cdot \nabla \left( \left| u^{(j)} \right|^2 +2p^{(j)} \right) + 2\nu \, u^{(j)}\cdot \Delta u^{(j)} 
\eqne
 in $\RR^3 \times [t_j,t_{j+1}]$ for all $\nu \in [0,\nu_0]$, where $ p^{(j)}$ is the pressure function corresponding to $u^{(j)}$ (recall that $C^\infty (\RR^3 \times [a ,b]; \RR^3)$ denotes the space of vector functions that are infinitely differentiable on $\RR^3\times (a-\eta , b+ \eta )$ for some $\eta>0$). Moreover, although the rescaling used in \eqref{uj_rescaling_def_recalling} might seem different from the one in \eqref{cantor_what_is_uj}, note that by \eqref{cantor_what_is_uj}, the rescaling used in  \eqref{rescaling_h^(j)} and (i), (ii) we obtain 
 \eqnb\label{cantor_what_is_supp_of_uj}
\supp\, u^{(j)} (t) = \bigcup_{m\in M(j)} \Gamma_m (G),\qquad t\in [t_j, t_{j+1}]
\eqne
(cf. the previous relation \eqref{what_is_supp_uj}) and 
  \eqnb\label{claim_for_u(j)_2}\begin{split}
\left| u^{(j)} (x,t_j ) \right| &= \tau^{-j} \sum_{m\in M(j)} h_0 (R^{-1} (\Gamma_m^{-1}(x))) , \\
 \left| u^{(j)} (x,t_{j+1} ) \right|^2 &> \tau^{-2j} \sum_{m\in M(j)} h_T (R^{-1}(\Gamma_m^{-1} (x)))^2 - \tau^{-2j}\theta . \quad x\in \RR^3.
\end{split}
\eqne
The last two lines can be used to show that  
 \eqnb\label{cantor_decrease_at_switching_of_ujs}
\left| u^{(j)} (x,t_j ) \right| \leq \left| u^{(j-1)} (x,t_j ) \right|, \qquad  x\in \RR^3, j \geq 1,
\eqne
in a similar way as \eqref{decrease_at_switching_of_ujs}. Indeed, in order to see this note that this inequality is nontrivial only for $x \in \bigcup_{m\in M(j)} \Gamma_m (G)$, and so let $j\geq 1$, $m\in M(j)$ be such that $x= \Gamma_m (y)$ for some $y\in G$. Then, in the light of \eqref{different_mutliindices_disjoint_subsets}, we see that $\Gamma_{\widetilde{m}}^{-1} (x) \not \in G$ for any $\widetilde{m}\in M(j)$, $\widetilde{m}\ne m$, and so the first line of \eqref{claim_for_u(j)_2} becomes simply
\eqnb\label{what_is_u(j)_at_x}
\left| u^{(j)} (x,t_j ) \right| = \tau^{-j}  h_0 (R^{-1} (y)).
\eqne
Furthermore, letting $\overline{m}\in M(j-1)$ be the sub-multiindex of $m$, that is $m=(\overline{m},m_j)$ for some $m_j \in \{ 1,\ldots , M\}$, we see that
\[
x=\Gamma_{\overline{m}} \left( \Gamma_{m_j} (y) \right).
\]
This means that, at $(j-1)$-th step (that is for $t\in [t_{j-1},t_j]$) $x$ was an element of $\Gamma_{\overline{m}}(G)$ (and at time $t_j$ this component of $\supp\,u^{(j-1)}$ will divide into $M$ disjoint copies, $\{ \Gamma_{\overline{m},n} (G) \}_{n=1,\ldots ,M}$, which will become $M$ out of $M^j$ components of $\supp\, u^{(j)}$ (see \eqref{cantor_what_is_supp_of_uj}); and among these copies $x$ belongs to $\Gamma_m (G)$). Therefore, as in \eqref{what_is_u(j)_at_x} above we see that the second line of \eqref{claim_for_u(j)_2} (taken for $j-1$) is simply
\[\begin{split}
 \left| u^{(j-1)} (x,t_{j} ) \right|^2&> \tau^{-2(j-1)}  h_T (R^{-1}(\Gamma_{\overline{m}}^{-1} (x)))^2 - \tau^{-2(j-1)}\theta \\
 &= \tau^{-2(j-1)}  h_T (R^{-1}(\Gamma_{m_j} (y)))^2 - \tau^{-2(j-1)}\theta.\\
 &=\tau^{-2(j-1)}  h_{2,T} (R^{-1}(\Gamma_{m_j} (y)))^2 - \tau^{-2(j-1)}\theta,
 \end{split}
\]
where, in the last equality, we used the fact that $h_{1,T} (R^{-1} (\Gamma_{m_j} (y)))=0$ (recall \eqref{CANTOR_gamma_n_maps_G_inside} gives $R^{-1} (\Gamma_{m_j} (y)) \in \overline{U_2}$).
From this and \eqref{what_is_u(j)_at_x} we obtain \eqref{cantor_decrease_at_switching_of_ujs} by an easy calculation,
\[\begin{split}
 \left| u^{(j-1)} (x,t_{j} ) \right|^2&> \tau^{-2(j-1)}  h_{2,T} (R^{-1}(\Gamma_{m_j} (y)))^2 - \tau^{-2(j-1)}\theta \\
 &> \tau^{-2j}  \left( f_1(R^{-1}(y)) + f_2(R^{-1}(y))  \right)^2 \\
 &= \tau^{-2j}  h_0 (R^{-1} (y))^2\\
 &= \left|u^{(j)} (x,t_j ) \right|^2,
\end{split}
\]
where we used \eqref{CANTOR_3.14} in the second inequality.

Hence, letting
\eqnb\label{cantor_mathfraku}
\mathfrak{u} (t)  \coloneqq \begin{cases}
 u^{(j)}(t)  \qquad &\text{ if } t\in [t_{j},t_{j+1}) \text{ for some }j\geq 0, \\
 0 &\text{ if } t\geq  T_0,
\end{cases}
\eqne
where $T_0\coloneqq \lim_{j\to \infty } t_j = T/(1-\tau^2)$ (as previously), we obtain a solution to Theorem \ref{1D_blowup_thm}. Indeed, that $\mathfrak{u}$ is a weak solution to the NSI follows as in the case of Theorem \ref{point_blowup_thm} (note that, in order to obtain the required regularity $\sup_{t>0} \| \mathfrak{u} \| <\infty $, $\nabla  \mathfrak{u} \in L^2 (\RR^3 \times (0,\infty )) $ it suffices to replace ``$\tau$'' by ``$M\tau$'' and $\sup_{t\in [t_0,t_1]} \| u^{(0)} (t) \|$, $\int_{t_0}^{t_1} \| \nabla u^{(0)} (t) \|^2 \d t $ by $\mathcal{C}$ (from Proposition \ref{prop_CANTOR_existence_of_vj} (iv); recall also the shorthand notation $\| \cdot \| \equiv \| \cdot \|_{L^2 (\RR^3 )})$) in the calculations \eqref{regularity_L2_is_finite}, \eqref{regularity_L2_of_nabla}). 

Furthermore the singular set of $\mathfrak{u}$ is
\[
S \times \{ T_0 \}= \left( \bigcap_{j\geq 0} \bigcup_{m\in M(j)} \Gamma_m (G) \right) \times \{ T_0 \}.
\]
Indeed, \eqref{cantor_what_is_supp_of_uj} shows that the support of $\mathfrak{u}(t)$ consists of $M^j$ components for $t\in [t_j, t_{j+1})$ and that it shrinks to the Cantor set $S$ as $t\to T_0^-$. That $\mathfrak{u}$ is unbounded in any neighbourhood $V$ of any point $(y,T_0) \in S\times \{ T_0 \}$ follows from Proposition \ref{prop_CANTOR_existence_of_vj} (ii) and \eqref{cantor_what_is_uj}, which show that the magnitude of $\mathfrak{u}$ grows uniformly on each component of its support; in other words given any positive number $\mathcal{N}$ let $y\in G$ be any point such that $h_0(R^{-1}(y))>0$ and let $j\geq 0$, $m\in M(j)$ be such that 
\[
\Gamma_m (G)\times [t_j,T_0] \subset V \quad \text{ and }\quad \tau^{-j} \geq \mathcal{N} / h_0 (R^{-1}(y)).
\]
Then, by \eqref{what_is_u(j)_at_x},
\[|\mathfrak{u}(\Gamma_m (y),t_j)|=\tau^{-j} h_0 (R^{-1}(y)) \geq \mathcal{N}.
\]
Thus $S\times \{ T_0 \}$ is a singular set of $\mathfrak{u}$ whose Hausdorff dimension is greater than $\xi$ (due to \eqref{cantor_bounds_on_dH_of_S}).

\subsection[Proof of Proposition 4.3]{Proof of Proposition \ref{prop_CANTOR_existence_of_vj}}\label{sec_cantor_pf_of_thm_uj}
Here we prove Proposition \ref{prop_CANTOR_existence_of_vj}, which concludes the proof of Theorem \ref{1D_blowup_thm} (given the geometric arrangement, which we discuss in Section \ref{sec_cantor_new_geometric_arrangement}). Fix $j\geq 0$. We will write, for brevity, $v=v^{(j)}$, $\overline{p}=\overline{p}^{(j)}$.\vspace{0.3cm}\\
\emph{Step 1.} We renumber the functions $h_i (\pi_m^{-1} (\tau^j x_1 ), x_2,t)$.\\

For brevity we let $\mathfrak{M} \coloneqq M^j$, identify each multiindex $m\in M(j)$ with an integer $\mathfrak{m}\in \{ 1,\ldots , \mathfrak{M} \}$, and let 
\eqnb\label{cantor_hm_is_a_translate}
h_i^\mathfrak{m} (x_1,x_2,t) \coloneqq h_i (\pi_m^{-1} (\tau^j x_1 ), x_2,t),\qquad i=1,2
\eqne
(recall that $h_t = h_1 (\cdot ,t) + h_2 (\cdot ,t)$), and 
\[
\begin{cases}
f_i^\mathfrak{m} (x_1,x_2) \coloneqq f_i (\pi_m^{-1} (\tau^j x_1 ), x_2),\\
v_i^\mathfrak{m} (x_1,x_2) \coloneqq v_i (\pi_m^{-1} (\tau^j x_1 ), x_2),\\
\phi_i^\mathfrak{m} (x_1,x_2) \coloneqq \phi_i (\pi_m^{-1} (\tau^j x_1 ), x_2),\\
U_i^\mathfrak{m} \coloneqq \{ (x_1,x_2) \colon (\pi_m^{-1} (\tau^j x_1 ), x_2) \in U_i \}
\end{cases}\qquad i=1,2.
\]
Then $h_1^\fm,h_2^\fm \in C^\infty ( P\times (-\delta , T+\delta ); [0,\infty ))$,
\[
(v_i^\fm,f_{i}^\fm,\phi_i^\fm ) \text{ and }(v_i^\fm,h_{i,t}^\fm,\phi_i^\fm ) \text{ are structures on } U_i^\fm \qquad \text{ for } t\in (-\delta , T+\delta ), i=1,2,
\]
and
\[
h_t^{(j)} = \sum_{\mathfrak{m}=1}^{\mathfrak{M}} \left( h_{1,t}^\mathfrak{m} +h_{2,t}^\mathfrak{m} \right).
\]
Moreover, 
\[
\supp \, (h_{1,t}^\fm + h_{2,t}^\fm ) = \overline{U_1^\fm} \cup \overline{U_2^\fm} =: K^\fm
\]
and the sets $K^\fm$ are pairwise disjoint translates of $\overline{U_1}\cup \overline{U_2}$ in the $x_1$ direction, such that the distance between any $K^\fm$ and $K^\fn$ for $\fm , \fn \in \{ 1,\ldots , \fmm \}$, $\fn\ne\fm$, is at least $\tau^{-1} \zeta$ (just as each element of the union $\bigcup_{m\in M(j)} \Gamma_m(G)$ is separated from the rest by at least $\tau^{j-1} \zeta$, see the comments preceding \eqref{cantor_bounds_on_dH_of_S}). Furthermore, we can assume that the bijection $m \longleftrightarrow \fm$ is such that $K^{\fm+1}$ is a positive translate of $K^\fm $ in the $x_1$ direction, that is 
\eqnb\label{Km_s_are_translates} K^{\fm+1}  = K^\fm + (a_\fm,0) \quad \text{ for some } a_\fm >0, \,\,\,\fm=1,\ldots , \fmm-1. \eqne
For such a bijection 
\eqnb\label{dist_between_Kn_s}
\mathrm{dist}(K^{\fn},K^\fm) \geq |\fn - \fm |\tau^{-1}\zeta, \qquad \fn, \fm=  1,\ldots , \fmm .
\eqne
\emph{Step 2.} We introduce the modification $q_{i,t}^{\fm ,k}$ of $h_{i,t}^\fm$.\\

Let
\eqnb\label{def_of_qm}
\left( q_{i,t}^{\fm,k} \right)^2 \coloneqq \left( h_{i,0}^\fm \right)^2 -2t\delta \phi_i^\fm - \int_0^t a_i^{\fm ,k} (s) v_i^\fm\cdot \left( \nabla (h_{i,s}^\fm)^2+2  \sum_{l=1,2} \sum_{\fn=1}^{\mathfrak{M}} \nabla p[a_l^{\fn ,k}(s) v_l^\fn ,h_{l,s}^\fn ] \right)\d s ,
\eqne
$i=1,2$, $k\in \NN$, $\fm=1,\ldots , \fmm$, where $a_i^{\fm ,k} \in C^\infty (\RR ; [-1,1])$, $i=1,2$, $m=1,\ldots , \mathcal{M}$ are the oscillatory functions which we contruct below. Observe that this is a natural extension of the idea from Section \ref{sec_strategy_for_u} to the case of $\fmm$ pairs $U_1^\fm, U_2^\fm$ (rather than a single pair $U_1$, $U_2$, which was the case previously). Note that such a definition gives
\eqnb\label{deriv_of_qm}
\p_t \left( q_{i,t}^{\fm ,k} \right)^2 = -2\delta \phi_i^\fm + a_i^{\fm ,k} (t) v_i^\fm \cdot \left(\nabla (h_{i,t}^\fm )^2 + 2  \sum_{l=1,2} \sum_{\fn=1}^{\fmm} \nabla p[a_l^{\fn ,k}(t) v_l^\fn ,h^\fn_{l,t} ]  \right).
\eqne
As in \eqref{conv_of_qs}, we will construct the oscillatory processes $a_1^{\fm , k},  a_2^{\fm , k} \in C^\infty (\RR ; [-1,1])$ in such a way that for each multiindex $\alpha = (\alpha_1 ,\alpha_2)$  
\eqnb\label{conv_of_qms}
\begin{cases}
q_{i,t}^{\fm , k} \to h_{i,t}^\fm  \\
\text{ and }\\
D^\alpha q_{i,t}^{\fm ,k} \to D^\alpha h_{i,t}^\fm\quad 
\end{cases} \qquad \text{ uniformly in } P \times [0,T], i=1,2, \fm\in \{ 1,\ldots , \fmm \}.
\eqne
As in Section \ref{sec_strategy_for_u}, we obtain that for $i=1,2$ and sufficiently large $k$
\[
(a_i^{\fm ,k}(t)v_i^\fm,q_{i,t}^{\fm ,k}, \phi_i^\fm) \text{ is a structure on } U_i^\fm\text{ for } t\in (-\delta_k,T+\delta_k), 
\]
and that 
\[
q_i^{\fm ,k} \in C^\infty (P\times (-\delta_k ,T+\delta_k);[0,\infty )) .
\]

Finally let
\eqnb\label{cantor_def_of_v}
v(x,t) \coloneqq \sum_{\fm=1}^{\fmm} \left( u[a_1^{\fm,k} (t) v_1^\fm , q_{1,t}^{\fm,k}](x)+ u[a_2^{\fm,k} (t) v_2^\fm , q_{2,t}^{\fm,k}](x) \right),
\eqne
\emph{Step 3.} We verify that $v$ satisfies the claims of the proposition.\\

We will now show that (given the existence of the oscillatory processes $a_1^{\fm , k},  a_2^{\fm , k}$, which we construct in Section \ref{sec_cantor_osc_processes}) the function \eqref{cantor_def_of_v} is a solution to Proposition \ref{prop_CANTOR_existence_of_vj}.

Claim (i) is trivial and so is claim (ii) given $k$ large enough such that
\[
\left| (q_{i,t}^{\fm ,k})^2 - (h_{i,t}^\fm)^2 \right| \leq \theta /2 \qquad \text{ in } P, t\in [0,T], i=1,2.
\]
As for claim (iii) (the Navier--Stokes inequality) note that (as previously) since $v^{(j)}$ is axisymmetric, it is equivalent to
 \[
\p_t |v(x,0,t)|^2 \leq -v(x,0,t) \cdot \nabla \left( |v(x,0,t)|^2 +2\overline{p}(x,0,t) \right) + 2\nu\, v(x,0,t) \cdot \Delta v(x,0,t),
\]
where $\nu \in [0,\nu_0]$, $x\in P$, $t\in[0,T]$ and $\overline{p}$ is the pressure function corresponding to $v$, that is 
\eqnb\label{cantor_what_is_p(t)}
\overline{p}(t) = \sum_{\fm=1}^\fmm \left( p^* [a_1^{\fm ,k}(t) v_1^\fm,q_{1,t}^{\fm , k} ]+p^* [a_2^{\fm , k}(t) v_2^\fm ,q_{2,t}^{\fm , k} ] \right)
\eqne
(recall \eqref{def_of_p*} and Lemma \ref{lem_prop_of_p[v,f]} (iii)), which in particular (i.e. restricting ourselves to $P$) means that
\[
\overline{p}(x,0,t) = \sum_{\fm=1}^\fmm \left( p [a_1^{\fm ,k}(t) v_1^\fm,q_{1,t}^{\fm ,k} ](x)+p [a_2^{\fm ,k}(t) v_2^\fm,q_{2,t}^{\fm ,k} ](x) \right), \quad x\in P.
\]
As in Section \ref{sec_pf_of_claims_of_thm} we fix $x\in P$, $t\in [0,T]$ and consider two cases.\vspace{0.4cm}\\
\emph{Case 1.} $\phi_1^\fm (x)+\phi_2^\fm (x)<1$ for all $\fm \in \{ 1, \ldots , \fmm\}$. 

For such $x$ we have $v_1^\fm (x) = v_2^\fm (x) =0$ (from the elementary properties of structures, recall Definition \ref{def_of_structure}) and the Navier--Stokes inequality follows trivially for all $k$ by writing
\[
\begin{split}
\p_t |v(x,0,t)|^2 &= \sum_{\fm=1}^\fmm \left( \p_t q_{1,t}^{\fm, k} (x)^2 + \p_t q_{2,t}^{\fm, k}  (x)^2 \right) \\
&= -2\delta \sum_{\fm=1}^\fmm  (\phi_1^\fm (x) + \phi_2^\fm (x) )\\
&\leq 0\\
&\leq - v(x,0,t)\cdot \nabla \left( |v(x,0,t)|^2 +2 \overline{p} (x,0,t) \right) + 2\nu \,v(x,0,t) \cdot \Delta v(x,0,t) ,
\end{split}
\]
where we used \eqref{prop_of_structure_1} and \eqref{prop_of_structure_2} in the last step.\vspace{0.4cm}\\
\emph{Case 2.} $\phi_1^\fm (x) + \phi_2^\fm (x) =1$ for some $\fm \in \{ 1, \ldots , \fmm \}$. 

In this case we need to use the convergence \eqref{conv_of_qms} with $k$ sufficiently large such that
\eqnb\label{cantor_conv_1}
 \left| v_i^\fm \right| \left( \left| \nabla (q_{i,t}^{\fm ,k})^2 - \nabla ( h_{i,t}^\fm )^2 \right| +2 \sum_{\fn=1}^\fmm \sum_{l=1,2}\left| \nabla p[a_l^{\fn, k} (t) v_l^\fn , q_{l,t}^{\fn ,k} ]-\nabla p [a_l^{\fn ,k} (t) v_l^\fn ,h_{l,t}^\fn ] \right|  \right)\leq \delta /2
\eqne
in $P$ (see Lemma \ref{lem_continuity_of_the_pressure_fcns} for a verification that \eqref{conv_of_qms} is sufficient for the convergence of the pressure functions) and
\eqnb\label{cantor_conv_2}
\begin{split}
&\nu_0 \left| u[a_i^{\fm ,k}(t) v_i^\fm , q_{i,t}^{\fm ,k} ] \cdot \Delta u [a_i^{\fm ,k}(t) v_i^\fm , q_{i,t}^{\fm ,k} ] \right| \\
&\leq \nu_0 \left| u[a_i^{\fm ,k}(t) v_i^\fm , h_{i,t}^\fm ] \cdot \Delta u [a_i^{\fm ,k}(t) v_i^\fm , h_{i,t}^\fm ] \right| + \delta/8 \leq \delta/4 
\end{split}
\eqne
in $\RR^3$ (see Lemma \ref{lem_continuity_of_u_dot_Delta_u} for a verification that \eqref{conv_of_qms} is sufficient for the the first inequality), for $t\in [0,T]$, $i=1,2$, where we used \eqref{how_small_is_nu0} in the last inequality. Recall that $\delta >0$ has been fixed below \eqref{cantor_def_of_h1h2}. We obtain
\[
\begin{split}
\p_t |v(x,0,t)|^2 &=  \p_t q_{1,t}^{\fm ,k} (x)^2 + \p_t q_{2,t}^{\fm , k}  (x)^2   \\
&= -2\delta - \left( a_1^{\fm ,k} (t) v_1^\fm (x) +a_2^{\fm ,k} (t)v_2^\fm (x) \right) \cdot \nabla \left( \left( h_{1,t}^\fm (x) \right)^2 + \left(  h_{2,t}^\fm (x) \right)^2  {\color{white} \sum_{\fn=1}^\fmm}\right. \\
&\hspace{4cm}+ \left. 2\sum_{\fn=1}^\fmm \left( p[a_1^{\fn ,k}(t) v_1^\fn ,h_{1,t}^\fn ](x) + p[a_2^{\fn ,k}(t) v_2^\fn ,h_{2,t}^\fn ](x)\right) \right) \\
&\leq -\delta - \left( a_1^{\fm ,k} (t) v_1^\fm (x) +a_2^{\fm ,k} (t)v_2^\fm (x) \right) \cdot \nabla \left( q_{1,t}^{\fm , k} (x)^2 + q_{2,t}^{\fm ,k} (x)^2 {\color{white} \sum_{\fn=1}^\fmm}  \right. \\
&\hspace{4cm}+ \left. 2 \sum_{\fn=1}^\fmm \left( p[a_1^{\fn ,k}(t) v_1^\fn ,q_{1,t}^{\fn ,k}](x) + p[a_2^{\fn ,k}(t) v_2^\fn ,q_{2,t}^{\fn ,k}](x) \right) \right) \\
&= -\delta - v_1(x,0,t) \p_{x_1} \left( |v(x,0,t)|^2 + 2\overline{p}(x,0,t) \right)\\
&\hspace{1cm}- v_2(x,0,t) \p_{x_2} \left( |v(x,0,t)|^2 + 2\overline{p}(x,0,t) \right),
\end{split}
\]
and so, recalling that $\p_{x_3} |v(x,0,t)|^2 = \p_{x_3} \overline{p}(x,0,t) =0$ (as a property of axisymmetric functions, see \eqref{d3_of_|u|_vanishes} and \eqref{dx3_of_p_is_zero}),
\[\begin{split}
\p_t |v(x,0,t)|^2  &\leq - \delta - v (x,0,t) \cdot \nabla \left( |v(x,0,t)|^2 +2 \overline{p} (x,0,t) \right) \\
&\leq 2 \nu \, v(x,0,t) \cdot \Delta v (x,0,t)- v (x,0,t) \cdot \nabla \left( |v(x,0,t)|^2 +2 \overline{p} (x,0,t) \right)
\end{split}
\]
for all $\nu \in [0,\nu_0]$, where we used \eqref{cantor_conv_2} in the last step.

It remains to verify (iv). For this note that 
\[
|v(x,t)|= \sum_{\fm=1}^\fmm \left| q_{1,t}^{\fm ,k} (R^{-1}x)+q_{2,t}^{\fm ,k}(R^{-1}x) \right|
\]
(recall that $\{q_{i,t}^{\fm , k}\}_{i=1,2, \, \fm =1,\ldots, \fmm}$ have disjoint supports $U_i^\fm$, respectively), 
and thus, in the view of \eqref{conv_of_qms}, for sufficiently large $k$
 \eqnb\label{cantor_v_is_bdd}
\begin{split}
|v(x,t)|&\leq  \sum_{\fm=1}^\fmm \left| h_{1,t}^{\fm }(R^{-1}x)+h_{2,t}^{\fm }(R^{-1}x) \right| +1 \\
&\leq \sup_{s\in[0,T]} \|h_{1,s}+h_{2,s}\|_{L^\infty } +1\\
&= \mathcal{C},\end{split}
\eqne
since the functions $h_{1,t}^{\fm }+h_{2,t}^{\fm }$ have disjoint supports $K^\fm$ ($\fm=1,\ldots , \fmm$). Here we write $\mathcal{C}$ for a constant that is independent of $j$ whose value might change from line to line (this should not be confused with $C$, which is a constant related to the decay of the pressure function and was fixed above Section \ref{sec_simplified_geom_arrangement}). Hence, since $\supp \,v(t)= \bigcup_{\fm=1}^\fmm R(K^\fm)$ consists of $\fmm$ copies of $R(\overline{U_1}\cup \overline{U_2})$ (recall \eqref{Km_s_are_translates}) we obtain, by H\"older's inequality, that 
\[
\| v(t) \|_{L^2} \leq \fmm \mathcal{C}, \qquad \text{ for }t\in [0,T]
\]
for some $\mathcal{C}>0$. 

As in \eqref{cantor_v_is_bdd} we have for sufficiently large $k$
\[
\|q_{1,t}^{\fm,k} + q_{2,t}^{\fm,k}  \|_{W^{1,\infty }} \leq \|h_{1,t}^{\fm} + h_{2,t}^{\fm}  \|_{W^{1,\infty }} +1, 
\]
and so, applying \eqref{bound_on_grad_of_u[v,f]}, we obtain
\eqnb\label{cantor_nabla_v_is_bdd}\begin{split}
|\nabla v(x,t) |&\leq \sum_{\fm=1}^{\fmm} \left| \nabla u[a_1^{\fm,k} (t) v_1^\fm , q_{1,t}^{\fm,k}](x)+ \nabla u[a_2^{\fm,k} (t) v_2^\fm , q_{2,t}^{\fm,k}](x) \right|\\
&\leq \max_{\fm\in \{ 1,\ldots ,\fmm\} }  \mathcal{C}(\|v_1^\fm + v_2^\fm  \|_{W^{1,\infty }} , \|q_{1,t}^{\fm,k} + q_{2,t}^{\fm,k}  \|_{W^{1,\infty }} )\\
&\leq \max_{\fm\in \{ 1,\ldots ,\fmm\}, s\in [0,T] } \mathcal{C}(\|v_1^\fm + v_2^\fm  \|_{W^{1,\infty }} , \|h_{1,s}^{\fm} + h_{2,s}^{\fm}  \|_{W^{1,\infty }}+1 )\\
&=   \max_{s\in [0,T]} \mathcal{C}(\|v_1 + v_2  \|_{W^{1,\infty }} , \|h_{1,s} + h_{2,s} \|_{W^{1,\infty }}+1 ) \\
&= \mathcal{C},
\end{split}
\eqne
and therefore
\[
\int_0^T \| \nabla v(t) \|_{L^2}^2 \d t \leq \fmm \mathcal{C},
\]
as required.

\subsection{The new oscillatory processes}\label{sec_cantor_osc_processes}

Here we prove the existence of oscillatory processes $a_1^{\fm,k}, a_2^{\fm,k} \in C^\infty (\RR ; [-1,1])$ which give the convergence \eqref{conv_of_qms}.
The construction of such oscillatory processes is a natural extension of the construction of the processes $a_1^k, a_2^k$ from Section \ref{sec_oscillatory_process} to the case of $\fmm$ pairs $U_1^\fm$, $U_2^\fm$ (and the corresponding structures, $\fm=1,\ldots , \fmm$). In particular we will use the following sharper version of Theorem \ref{thm_existence_of_osc_proc}.
\begin{theorem}\label{thm_existence_of_osc_proc_cantor}
For each $k\geq 1$, $\fm=1,\ldots , \fmm$ there exists a pair of functions $a_1^{\fm,k}, a_2^{\fm,k} \in C^\infty (\RR ; [-1,1])$, $i=1,2$, such that
\eqnb\label{osc_process_cantor_conv}\begin{split}
&\int_0^t a_i^{\fm,k} (s) \left( G_i^\fm (x,s) + \sum_{l=1,2}\sum_{\fn=1}^{\fmm} F_{i,l}^{\fm,\fn} \left( x,s,a_j^{\fn ,k}(s) \right)  \right) \d s\\
&\stackrel{k\to \infty}{\longrightarrow } \begin{cases} \frac{1}{2}\int_0^t  \left( F_{2,1}^{\fm,\fm} (x,s,1)-F_{2,1}^{\fm,\fm} (x,s,0) \right) \d s \quad &i=2, \\
0 &i=1
\end{cases}
\end{split}
\eqne
uniformly in $(x,t) \in P\times [0,T ]$, $\fm=1,\ldots , \fmm$ for any bounded and uniformly continuous functions 
\[
G_i^\fm \colon P\times [0,T] \to \RR , \qquad F_{i,l}^{\fm,\fn} \colon P\times [0,T] \times [-1,1] \to \RR,
\]
$i,l=1,2$, $\fm,\fn=1,\ldots ,\fmm$ satisfying 
\[
F_{i,l}^{\fm,\fn} (x,t,-1) = F_{i,l}^{\fm,\fn}  (x,t,1) \qquad \text{for }x\in P, t\in [0,T],\, i,l=1,2,\, \fm,\fn=1,\ldots ,\fmm .
\]
\end{theorem}
Note that, as in Section \ref{sec_oscillatory_process}, this theorem gives \eqref{conv_of_qms} simply by taking 
\[
\begin{split}
G_i^\fm (x,t) &\coloneqq v_i^\fm (x) \cdot \nabla (h_i^\fm (x,t))^2,\\
F_{i,l}^{\fm,\fn} (x,t,a) &\coloneqq 2v_i^\fm (x) \cdot \nabla p[av^\fn_l,h^\fn_{l,t}](x)
\end{split}
\]
(recall $F_{i,l}^{\fm,\fn} (x,t,-1) = F_{i,l}^{\fm,\fn} (x,t,1)$ by the property $p[v,f]=p[-v,f]$, see Lemma \ref{lem_prop_of_p[v,f]} (i)) and by taking 
\[
\begin{split}
G^\fm_i(x,t) &\coloneqq D^\alpha \left( v_i^\fm (x) \cdot \nabla (h_i^\fm(x,t))^2 \right),\\
F^{\fm,\fn}_{i,l} (x,t,a) &\coloneqq D^\alpha \left( 2v_i^\fm (x) \cdot \nabla p[av^\fn_l,h^\fn_{l,t}](x) \right)
\end{split}
\]
for any given multiindex $\alpha=(\alpha_1,\alpha_2)$.\\

In order to see that the theorem above is a sharpening of Theorem \ref{thm_existence_of_osc_proc}, recall that the role of the processes $a_1^k, a_2^k$ (given by Theorem \ref{thm_existence_of_osc_proc}) was (in a sense) to ``pick'' (among all influences of the set $U_i$ on the set $U_j$, $i,j\in \{1,2\}$) only the influence of $U_1$ on $U_2$ (recall the comments following Theorem \ref{thm_existence_of_osc_proc}). Here, instead of a pair $U_1,U_2$ we have to deal with $\fmm$ pairs $U_1^\fm, U_2^\fm$ ($\fm=1,\ldots , \fmm$) and the role of the processes $a_1^{\fm,k}, a_2^{\fm,k}$ is to ``pick'' (among all influences of $U_i^\fn$ on $U_l^\fm$, $i,l\in \{1,2\}$, $\fn,\fm\in \{1,\ldots, \fmm\}$)  only the influence of $U_1^\fm$ on $U_2^\fm$ for all $\fm \in \{1,\ldots , \fmm \}$ (that is for each pair pick only the influence of the first set on the second one).
Thus, recalling that the choice of the processes $a_1^k, a_2^k$ (in Section \ref{sec_oscillatory_process}) was based on the ``basic processes'' $b_1,b_2$ (recall \eqref{def_of_b1_b2}) having the simple integral property \eqref{basic_processes_magic}, we can obtain the processes $a_1^{\fm,k},a_2^{\fm,k}$ by finding processes $b_1^{(\fm )}, b_2^{(\fm )}$, $\fm=1,\ldots, \fmm$ such that an analogous property holds:
\eqnb\label{cantor_basic_processes_magic}
\int_0^T b_i^{(\fn )} (s) f\left( b_l^{( \fm )} (s)\right) \d s =\begin{cases}
\frac{T}{2} (f(1)-f(0))  &(i,l) = (2,1),\fm=\fn ,\\
0 \quad &\text{otherwise}
\end{cases} 
\eqne
for any $f\colon [-1,1]\to \RR$ such that $f(-1)=f(1)$. Such processes can be obtained by letting $b_1^{(1)}\coloneqq b_1$, $b_2^{(1)}\coloneqq b_2$ and letting $b^{(\fm )}_i$ have $4$ times higher frequency than $b^{(\fm-1)}_i$, $i=1,2$, $\fm\in \{ 2,\ldots , \fmm \}$, that is letting
\eqnb\label{def_of_b1m_b2m}
b_1^{(\fm )}(t) \coloneqq b_1 (4^{\fm-1}t), \qquad b_2^{(\fm )}(t) \coloneqq b_2 (4^{\fm-1}t) 
\eqne
where we have also extended $b_1, b_2$ $T$-periodically to the whole line, see Fig.\ \ref{processes_b1m_b2m}. 
\begin{figure}[h]
\centering
 \includegraphics[width=\textwidth]{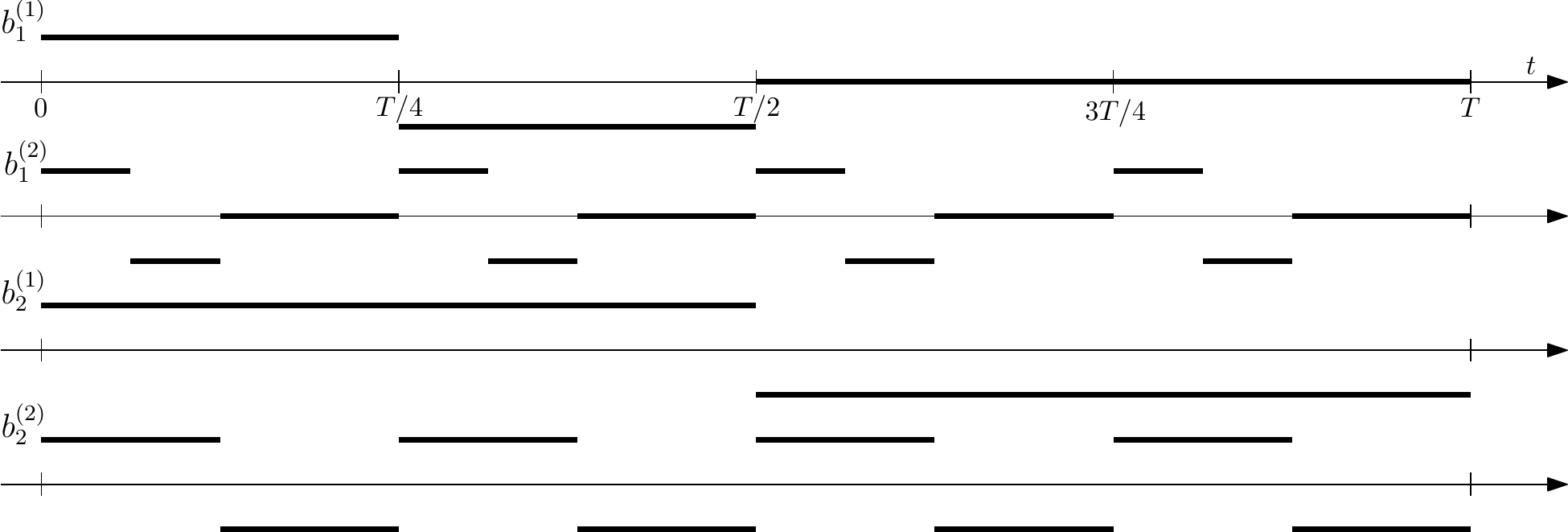}
 \nopagebreak
  \captionof{figure}{The processes $b_1^{(\fm )}$, $b_2^{(\fm )}$, $\fm=1,\ldots , \fmm$. Here $\fmm =2$.}\label{processes_b1m_b2m} 
\end{figure}
Analogously as in Section \ref{sec_oscillatory_process} the convergence \eqref{osc_process_cantor_conv} can be obtained by letting, for each $k$, $b_1^{\fm,k}$, $b_2^{\fm,k}$ ($\fm=1,\ldots, \fmm$) be oscillations of the above form with frequency increasing with $k$, that is
\eqnb\label{def_of_b1mk_b2mk}
b_1^{\fm,k}(t) \coloneqq  b_1^{(\fm )}(kt)=b_1(k4^{\fm - 1}t) , \quad  b_2^{\fm,k}(t) \coloneqq  b_2^{(\fm )}(kt)=b_2(k4^{\fm - 1}t).
\eqne
We omit the detailed calculation.\\

Finally, as in Section \ref{sec_oscillatory_process}, the smoothness of the processes can be obtained by smooth approximation of the processes $b_1^{\fm,k},b_2^{\fm,k}$, that is by letting $a_1^{\fm , k},a_2^{\fm,k} \in C^\infty (\RR; [-1,1])$ be such that
\[
\left| \lewy t \in [0,T] \colon a_i^{\fm,k} (t) \ne b_i^{\fm,k} (t) \prawy \right| \leq \frac{1}{k} ,\qquad i=1,2,\fm =1,\ldots , \fmm .
\]

\subsection{The new geometric arrangement}\label{sec_cantor_new_geometric_arrangement}
In this section we construct the geometric arrangement as described in Section \ref{sec_sketch_of_pf_of_thm2}. That is we need to find $U_1,U_2 \Subset P$ (with disjoint closures) together with the corresponding structures $(v_1,f_1,\phi_1)$, $(v_2,f_2,\phi_2)$ and numbers $T>0$, $\tau \in (0,1)$, $z=(z_1,z_2,0)\in \RR^3$, $X>0$, $M\in \NN$ such that except for \eqref{fairies_extra_ineq}, \eqref{fairies_scaling} (which was all that we required in the proof of Theorem \ref{point_blowup_thm}, recall Section \ref{sec_the_setting}) we also have \eqref{CANTOR_tau_xi_M}, \eqref{CANTOR_gamma_n_maps_G_inside} and \eqref{cantor_fairies_scaling}, that is $\{\Gamma_n (G)\}_{n=1,\ldots , M}$ is a family of pairwise disjoint subsets of $G_2$ (recall $G=G_1\cup G_2=R(\overline{U_1})\cup R( \overline{U_2})$),
\[
\tau^\xi M\geq 1,\qquad \tau M<1
\]
and
\[
f_2^2(y_n) +T v_2(y_n) \cdot F[v_1,f_1](y_n) > \tau^{-2} \left(  f_1 (R^{-1}x) +  f_2(R^{-1}x) \right)^2 
\]
for all $x\in G$ and $n=1,\ldots, M$, where $y_n =  R^{-1} (\Gamma_n (x))$. The construction builds on the objects defined previously (in Section \ref{sec_geom_arrangement}) and, remarkably, can be obtained simply by taking $\varepsilon >0$ (the main parameter of the previous geometric arrangement, recall \eqref{how_small_is_eps}) smaller, which we present in several steps.\vspace{0.4cm}\\
\emph{Step 1.} We recall some objects from Section \ref{sec_geom_arrangement}.\\

Let
\[
U, v,f,\phi, F, A,B,C,D, \kappa \quad \text{ and } a',r',s',a'',r'',s'', H, E
\]
be as in Section \ref{sec_geom_arrangement}. In particular, $U$ is a rectangle in $P$, $(v,f,\phi)$ is a structure on $U$, $F=F[v,f]$ is the corresponding pressure interaction function, the constants $A,B,C,D \in \RR$ are given by the properties of the pressure interaction function $F$ (recall Lemma \ref{lem_properties_of_F}), $\kappa = 10^4 C/D\geq 10^4$ (recall \eqref{def_of_kappa}), the numbers $a',r',s',a'',r'',s''$ define the copies $U^{a',r'}$, $U^{a'',r''}$ of $U$ (and the copies of the corresponding structures) in a way that the joint pressure interaction function $H=F+F^{a',r',s'}+F^{a'',r'',s''}$ has certain decay and certain behaviour on the $x_1$ axis (that is (i)-(iii) from Section \ref{sec_copies_of_U} hold), and $E>0$ is sufficiently small such that the strip $0<x_2<E$ is disjoint with $U\cup U^{a',r'} \cup U^{a'',r''}$ and $H$ enjoys certain properties in this strip (that is (iv)-(vi) from Section \ref{sec_copies_of_U} hold).\vspace{0.4cm}\\
\emph{Step 2.} We consider disjoint copies of $U\cup U^{a',r'} \cup U^{a'',r''}$ in the $x_1$ direction.\\

That is we we let $X>0$ be sufficiently large so that
\eqnb\label{how_large_is_X}
\begin{split}
&X> \mathrm{diam}\left( U \cup U^{a',r'} \cup U^{a'',r''} \right),  \qquad X>4|A|, \\
& 2C X^{-4} \sum_{k\in \ZZ} \left( |k| -\frac{1}{2}  \right)^{-4} < 0.01B, \qquad \text{ and } \qquad X > 2\kappa E ,
\end{split}
\eqne
and consider the collection of copies of $U\cup U^{a',r'} \cup U^{a'',r''}$:
\eqnb\label{family_of_sets}
\lewy U^{nX,1} \cup U^{a'+nX,r'} \cup U^{a''+nX,r''} \prawy_{n\in \ZZ},
\eqne
together with the structures that are the corresponding translations by $(nX,0)$ of 
\[
(v,f,\phi )+\left( v^{a',r',s'},f^{a',r',s'},\phi^{a',r'}\right)+\left( v^{a'',r'',s''},f^{a'',r'',s''},\phi^{a'',r''} \right) ,
\]
recall \eqref{defs_of_copying} (see Fig. \ref{fig_copies_of_u}). The role of $X$ is to separate these copies (and the corresponding structures) sufficiently far from each other. In particular we see that they have disjoint closures by the first inequality in \eqref{how_large_is_X}. Note also that since each of $U^{nX} \cup U^{a'+nX,r'} \cup U^{a''+nX,r''}$, $n\in \ZZ$, is a translation in the $x_1$ direction of $U \cup U^{a',r'} \cup U^{a'',r''}$, it is disjoint with the strip $\{ 0 < x_2 < E \}$ (recall (iv) in Section \ref{sec_copies_of_U}), see Fig. \ref{fig_copies_of_u}.
\begin{figure}[h]
\centering
 \includegraphics[width=0.7\textwidth]{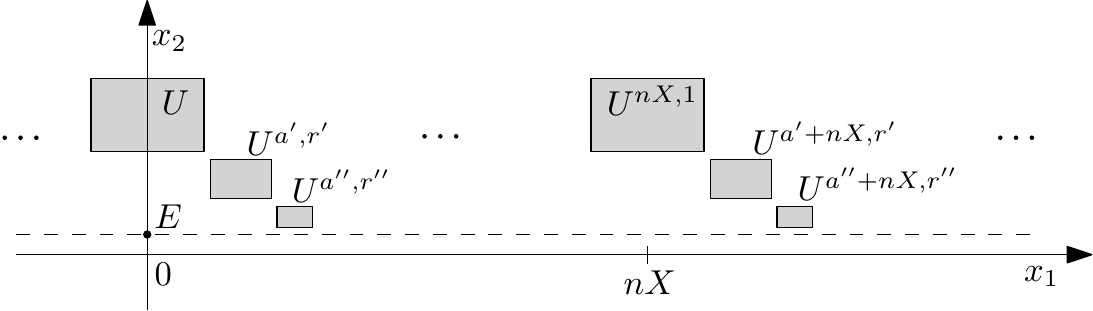}
 \nopagebreak
  \captionof{figure}{The sets $U^{nX,1} \cup U^{a'+nX,r'} \cup U^{a''+nX,r''}$, $n\in \ZZ$.}\label{fig_copies_of_u} 
\end{figure}

Moreover, note that for each $n\in \ZZ$
\[H(x_1-nX,x_2)=\left( F^{nX,1}+F^{a'+nX,r',s'}+F^{a''+nX,r'',s''}\right) (x_1,x_2),\qquad  (x_1,x_2)\in \RR^2,
\]
that is $H(x_1-nX,x_2)$ is the pressure interaction function corresponding to $U^{nX,1} \cup U^{a'+nX,r'} \cup U^{a''+nX,r''}$ (with the structure as pointed out above). We now show that the choice of $X$ above gives that for each $k\in \ZZ$ the total pressure interaction of the sets \eqref{family_of_sets} with $n\ne k$ (and their structures) is very small near $U^{kX} \cup U^{a'+kX,r'} \cup U^{a''+kX,r''}$, which we make precise in the following lemma.
\begin{lemma}
Given $x_1\in \RR$ let $k\in \ZZ$ be such that
\[
|x_1-kX|=\min_{n\in \ZZ} |x_1-nX |.
\]
Then
\[
\sum_{n\ne k} |H(x_1-nX,x_2) |< 0.01 B, \qquad \text{ for all }x_2\in [0,E).
\]
\end{lemma} 
\begin{proof}
Note that the dinition of $k$ means simply that
\[
x_1 \in [kX-X/2,kX+X/2].
\]
Thus if $n\ne k$ then 
\[
|x_1-nX | \geq \left( |n-k | -\frac{1}{2} \right) X,
\]
cf. Fig. \ref{fig_copies_of_u}. Thus in particular
\[
|x_1 - nX | \geq X/2 \geq 2|A|,
\]
where we used the fact that $X \geq 4|A|$ (see \eqref{how_large_is_X}), and so we can use the decay of $H$ (see property (iii) of $H$) to write
\[
| H(x_1-nX,x_2 ) | \leq 2C |x_1-nX |^{-4} \leq 2C \left( |n-k | -\frac{1}{2} \right)^{-4} X^{-4} .
\]
Summing up in $n$ and using the third inequality in \eqref{how_large_is_X} we obtain
\[
\sum_{n\ne k} |H(x_1-nX,x_2) | \leq 2CX^{-4} \sum_{n\ne k} \left( |n-k | -\frac{1}{2} \right)^{-4} \leq 0.01B .\qedhere
\]
\end{proof}
Thus, for any $M\in \NN $ the function
\[
H^* (x_1,x_2) \coloneqq \sum_{n=0}^{M-1} H(x_1-nX, x_2)
\]
is the pressure interaction function corresponding to
\[
\bigcup_{n=0}^{M-1} \left( U^{nX} \cup U^{a'+nX,r'} \cup U^{a''+nX,r''} \right),
\]
and the above lemma and properties (v) and (vi) of $H$ give
\begin{enumerate}
\item[(i)] $H_1^* (x) \geq -1.02B$ in the strip $\{ 0<x_2 <E\},$ 
\item[(ii)] $H_1^* (x) \geq 6.98B$ for $x\in P$ with $|x_1-A-(m-1)X| < \kappa E$, $0<x_2<E$ for any $m=1,\ldots , M$.
\end{enumerate}\vspace{0.4cm}
\emph{Step 3.} We take $\varepsilon >0$ small, and define $v_2$, $U_2$.\\

Given $\varepsilon >0$ let $\tau\coloneqq 0.48\varepsilon$ and
\eqnb\label{def_of_rdM}
r\coloneqq E/\varepsilon ,\qquad d \coloneqq \kappa r ,\qquad M\coloneqq 1+ \frac{d}{4X}.
\eqne
(Recall $\kappa \geq 10^4$ (see Step 1).) Note each of $r,d,M$ (is defined in the same way as previously and) is of order $\varepsilon^{-1}$. Let $\varepsilon $ be small such that in addition to \eqref{how_small_is_eps} we also have that $M$ is a positive integer and 
\eqnb\label{CANTOR_how_smaller_is_eps}
\tau^\xi M \geq 1,\qquad 
\varepsilon^2 M< \frac{10^{-6} B E^4}{2C}.
\eqne
Note that this gives \eqref{CANTOR_tau_xi_M}, which is clear from the first of the two inequalities above and by writing
\[
\tau M = \tau + \frac{\tau d }{4X}  =\tau + \frac{0.48 \kappa E}{4X}< \tau + \frac{1}{2} <1,
\]
where we used the facts $X>\kappa E/4$ (recall \eqref{how_large_is_X}) and $\tau<1/2$ (recall that in fact $\tau<1/20$ by the first inequality in \eqref{how_small_is_eps}).

Having fixed $\varepsilon$ we let (as previously) $v_2$ be given by Lemma \ref{lemma_existence_of_v2} and the sets $U_2$, $\mbox{\emph{BOX}}$, $\mbox{\emph{RECT}}$, $\mbox{\emph{SBOX}}$ be defined as in \eqref{defs_subsets_of_P}, that is
\[
\begin{split}
U_2 &\coloneqq (-d,d)\times (0.005 \varepsilon r , r) \setminus [-(d-r),d-r]\times [\varepsilon r , r/10 ], \\
\mbox{\emph{BOX}} &\coloneqq [-d,d] \times [0,r], \\
\mbox{\emph{SBOX}} &\coloneqq [A-\kappa E, A+ \kappa E] \times [0.02\varepsilon r , 0.98 \varepsilon r] ,\\
\mbox{\emph{RECT}} &\coloneqq [-(d-r),d-r] \times [0.02 \varepsilon r , 0.98 \varepsilon r ] .
\end{split}
\]
Note that $U_2$ encompasses the union of all $M$ copies of $U\cup U^{a',r'} \cup U^{a'',r''}$, that is
\eqnb\label{cantor_3M_copies_of_U_fit}
\bigcup_{n=0}^{M-1} \left( U^{nX,1}\cup U^{a'+nX,r'} \cup U^{a''+nX,r''} \right) \subset (-(d-r),d-r )\times (\varepsilon r,r/10 )
\eqne
(see Fig. \ref{cantor_the_arrangement}), which can be verified in the same way as \eqref{three_copies_are_encompassed}, except for the use of the inequality $d-r>\mathrm{diam}\, \left( U\cup U^{a',r'} \cup U^{a'',r''} \right) $, which can be improved by using the fifth inequality in \eqref{how_small_is_eps}, 
\[
d>2\mathrm{diam}\, \left( U\cup U^{a',r'} \cup U^{a'',r''} \right), 
\]
to give 
\[\begin{split}
d-r&= \left( \frac{d}{2}-r \right) +\frac{d}{2} \\
&> \frac{d}{4} + \mathrm{diam}\left( U\cup U^{a',r'} \cup U^{a'',r'' }\right)\\
& = (M-1)X + \mathrm{diam}\left( U\cup U^{a',r'} \cup U^{a'',r'' }\right),
\end{split}
\]
where we also used the fact that $d>4r$ (recall \eqref{def_of_rdM} above). Thus we obtain \eqref{cantor_3M_copies_of_U_fit}. \\

Let
\eqnb\label{def_of_SBOX_m}
\mbox{\emph{SBOX}}_m \coloneqq \mbox{\emph{SBOX}} + (m-1)(X,0) , \quad m=1,\ldots , M,
\eqne
and observe that $\{ \mbox{\emph{SBOX}}_m \}_{m=1}^M$ is a family of pairwise disjoint subsets of $\mbox{\emph{RECT}}$ (cf. Fig. \ref{cantor_subsets_of_p}). Indeed, the disjointness follows from the fact that $X>2\kappa E$ (recall the last inequality in \eqref{how_large_is_X}), the inclusion $\mbox{\emph{SBOX}}_1 \subset \mbox{\emph{RECT}}$ follows as previously (recall the comment following \eqref{defs_subsets_of_P}) and the inclusion $\mbox{\emph{SBOX}}_M \subset \mbox{\emph{RECT}}$ follows by writing
\[
(M-1)X+A+\kappa E = \frac{d}{4} + A+\kappa E < \frac{d}{4} +(d-r)/2 < d-r,
\]
where we used the second inequality in \eqref{how_small_is_eps} and the fact that $d>2r$ (recall \eqref{kappa_is_large}).

Let
\[
a\coloneqq -\kappa r /\varepsilon, \quad \frac{s^2}{r} \coloneqq 1.04 \left(- \frac{a}{r} \right)^4 B/D 
\]
(as previously, see \eqref{def_of_as}) and recall that then Lemma \ref{lemma_Fars} gives 
\eqnb\label{CANTOR_Fars}
1.03 B \leq F_1^{a,r,s} \leq 1.05B \quad \text{ and }\quad  |F_2^{a,r,s} | \leq 0.01 \varepsilon B \quad \text{ in } \mbox{\emph{BOX}}.
\eqne
\emph{Step 4.} We define $U_1$, its structure $(v_1,f_1,\phi_1)$, and show the lower bound $v_2\cdot F[v_1,f_1]\geq -1.1\varepsilon B$.\\

Letting
\[
U_1 \coloneqq \bigcup_{n=0}^{M-1} \left( U^{nX,1}\cup U^{a'+nX,r'} \cup U^{a''+nX,r''} \right) \cup U^{a,r} , \]
and
\[
\begin{split}
f_1 &\coloneqq \sum_{n=0}^{M-1} \left( f^{nX,1,1} + f^{a'+nX,r',s'}+ f^{a''+nX,r'',s''}\right)+ f^{a,r,s} ,\\
v_1 &\coloneqq \sum_{n=0}^{M-1} \left( v^{nX,1,1} + v^{a'+nX,r',s'}+ v^{a''+nX,r'',s''}\right) + v^{a,r,s},\\
\phi_1 &\coloneqq \sum_{n=0}^{M-1} \left( \phi^{nX,1} + \phi^{a'+nX,r'}+ \phi^{a''+nX,r''}\right) + \phi^{a,r}\\
\end{split}
\]
we obtain a structure $(v_1,f_1,\phi_1)$ on $U_1$. We see that $\overline{U^{a,r}}$ is located to the left of $\mbox{\emph{BOX}}$ (as previously, see \eqref{Uar_is_to_the_left_of_box}) and so, in the view of \eqref{cantor_3M_copies_of_U_fit},
\eqnb\label{cantor_U1_U2_are_disjoint}
U_1,U_2 \Subset P \,\text{ have disjoint closures.}
\eqne
Denoting by $F^*$ the total pressure interaction function,
\[\begin{split}
F^* \coloneqq F[v_1,f_1]&= \sum_{n=0}^{M-1} \left( F^{nX,1,1} + F^{a'+nX,r',s'}+ F^{a''+nX,r'',s''}\right)+ F^{a,r,s} \\
&= H^* +F^{a,r,s} ,
\end{split}
\]
we see that properties (i), (ii) of $H^*$ above (see Step 2) and \eqref{CANTOR_Fars} give
\eqnb\label{cantor_Fstar_in_sets}
\begin{cases}
F^*_1 \geq 0.01 B \quad & \text{ in } (-(d-r),d-r)\times (0,\varepsilon r) \supset \mbox{\emph{RECT}}, \\
F^*_1 \geq 8B \quad &\text{ in } \mbox{\emph{SBOX}}_m , \quad m=1,\ldots ,M,
\end{cases}
\eqne
cf. \eqref{Fstar_in_sets}. Moreover
\eqnb\label{cantor_v_times_F_loweR_bound}
v_2 \cdot F^* \geq -1.1\varepsilon B \qquad \text{ in } \mbox{\emph{BOX}},
\eqne
which is an analogue of the previous relation \eqref{v_times_F_loweR_bound} and which we now verify. Let $x\in \supp \, v_2$ (otherwise the claim is trivial). \\

\emph{Case 1.} $x \in [0,(M-1)X]\times \{ 0 \} + B(0,r/10)$. In this case $x_2\in (0,\varepsilon r)$ (see Fig. \ref{cantor_subsets_of_p}) and
\[
-(d-r)<r/10<x_1<(M-1)X +r/10=d/4+r/10<d-r,
\]
where the left-most and the right-most inequalities follow from the fact that $d\geq 10^4r$ (recall \eqref{kappa_is_large}). 
\begin{figure}[h]
\centering
 \includegraphics[width=\textwidth]{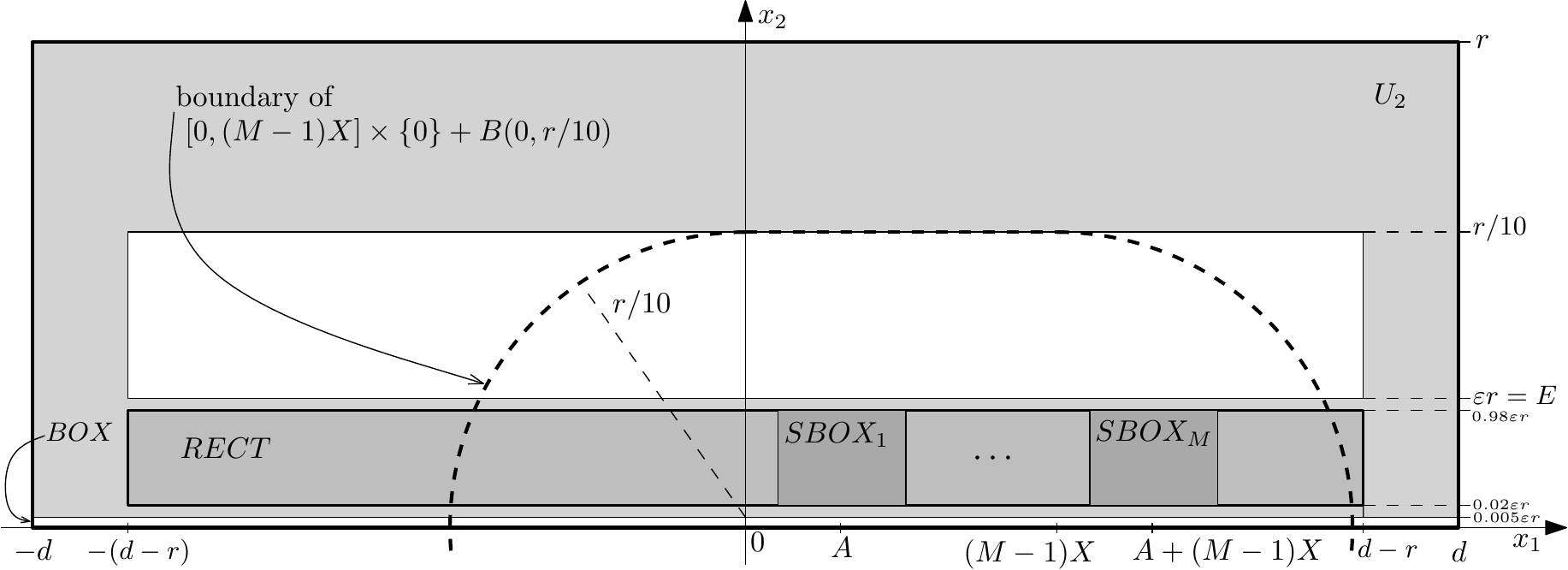}
 \nopagebreak
 \captionsetup{width=0.9\textwidth} 
  \captionof{figure}{The sets $U_2$, $\mbox{\emph{BOX}}$, $\mbox{\emph{RECT}}$, and $\mbox{\emph{SBOX}}_m$, $m=1,\ldots , M$ (compare with Fig. \ref{subsets_of_p}). Note that some proportions are not conserved on this sketch.}\label{cantor_subsets_of_p} 
\end{figure}Thus $x\in (-(d-r),d-r)\times (0,\varepsilon r) $ (see Fig. \ref{cantor_subsets_of_p}) and consequently the choice of $v_2$ (see Lemma \ref{lemma_existence_of_v2} (iii)) and \eqref{cantor_Fstar_in_sets} give
\[
v_2 (x) \cdot F^*(x) = v_{21} (x) F_1^* (x) \geq 0.01B v_{21} (x) >0 > -1.1\varepsilon B.
\]

\emph{Case 2.} $x \not \in [0,(M-1)X]\times \{ 0 \} + B(0,r/10)$. In this case
\eqnb\label{cantor_H*_bound}
|H^* (x) | \leq 0.01 \varepsilon^2 B,
\eqne
which is an analogue of \eqref{the_use_of_decay_of_H} and which follows from the decay of $H$ (that is property (iii) in Section \ref{sec_copies_of_U}). Indeed, since in this case
\[
|(x_1-nX, x_2) | \geq r/10 ,\qquad n=0,\ldots ,M-1,
\]
and since $r>20|A|$ (recall \eqref{how_small_is_eps}) we obtain
\[
|(x_1-nX, x_2) | \geq 2|A| \qquad n=0,\ldots ,M-1.
\]
Thus we can use property (iii) of $H$ (see Section \ref{sec_copies_of_U}) to write 
\[\begin{split}
|H^* (x_1,x_2) |&\leq \sum_{n=0}^{M-1} |H(x_1-nX,x_2) | \\
&\leq 2C \sum_{n=0}^{M-1}  |(x_1-nX,x_2) |^{-4} \\
&\leq 2C \sum_{n=0}^{M-1}  \left( \frac{10}{r}\right)^4 \\
&= 2\cdot 10^4 C M \varepsilon^{4} E^{-4}\\
& < 0.01 \varepsilon^2 B, 
\end{split}
\]
where we used \eqref{CANTOR_how_smaller_is_eps} in the last step. Hence we obtained \eqref{cantor_H*_bound}, and so, from our choice of $v_2$ (namely that $|v_2|\leq 2$, $v_{21}\geq -\varepsilon^2$, $|v_{22}|\leq \varepsilon/2$, recall Lemma \ref{lemma_existence_of_v2} (iii)) and the bounds on $F^{a,r,s}$ (see \eqref{CANTOR_Fars}) we obtain \eqref{cantor_v_times_F_loweR_bound} by writing
\[
\begin{split}
v_2(x) \cdot F^* (x) &= v_2 (x) \cdot H(x) + v_{21} (x) F_1^{a,r,s} (x) + v_{22} (x) F_2^{a,r,s} (x) \\
&\geq - 2 (0.01 \varepsilon^2 B) - \varepsilon^2 (1.05 B) - \frac{\varepsilon  }{2} (0.01 B\varepsilon ) \\
&= - \varepsilon^2 B (0.02 +1.05+0.005) \\
&\geq -1.1 \varepsilon^2 B.
\end{split}
\]
\emph{Step 5.} We verify \eqref{CANTOR_gamma_n_maps_G_inside}, i.e. that $\{ \Gamma_n (G) \}_{n=1}^{M}$ is a family of pairwise disjoint subsets of $G_2$.\\

As previously we let 
\[z\coloneqq (A,\varepsilon r/2 ,0)\]
and observe that
\eqnb\label{boxes_fit_into_sboxes}
R^{-1} (\Gamma_m (R(\mbox{\emph{BOX}})) ) \subset \mbox{\emph{SBOX}}_m \qquad m=1,\ldots , M,
\eqne
which follows in the same way as the prevous property \eqref{box_fits_into_sbox}. In fact \eqref{box_fits_into_sbox} corresponds to the case $m=1$, and the claim for other values of $m$ follows by translating in the $x_1$ direction both sides of \eqref{box_fits_into_sbox} by multiplies of $X$, see Fig. \ref{cantor_the_arrangement}. Thus, since the sets $\mbox{\emph{SBOX}}_m$, $m=1,\ldots , M$, are pairwise disjoint (recall the comment below \eqref{def_of_SBOX_m}),
\[ \{ R^{-1} (\Gamma_m (R(\mbox{\emph{BOX}})) ) \}_{m=1}^M \text{ is a family of pairwise disjoint subsets of }\mbox{\emph{RECT}}\text{.}
\] 
\begin{figure}[h]
\centering
 \includegraphics[width=\textwidth]{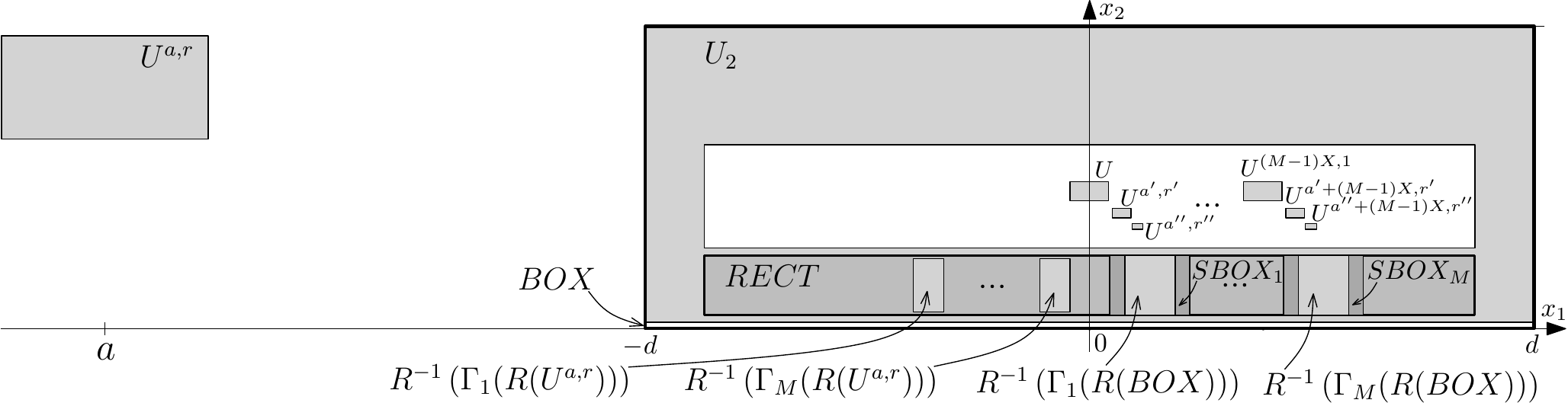}
 \nopagebreak
 \captionsetup{width=0.9\textwidth} 
  \captionof{figure}{The geometric arrangement for Theorem \ref{1D_blowup_thm} (compare with the previous geometric arrangement, see Fig. \ref{the_arrangement}). We note that proportions are not conserved on this sketch.}\label{cantor_the_arrangement} 
\end{figure}

We now show that 
\eqnb\label{cantor_Uar_is_mapped_to_disjoint_subsets_of_rect}\begin{split}
\lewy R^{-1} \left( \Gamma_m \left( \overline{U^{a,r}} \right) \right)\prawy_{m=1}^M \text{ is }&\text{a pairwise disjoint family of subsets of }\mbox{\emph{RECT}}\\
&\text{which are located to the left of }\mbox{\emph{SBOX}}_1,
\end{split}
\eqne
which is an analogue of the previous relation \eqref{Uar_fits_into_rect}, see Fig. \ref{cantor_the_arrangement}. Here ``to the left of'' refers to the property that the $x_1$ coordinate of any point of $R^{-1} \left( \Gamma_M \left( \overline{U^{a,r}} \right) \right)$ is strictly less than the $x_1$ coordinate of any point of $\mbox{\emph{SBOX}}_1$; since both $R^{-1} \left( \Gamma_M \left( \overline{U^{a,r}} \right) \right)$ and $\mbox{\emph{SBOX}}_1$ are rectangles, this is simply  
\[
\tau (a+r) +A+(M-1)X < A-\kappa E.
\]
This inequality can be verified using the facts $\varepsilon <1/10$ (recall \eqref{how_small_is_eps}) and $\kappa >1$ (recall \eqref{kappa_is_large}) by writing
\[\begin{split}
\tau (a+r) +(M-1)X &=\tau r (1-\kappa/\varepsilon ) + d/4 \\
&= 0.48 r (\varepsilon - \kappa ) + \kappa r /4 \\
& = 0.48 r\varepsilon - 0.23 \kappa r \\
&< 0.48 r\varepsilon - 2\varepsilon \kappa r \\
&=\varepsilon r (0.48 - 2\kappa ) \\
&< -\kappa \varepsilon r \\
&= -\kappa E, 
\end{split}
\]
as required, where we used the fact that $2 \varepsilon <0.23$ (recall \eqref{how_small_is_eps}) in the first inequality. Property \eqref{cantor_Uar_is_mapped_to_disjoint_subsets_of_rect} is now clear by observing that the sets $R^{-1} \left( \Gamma_m \left( \overline{U^{a,r}} \right) \right)$ are pairwise disjoint (recall each of these sets is a rectangle whose length (in the $x_1$ direction) is $2\tau r$ (see the comment following \eqref{Uar_fits_into_rect}) and that $X>2E=2\varepsilon r > 2 \tau r$, where the first inequality holds by our choice of $X$ (recall \eqref{how_large_is_X}) and the second one is simply that $\tau = 0.48 \varepsilon<\varepsilon$) and by recalling the previous property \eqref{Uar_fits_into_rect}, 
\[
R^{-1} \left( \Gamma_1 \left( \overline{U^{a,r}} \right) \right) \subset \mbox{\emph{RECT}}.
\]

Properites \eqref{boxes_fit_into_sboxes} and \eqref{cantor_Uar_is_mapped_to_disjoint_subsets_of_rect} give that
\[
\lewy  R^{-1} \left( \Gamma_m (G) \right) \prawy_{m=1}^M \, \text{ is a family of disjoint subsets of } \mbox{\emph{RECT}}
\]
(recall $G= R(\overline{U_1}\cup \overline{U_2})$), which gives \eqref{CANTOR_gamma_n_maps_G_inside}. Indeed, by \eqref{boxes_fit_into_sboxes} and \eqref{cantor_Uar_is_mapped_to_disjoint_subsets_of_rect},
\[
\Gamma_m (G) \subset R(\mbox{\emph{RECT}}) \subset R(\overline{U_2}) = G_2,\quad m=1,\ldots , M,
\]
and the disjointness follows from the disjointness of the cylindrical projections.\vspace{0.8cm}\\
\emph{Step 6.} Define $T$, $f_2$ and $\phi_2$ and show the remaining claims \eqref{cantor_fairies_extra_ineq} and \eqref{cantor_fairies_scaling}.\\ \hspace{0.4cm}(This finishes the construction of the geometric arrangement and thus also finishes the proof of Theorem \ref{1D_blowup_thm}.)\\

Let $T$, $f_2$, $\phi_2$ be defined as previously (see Section \ref{sec_constr_f2_and_rest}). Then $(v_2,f_2,\phi_2)$ is a structure on $U_2$ and \eqref{cantor_fairies_extra_ineq} and \eqref{cantor_fairies_scaling} follow in the same way as \eqref{fairies_extra_ineq_rewritten}, \eqref{fairies_scaling_rewritten} in Section \ref{sec_constr_f2_and_rest} by making the following replacements. Replace $y$ by $y_n$ and $\mbox{\emph{SBOX}}$ by $\mbox{\emph{SBOX}}_n$ ($n=1,\ldots , M$), and use the relations \eqref{cantor_v_times_F_loweR_bound}, \eqref{boxes_fit_into_sboxes}, \eqref{cantor_Fstar_in_sets}, \eqref{cantor_Uar_is_mapped_to_disjoint_subsets_of_rect} instead of the previous relations \eqref{v_times_F_loweR_bound}, \eqref{box_fits_into_sbox}, \eqref{Fstar_in_sets}, \eqref{Uar_fits_into_rect} (respectively).

\subsection[Blow-up with ``almost equality'']{Blow-up on a Cantor set with ``almost Euler equality'', Theorem \ref{thm_Euler_almost_equality}}\label{sec_pf_of_thm3}

In this section we prove Theorem \ref{thm_Euler_almost_equality}, that is given $\xi\in (0,1)$, $\vartheta >0$ we construct a weak solution $\mathfrak{u}$ to the Navier--Stokes inequality with viscosity $\nu=0$ (which in this case should perhaps be called ``Euler inequality'') such that $\xi \leq d_H(S) \leq 1$, where $S$ is the singular set of $\mathfrak{u}$ and the ``approximate inequality''
\[
-\vartheta \leq  u\cdot \left( \p_t u + (u\cdot \nabla ) u +\nabla p \right)\leq  0
\]
holds in the sense that 
\eqnb\label{approx_ineq_what_we_really_mean_rewritten}
-\vartheta \leq  u\cdot \left( \p_t u+ (u\cdot \nabla ) u +\nabla p \right)\leq  0 \quad \text{ everywhere in } \RR^3 \times I_k \text{ for every }k,
\eqne
for some choice of pairwise disjoint time intervals $I_k$ with $\bigcup \overline{I_k} = [0,\infty )$. 

We explain below that $\mathfrak{u}$ can be obtained by a straightforward sharpening of the construction from Sections \ref{sec_CANTOR_set}-\ref{sec_cantor_new_geometric_arrangement} above; namely by replacing  $\delta$ in \eqref{cantor_def_of_h1h2} by
\eqnb\label{euler_replacement}
\delta_j \coloneqq \min \lewy \delta , 2\tau^{4j} \vartheta /3 \prawy.
\eqne
We note that such a trick immediately justifies our assumption of zero viscosity. Indeed since our choice of $\nu_0$ in the construction from Sections \ref{sec_CANTOR_set}-\ref{sec_cantor_new_geometric_arrangement} is dictated by \eqref{how_small_is_nu0} (recall the beginning of the paragraph below \eqref{CANTOR_3.14}), we see that \eqref{euler_replacement} gives
\[
\nu_0 \sup_{x\in R(U_1\setminus \supp\, \phi_1)} \left| u[0 , f_1 ](x) \cdot \Delta u[0, f_1 ](x) \right| \leq \tau^{4j} \vartheta /6 ,
\]
and so (by taking $j\to \infty$) we obtain $\nu_0=0$.\\

In order to make the idea of the proof (i.e. that the replacement \eqref{euler_replacement} proves Theorem \ref{thm_Euler_almost_equality}) more convincing we now briefly go through the main steps of the construction from  Sections \ref{sec_CANTOR_set}-\ref{sec_cantor_new_geometric_arrangement} articulating the main differences (which are cosmetic) as well as demonstrating how is \eqref{approx_ineq_what_we_really_mean_rewritten} obtained.\vspace{0.4cm}\\
\emph{Step 1.} Construct the geometric arrangement as in Section \ref{sec_cantor_new_geometric_arrangement}.\vspace{0.4cm}\\
\emph{Step 2.} For each $j$ let $h_1$, $h_2$ be defined as in \eqref{cantor_def_of_h1h2}, but with $\delta>0$ replaced by $\delta_j$, i.e.
\eqnb\label{euler_def_of_h1h2}
\begin{split}
h_{1,t}^2 &\coloneqq f_1 ^2 - 2t \delta_j \phi_1  ,\\
h_{2,t}^2 &\coloneqq f_2^2 -2t\delta_j \phi_2 + \int_0^t v_2  \cdot  F[v_1,h_{1,r}] \,\d r ,
\end{split}
\eqne
and let $h_t \coloneqq h_{1,t}+h_{2,t}$. Note that, as in Section \ref{sec_sketch_of_pf_of_thm2}, $h_1,h_2 \in C^\infty ( P\times (-\delta_j , T+\delta_j ); [0,\infty ))$,
\[
(v_i,h_{i,t},\phi_i ) \text{ is a structure on } U_i \qquad \text{ for } t\in (-\delta_j , T+\delta_j ), i=1,2,
\]
and
\eqnb\label{euler_3.14}
h_{2,T}^2 (y_n)    > \tau^{-2}  \left(  f_1 (R^{-1}x) +  f_2(R^{-1}x) \right)^2  + \theta 
\eqne
for $x\in G$, $n=1,\ldots, M$, where $y_n=R^{-1} (\Gamma_n (x))$.\vspace{0.4cm}\\
\emph{Step 3.} As in \eqref{rescaling_h^(j)} let 
\[
h_t^{(j)} (x_1,x_2) \coloneqq \sum_{m\in M(j)} h_t (\pi_m^{-1} (\tau^j x_1),x_2) ,
\]
and let $v^{(j)}  \in C^\infty \left( \RR^3 \times [0,T]; \RR^3 \right)$ be such that conditions (i)-(iv) of Proposition \ref{prop_CANTOR_existence_of_vj} are satisfied with $\nu_0=0$ and
\eqnb\label{euler_ineq_for_vj}
 \p_t \left| v^{(j)} \right|^2 +v^{(j)}\cdot \nabla \left( \left| v^{(j)} \right|^2 +2\overline{p}^{(j)} \right)  \geq -2 \tau^{4j}\vartheta 
\eqne
in $\RR^3 \times [0,T]$ (where $\overline{p}^{(j)}$ is the pressure function corresponding to $v^{(j)}$).\\

To this end, we repeat the proof of Proposition \ref{prop_CANTOR_existence_of_vj} with $\delta$ replaced by $\delta_j$ and, in order to obtain the extra property \eqref{euler_ineq_for_vj}, we amend the calculations from ``Case 1'' and ``Case 2'' from Step 3 in Section \ref{sec_cantor_pf_of_thm_uj} as follows. Fix $x\in P$, $t\in [0,T]$ and write, for brevity, $v=v^{(j)}$, $\overline{p}=\overline{p}^{(j)}$.\vspace{0.2cm}\\
\emph{Case 1.} $\phi_1^\fm (x)+\phi_2^\fm (x)<1$ for all $\fm \in \{ 1, \ldots , \fmm\}$. Then $v_1^\fm (x) = v_2^\fm (x) =0$ and so
\[
\begin{split}
\p_t |v(x,0,t)|^2 &=  -2\delta_j \sum_{\fm=1}^\fmm  (\phi_1^\fm (x) + \phi_2^\fm (x) )\\
&\geq -2\tau^{4j}\vartheta \\
&= -2\tau^{4j}\vartheta - v(x,0,t)\cdot \nabla \left( |v(x,0,t)|^2 +2 \overline{p} (x,0,t) \right), 
\end{split}
\]
where we also used \eqref{prop_of_structure_2}.\vspace{0.2cm}\\
\emph{Case 2.} $\phi_1^\fm (x) + \phi_2^\fm (x) =1$ for some $\fm \in \{ 1, \ldots , \fmm \}$. In this case \eqref{cantor_conv_1} (with $\delta$ replaced by $\delta_j$) reads 
\[
 \left| v_i^\fm \right| \left( \left| \nabla (q_{i,t}^{\fm ,k})^2 - \nabla ( h_{i,t}^\fm )^2 \right| +2 \sum_{\fn=1}^\fmm \sum_{l=1,2}\left| \nabla p[a_l^{\fn, k} (t) v_l^\fn , q_{l,t}^{\fn ,k} ]-\nabla p [a_l^{\fn ,k} (t) v_l^\fn ,h_{l,t}^\fn ] \right|  \right)\leq \delta_j /2
\]
for $i=1,2$, and so
\[
\begin{split}
\p_t |v(x,0,t)|^2 &=  \p_t q_{1,t}^{\fm ,k} (x)^2 + \p_t q_{2,t}^{\fm , k}  (x)^2   \\
&= -2\delta_j - \left( a_1^{\fm ,k} (t) v_1^\fm (x) +a_2^{\fm ,k} (t)v_2^\fm (x) \right) \cdot \nabla \left( \left( h_{1,t}^\fm (x) \right)^2 + \left(  h_{2,t}^\fm (x) \right)^2  {\color{white} \sum_{\fn=1}^\fmm}\right. \\
&\hspace{4cm}+ \left. 2\sum_{\fn=1}^\fmm \left( p[a_1^{\fn ,k}(t) v_1^\fn ,h_{1,t}^\fn ](x) + p[a_2^{\fn ,k}(t) v_2^\fn ,h_{2,t}^\fn ](x)\right) \right) \\
&\geq -3\delta_j - \left( a_1^{\fm ,k} (t) v_1^\fm (x) +a_2^{\fm ,k} (t)v_2^\fm (x) \right) \cdot \nabla \left( q_{1,t}^{\fm , k} (x)^2 + q_{2,t}^{\fm ,k} (x)^2  {\color{white} \sum_{\fn=1}^\fmm}\right. \\
&\hspace{4cm}+ \left. 2 \sum_{\fn=1}^\fmm \left( p[a_1^{\fn ,k}(t) v_1^\fn ,q_{1,t}^{\fn ,k}](x) + p[a_2^{\fn ,k}(t) v_2^\fn ,q_{2,t}^{\fn ,k}](x) \right) \right) \\
&= -3\delta_j - v_1(x,0,t) \p_{x_1} \left( |v(x,0,t)|^2 + 2\overline{p}(x,0,t) \right)\\
&\hspace{1.35cm}- v_2(x,0,t) \p_{x_2} \left( |v(x,0,t)|^2 + 2\overline{p}(x,0,t) \right)\\
&\geq  -2\tau^{4j}\vartheta - v (x,0,t) \cdot \nabla \left( |v(x,0,t)|^2 +2 \overline{p} (x,0,t) \right),
\end{split}
\]
as required by \eqref{euler_ineq_for_vj}, where we used \eqref{euler_replacement} and (as in Section \ref{sec_cantor_pf_of_thm_uj}) the fact that $\p_{x_3} |v(x,0,t)|^2 = \p_{x_3} \overline{p}(x,0,t) =0$.\vspace{0.4cm}\\
\emph{Step 4.} Define the solution $\mathfrak{u}$ and conclude the proof. That is we let $\mathfrak{u}$ be as in \eqref{cantor_mathfraku},
\[
\mathfrak{u} (t)  \coloneqq \begin{cases}
 u^{(j)}(t)  \qquad &\text{ if } t\in [t_{j},t_{j+1}) \text{ for some }j\geq 0, \\
 0 &\text{ if } t\geq  T_0,
\end{cases}
\]
where
\[
u^{(j)}(x_1,x_2,x_3,t) \coloneqq \tau^{-j} v^{(j)} (\tau^{-j} x_1 ,\gamma^{-j} (x_2), \tau^{-j}x_3, \tau^{-2j}(t-t_j)),
\]
and we let $I_j\coloneqq (t_j,t_{j+1})$, and $I_{\infty }\coloneqq (T_0 , \infty )$.

Then \eqref{euler_ineq_for_vj} gives
\[
 \p_t \left| u^{(j)} \right|^2 +u^{(j)}\cdot \nabla \left( \left| u^{(j)} \right|^2 +2p \right)  \geq -2 \vartheta ,\qquad \text{ in } \RR^3 \times I_j,
\]
for every $j\geq 0$.
and the rest of the claims of Theorem \ref{thm_Euler_almost_equality} (that is the facts that $\mathfrak{u}$ is a weak solution of the Navier--Stokes inequality (with $\nu_0=0$) and that $\xi\leq d_H(S)\leq 1$) follow as in Section \ref{sec_sketch_of_pf_of_thm2}.

\section{Acknowledgements}
The author would like to thank James Robinson for his enthusiasm and interest in this work as well as for reading various versions of the manuscript. His numerous comments greatly improved the quality of the text.

This work arose in part from the fluid mechanics reading group organised at the University of Warwick by James Robinson and Jos\'e Rodrigo.

The author was supported by EPSRC as part of the MASDOC DTC at the University of Warwick, Grant No. EP/HO23364/1.

\appendix

\section{Appendix}
\subsection{The function $f$ supported in $\overline{U}$ and with $Lf > 0$ near $\partial U$}\label{sec_claimX}
Here we show that for any set $U \Subset  P$ of the shape of a rectangle or a ``rectangle ring'', that is $U=V\setminus \overline{W}$ for some open rectangles $V,W$ with $W\Subset V$, and any $\eta >0$ there exists $\delta \in (0,\eta )$ and $f\in C_0^\infty (P ; [0,1])$ such that 
\[
\supp \, f = \overline{U},\quad f>0 \text{ in } U \text{ with } f=1 \text{ on } U_\eta
\]
and
\[
Lf >0 \quad \text{ in } U \setminus U_\delta .
\]
(Recall that $U_{\eta }$ denotes the $\eta $-subset of $U$, see \eqref{def_of_a_eta_subset_of_U}.)\\
The claim follows from Lemma \ref{lemma_existence_of_f_with_Lf_rectangle} below (which corresponds to the case of a rectangle) and from Lemma \ref{lemma_existence_of_f_with_Lf_rings} (which corresponds to the case of a rectangle ring).

We will need a certain generalisation of the Mean Value Theorem. For $f\colon \RR \to \RR$ let $f[a,b]$ denote the finite difference of $f$ on $[a,b]$,
\[
f[a,b] \coloneqq \frac{f(a)-f(b)}{a-b}
\]
and let $f[a,b,c]$ denote the finite difference of $f[\cdot , b]$ on $[a,c]$,
\[
f[a,b,c] \coloneqq \left( \frac{f(a)-f(b)}{a-b}-\frac{f(c)-f(b)}{c-b} \right)/(a-c) .
\]
\begin{lemma}[generalised Mean Value Theorem]\label{lemma_gen_MVT}
If $a<b<c$, $f$ is continuous in $[a,c]$ and twice differentiable in $(a,c)$ then there exists $\xi \in (a,c)$ such that $f[a,b,c]=f''(\xi )/2$.
\end{lemma}
\begin{proof} We follow the argument of Theorem 4.2 in \cite{conte}. Let
\[
p(x) \coloneqq f[a,b,c] (x-b)(x-c) + f[b,c] (x-c) + f(c) .
\]
Then $p$ is a quadratic polynomial approximating $f$ at $a,b,c$, that is $p(a)=f(a)$, $p(b)=f(b)$, $p(c)=f(c)$. Thus the error function $e(x) \coloneqq f(x)- p(x)$ has at least $3$ zeros in $[a,c]$. A repeated application of Rolle's theorem gives that $e''$ has at least one zero in $(a,c)$. In other words, there exists $\xi \in (a,c)$ such that $f'' (\xi ) = p''(\xi) = 2 f[a,b,c]$.
\end{proof}
\begin{corollary}\label{corollary_of_gen_MVT}
If $f\in C^3$ is such that $f=0$ on $(a-\delta , a]$ and $f''' >0$ on $(a,a+\delta )$ for some $a\in \RR$, $\delta >0$ then
\[
\begin{cases}
f'' (x) >0 ,\\
0<f'(x) < (x-a) f''(x),\\
f(x) < (x-a)^2 f'' (x) 
\end{cases} \qquad \text{ for } x\in (a,a+\delta ).
\]
Similarly, if $g=0$ on $[a,a+\delta )$ and $g''' <0 $ on $(a-\delta , a)$ then
\[
\begin{cases}
g'' (x) >0 ,\\
0>g'(x) > (x-a) g''(x),\\
g(x) < (x-a)^2 g'' (x) 
\end{cases} \qquad \text{ for } x\in (a-\delta ,a ).
\]
\end{corollary}
\begin{proof}
Since $f'''>0$ on $(a,a+\delta )$ we see that $f''$ is positive and increasing on this interval and so the first two claims for $f$ follow from the Mean Value Theorem. The last claim follows from the lemma above by noting that $2a-x \in (a-\delta , a]$ (so that $f(2a-x)=f(a)=0$), 
\[
\begin{split}
f(x) &= f(2a-x)-2f(a) + f(x) \\
&=2 (x-a)^2 f[2a-x,a,x]\\
&=(x-a)^2 f'' (\xi ) \\
&< (x-a)^2 f'' (x),
\end{split}
\]
where $\xi \in (2a-x,x)$. The claim for $g$ follows by considering $f(x)\coloneqq g(2a-x)$.
\end{proof}
We now show the claim in the case of $U$ in the shape of a rectangle.
\begin{lemma}[The cut-off function on a rectangle]\label{lemma_existence_of_f_with_Lf_rectangle}
Let $U\Subset P$ be an open rectangle, that is $U=(a_1,b_1 ) \times (a_2,b_2)$ for some $a_1,a_2,b_1,b_2 \in \RR$ with $b_1>a_1$, $b_2>a_2 >0$. Given $\eta >0$ there exists $\delta \in (0,\eta )$ and $f\in C_0^\infty (P ; [0,1])$ such that 
\[
\supp \, f = \overline{U},\quad f>0 \text{ in } U \text{ with } f=1 \text{ on } U_\eta,
\]
\[
Lf >0 \quad \text{ in } U \setminus U_\delta ,
\]
and $f$ is symmetric with respect to the vertical axis of $U$, that is
\[
f\left( \frac{a_1+b_1}{2}-x_1,x_2\right)=f\left( \frac{a_1+b_1}{2}+x_1,x_2\right), \quad (x_1,x_2)\in P.
\]
\end{lemma}
\begin{proof}
Let $f_1,f_2 \in C_0^\infty (\RR; [0,1])$ be such that $\mathrm{supp}\, f_i = [a_i,b_i]$, $f_i >0$ on $(a_i , b_i)$ with $f_i = 1$ on $[a_i+\eta,b_i-\eta]$,
\[
f_i''' >0 \text{ on } (a_i,a_i + \varepsilon ) \quad \text{ and } \quad  f_i''' <0 \text{ on } (b_i-\varepsilon,b_i ),\quad i=1,2,
\]
for some $\varepsilon \in (0,\eta )$. (Take for instance $f_i$'s such that
\[
f_i(x) = \begin{cases}
0 &x\leq a_i,\\
\exp \left( -(x-a_i )^{-2} \right) \quad & x\in (a_i, a_i +\varepsilon ) ,\\
1 & x\in (a_i + \eta , b_i - \eta ),\\
\exp  \left( -(b_i -x )^{-2} \right) \quad & x\in (b_i-\varepsilon , b_i ) ,\\
0 & x\geq b_i,
\end{cases}
\]
where $\varepsilon \in (0,\eta )$ is sufficiently small such that $f_i \leq 1$ on each of the intervals above, and define $f_i$ on the remaining intervals $[a_i + \varepsilon , a_i + \eta ]$, $[b_i - \eta , b_i - \varepsilon ]$ in the way such that $f_i \in C^\infty$, $f_i \leq 1$ and 
\[
f_i\left( \frac{a_i+b_i}{2}-x\right)=f_i\left( \frac{a_i+b_i}{2}+x\right), \quad x\in \RR,
\]
$i=1,2$, see Fig. \ref{fi_s_square_case})
\begin{figure}[h]
\centering
 \includegraphics[width=0.8\textwidth]{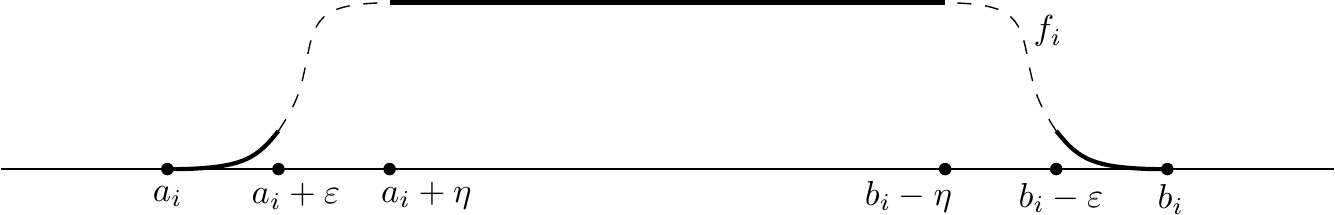}
 \nopagebreak
  \captionof{figure}{The $f_i$'s, $i=1,2$.}\label{fi_s_square_case} 
\end{figure}

Let $f(x_1,x_2) \coloneqq f_1 (x_1) f_2(x_2)$. Clearly $\supp \, f = \overline{U}$, $f>0$ in $U$, $f=1$ on $U_\eta$ and the last requirement of the lemma is satisfied due to the equality above. It remains to show that $Lf >0$ on $U \setminus U_\delta $ for some $\delta >0$. Let
\[
\begin{split}
g_1 (x_1) &\coloneqq f_1''(x_1), \\
g_2 (x_2) & \coloneqq f_2''(x_2) + f_2'(x_2)/x_2 - f_2 (x_2 )/x_2^2 .
\end{split}
\]
Then
\[
\begin{split}
Lf (x_1,x_2) &= f_1'' (x_1) f_2(x_2) + f_1 (x_1) f_2''(x_2) + f_1 (x_1) f_2'(x_2)/x_2 - f_1(x_1) f_2 (x_2) / x_2^2 \\
&= g_1 (x_1) f_2 (x_2) + f_1(x_1) g_2 (x_2).
\end{split}
\]
\emph{Claim}: There exists $d>0$ such that
\[
g_2 > f_2'' /4 >0 \qquad \text{ on } \left( a_2, a_2+d \right) \cup \left( b_2 -d , b_2\right).
\]
The claim follows from the corollary of the generalised Mean Value Theorem (see Corollary \ref{corollary_of_gen_MVT} above) by writing, for $d >0$ small such that $d<a_2/2$, $d < \varepsilon$ and $d / (b_2 -d )<1/2$,
\[
\begin{split}
g_2(x_2) &> f_2'' (x_2) - f_2 (x_2)/x_2^2 \\
&> f_2'' (x_2) \left( 1- \left( \frac{x_2-a_2}{x_2} \right)^2 \right) \\
&> f_2'' (x_2) \left( 1- \left( \frac{d}{a_2} \right)^2 \right) \\
&> \frac{3}{4} f_2'' (x_2)\\
& > \frac{1}{4} f_2'' (x_2)\\
& > 0
\end{split}
\]
for $x_2 \in (a_2 , a_2 +d )$, and 
\[
\begin{split}
g_2(x_2) &= f_2''(x_2) + f_2'(x_2)/x_2 - f_2 (x_2 )/x_2^2 > f_2'' (x_2) \left( 1 + \frac{x_2-b_2}{x_2} - \left( \frac{x_2-b_2}{x_2} \right)^2 \right) \\
&> f_2'' (x_2) \left( 1 - \frac{d}{b_2-d} - \left( \frac{d}{b_2-d} \right)^2 \right)>f_2''(x_2) /4 >0
\end{split}
\]
for $x_2 \in (b_2 - d , b_2)$.\\

Using the Claim and Corollary \eqref{corollary_of_gen_MVT} (particularly positiveness of 2nd derivatives) we see that $g_i, f_i$ are positive on $(a_i,a_i +d) \cup (b_i -d , b_i)$, $i=1,2$. Thus
\[
Lf >0 \quad \text{ in } \left( (a_1,a_1 +d) \cup (b_1 -d , b_1) \right) \times  \left( (a_2,a_2 +d) \cup (b_2 -d , b_2) \right),
\]
that is in the ``$d$-corners'' of $U$, see Fig. \ref{d_corners}.\\
\begin{figure}[h]
\centering
 \includegraphics[width=0.7\textwidth]{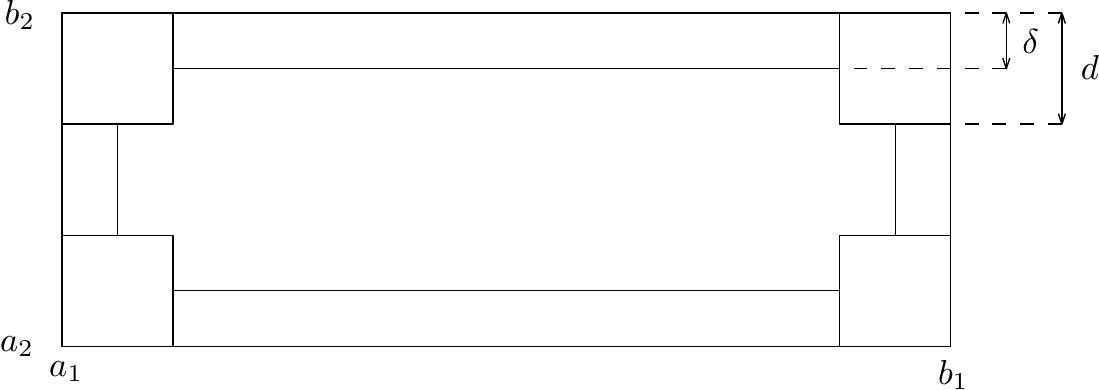}
 \nopagebreak
  \captionof{figure}{The ``$d$-corners'' and ``$\delta$-stripes''.}\label{d_corners} 
\end{figure}

Now let $m, M>0$ be small such that $f_i >m$, $|g_i |<M$ in $[a_i +d,b_i-d]$, $i=1,2$. Let $\delta \in (0,d)$ be such that $m/4-\delta^2M >0$. The proof of the lemma is complete when we show that
\[
Lf >0 \quad \text{ in } [a_i+d,b_i-d] \times \left( (a_j , a_j+\delta ) \cup (b_j-\delta , b_j)  \right),\, (i,j)=(1,2), (2,1),
\]
that is in the ``$\delta$-strips'' at $\partial U$ between the $d$-corners, see Fig. \ref{d_corners}.\\

Let $x_1 \in [a_1+d , b_1-d]$ and $x_2 \in (a_2, a_2 + \delta )$. Then $g_1(x_1)>-M$, $g_2(x_2)>f_2'' (x_2)/4$ (from \emph{Claim}), $f_2(x_2)< (x_2-a_2)^2 f_2''(x_2)$ (from the generalised Mean Value Theorem, see Corollary \ref{corollary_of_gen_MVT}), $f_1(x_1)>m$, and so
\[
\begin{split}
Lf(x_1,x_2 ) & = g_1 (x_1 ) f_2 (x_2) + f_1 (x_1) g_2 (x_2) \\
&> -M f_2 (x_2) + f_1(x_1) f_2''(x_2)/4 \\
&>f_2'' (x_2) \left( -M (x_2 - a_2 )^2 + m/4 \right) \\
& > f_2''(x_2) \left( m/4 -M \delta^2 \right) \\
&>0.
\end{split}
\]
As for $x_2 \in (b_2-\delta , b_2)$, simply replace $a_2$ in the above calculation by $b_2$. The opposite case, that is the case $x_1 \in (a_1, a_1 + \delta ) \cup (b_1-\delta , b_1)$, $x_2 \in [a_2+d , b_2-d]$, follows similarly.
\end{proof}
Let 
\[
U^\eta \coloneqq \{ x\in \RR^2 \colon \mathrm{dist}\,(x, U ) < \delta \}
\]
denote the \emph{$\eta$-neighbourhood} of $U$ (this is not to confused with $U_\eta$ which denotes the $\eta$-subset, recall the beginning of this section). We now extend the above lemma to the case of $U$ in the shape of a ``rectangular ring''.
\begin{lemma}[The cut-off function on a rectangular ring]\label{lemma_existence_of_f_with_Lf_rings}
If $U \Subset  P$ is a \emph{rectangular ring}, that is $U=V\setminus \overline{W}$ where $V$, $W$ are open rectangles with $W \Subset V$, then the assertion of the last lemma is valid.
\end{lemma}
\begin{proof}
It is enough to show that there exist $\delta >0$ and $f\in C^\infty (P; [0,1])$ such that $f=0$ on $\overline{W}$, $f>0$ outside $\overline{W}$ with $f=1$ outside $W^\eta$ and
\[
Lf >0 \quad \text{ in } W^\delta \setminus \overline{W}.
\]
Then the lemma follows by letting
\[
g \coloneqq \begin{cases}
\widetilde{f}  \quad & \text{ on } P \setminus W^\eta,\\
f \quad &\text{ on } W^\eta,
\end{cases}
\]
where $\widetilde{f}$ is from the previous lemma applied to $V$.

We write $W=(a_1,b_1)\times (a_2,b_2)$ for some $a_1,a_2,b_1,b_2 \in \RR$ with $b_1>a_1$ and $b_2>a_2>0$. Let $f_1,f_2\in C^\infty (\RR ; [0,1])$ be such that $f_i =1$ outside $(a_i-\eta, b_i +\eta  )$, $f_i =0$ on $(a_i,b_i)$ and
\[
f_i''' <0 \text{ on } (a_i-\varepsilon , a_i ) \quad \text{ on } \quad f_i''' >0 \text{ on } (b_i, b_i +\varepsilon ),\quad i=1,2,
\]
for some $\varepsilon \in (0,\eta/2)$. (Such functions can be constructed by a use of the exponential function, as in the previous lemma, see also Fig. \ref{fi_s_ring_case})
\begin{figure}[h]
\centering
 \includegraphics[width=0.8\textwidth]{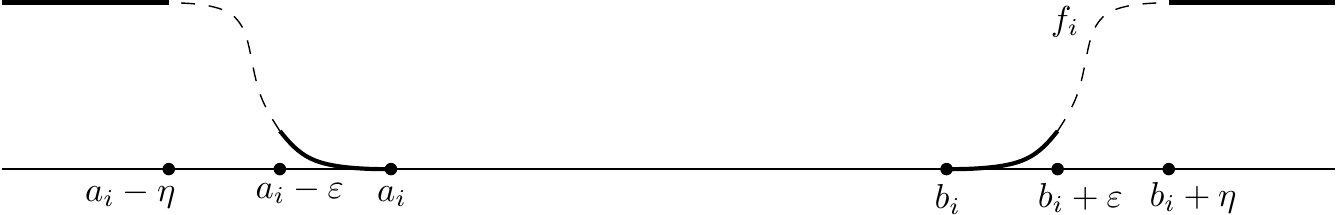}
 \nopagebreak
  \captionof{figure}{The $f_i$'s, $i=1,2$ (cf. Fig. \ref{fi_s_square_case}).}\label{fi_s_ring_case} 
\end{figure}
Let $f(x_1,x_2)\coloneqq f_1(x_1) f_2 (x_2)$. Then $f=0$ on $\overline{W}$, $f>0$ outside $\overline{W}$ with $f=1$ outside $W^\eta$. It remains to show that $Lf >0$ in $W^\delta \setminus \overline{W}$ for some $\delta >0$. Note that
\[\begin{split}
Lf(x_1,x_2) &= \left( f_1'' (x_1) - f_1 (x_1)/x_2^2 \right) + \left( f_2'' (x_2) + f_2' (x_2) /x_2 -f_2(x_2) / x_2^2 \right) \\
&=: g_1 (x_1,x_2) + g_2 (x_2).
\end{split}\]
As in \emph{Claim} in the proof of the previous lemma we see that
\[
g_2 > f_2''/4 >0 \quad \text{ in } (a_2-\delta , a_2) \cup (b_2, b_2 + \delta )
\]
for sufficiently small $\delta >0$. Thus since $f_2$ vanishes on $[a_2,b_2]$ we see that
\eqnb\label{g_2_positive}
g_2 \geq 0 \text { on } (a_2-\delta , b_2+\delta ) \quad \text{ with } \quad g_2>0 \text{ outside } [a_2,b_2].
\eqne
As for $g_1$ let $\delta $ be such that $\delta /(a_2-\delta )< 1/2$. Then, using the corollary of the generalised Mean Value Theorem (Corollary \ref{corollary_of_gen_MVT}), we obtain for any $x_2 > a_2-\delta $
\[\begin{split}
g_1 (x_1,x_2) &= f_1'' (x_1) - f_1 (x_1)/x_2^2 > f_1''(x_1) \left( 1- \left( \frac{x_1-a_1}{x_2} \right)^2 \right) \\
&> f_1''(x_1) \left( 1- \left( \frac{\delta }{a_2-\delta } \right)^2 \right) > \frac{3}{4} f_1''(x_1) >0
\end{split}\]
for $x_1\in (a_1-\delta , a_1)$. As for $x_1 \in (b_1, b_1 + \delta )$ replace $a_1$ in the above calculation by $b_1$. Thus, since $f_1$ vanishes on $[a_1,b_1]$ we see that for each $x_2 >a_2-\delta $
\[
g_1 (\cdot , x_2) \geq 0 \text{ on } (a_1-\delta ,b_1+\delta ) \quad \text{ with } \quad g_1 (\cdot , x_2) >0 \text{ outside } [a_1,b_1].
\]
This and \eqref{g_2_positive} give
\[
Lf \geq 0 \text{ on } W^\delta \quad \text{ with } \quad Lf >0 \text{ outside } \overline{W},
\]
as required.
\end{proof}

\subsection{Preliminary calculations}\label{app_prelim_calc}
Here we briefly verify \eqref{scalar_fcns_rotationally_invariant}. Let $\varphi \in [0,2\pi )$ and let $R\coloneqq R_\varphi$ for brevity of notation. We can represent $R$ in the matrix form
\[
R=\begin{pmatrix}
1 & 0 & 0 \\
0 & \cos \varphi & -\sin \varphi \\
0 & \sin \varphi & \cos \varphi
\end{pmatrix} .
\]
Note that $R$ is orthogonal, so that $R^T R=I$, where $I$ denotes the identity matrix. We write
\[
\nabla u \coloneqq \begin{pmatrix}
\p_1 u_1 & \p_2 u_1  & \p_3 u_1 \\
\p_1 u_2 &\p_2 u_2 & \p_3 u_2 \\
\p_1 u_3 &\p_2 u_3 & \p_3 u_3
\end{pmatrix},\qquad \nabla q \coloneqq \begin{pmatrix}
\p_1 q \\
\p_2 q \\
\p_3 q
\end{pmatrix};
\]

If $u(Rx)=Ru(x)$ and $q(Rx)=q(x)$ we can use the calculus identities
\[
\nabla (u(Rx)) = \nabla u(Rx) R,\qquad \nabla (Ru)=R\nabla u\qquad \text{ and } \qquad \nabla (q(Rx))=R^T \nabla q(Rx)
\]
to write
\[
\begin{split}
( ( u \cdot \nabla ) u ) (Rx) & = \nabla u (Rx) u(Rx) =  \nabla (u (Rx))R^T u(Rx) =  \nabla ( R u(x) ) R^T R u (x)\\
&= R \nabla u (x)  u (x) = R ((u\cdot \nabla ) u) (x),
\end{split}
\]
\[
|u|^2 (Rx) = u(Rx)^T u(Rx) =(  R u(x) )^T (R u(x))= u(x)^T u(x) = |u|^2 (x),
\]
\[
\begin{split}
\mathrm{div}\, u (Rx) &= \mathrm{tr} \nabla u (Rx) = \mathrm{tr} (\nabla (u(Rx))R^T )= \mathrm{tr} (\nabla (Ru(x))R^T )= \mathrm{tr} (R \nabla u(x)R^T ) \\
&=\mathrm{tr}  \nabla u(x) =\mathrm{div}\, u (x),
\end{split}
\]
where we used the fact that $\mathrm{tr} RAR^T = \mathrm{A}$ for any matrix $A$,
\[
(u\cdot \nabla q )(Rx) = u(Rx)^T \nabla q (Rx) = (Ru(x) )^T (R \nabla (q(Rx)))= u(x)^T \cdot \nabla q(x)=(u\cdot \nabla ) q.
\]
By taking $q\coloneqq |u|^2$ we obtain
\[
(u\cdot \nabla |u|^2 )(Rx)=(u\cdot \nabla |u|^2 )(x) .
\]
Also, since $u(x)=R^T u(Rx)$ we have 
\[
\begin{split}
\Delta u_k (x &= \Delta (u_k (x) ) = \sum_j \Delta (R_{jk} u_j (Rx) ) = \sum_{i,j } \p_i \p_i (R_{jk} u_j (Rx)) \\
&= \sum_{i,j,l } R_{jk} R_{li} \p_i ( \p_l u_j (Rx)  )= \sum_{i,j,l,m } R_{jk } R_{li} R_{mi}  \p_m \p_l u_j (Rx)\\
& = \sum_{j,l,m } R_{jk} \delta_{ml} \p_m \p_l u_j (Rx) =\sum_{j} R_{jk} \Delta u_j (Rx) .
\end{split}
\]
for each $k\in \{ 1,2,3 \}$, where $\delta_{ml}$ denotes the Kronecker delta. Thus
\[
\Delta u (Rx) = R \Delta u (x) , 
\]
as needed. Finally
\[
(u\cdot \Delta u ) (Rx) = u (Rx)^T \Delta u (Rx) = (Ru (x))^T R \Delta u (x) = u (x)^T \Delta u (x) =( u\cdot \Delta u ) (x) .
\]
\subsection{Some continuity properties of $p[v,f]$ and $u[v,f]$ with respect to $f$}\label{sec_continuity_of_p,u_wrt_f}
In this section we discuss two rather technical results regarding continuity of $\nabla p [v,f]$ with respect to $f$ (Lemma \ref{lem_continuity_of_the_pressure_fcns} below) as well as continuity of the term $u[v,f]\cdot \Delta u [v,f]$ with respect to $f$ (Lemma \ref{lem_continuity_of_u_dot_Delta_u}). We mentioned them at the end of Section \eqref{u[v,f]_section}. They are used in \eqref{conv_1}, \eqref{cantor_conv_1} and \eqref{conv_2}, \eqref{cantor_conv_2} respectively.
\begin{lemma}\label{lem_continuity_of_the_pressure_fcns}
Suppose that $v_k \in C^\infty (P;\RR^2 )$, $f_k,f\in C^\infty (P; [0,\infty ))$ are such that $|v_k|<\min (f_k,f)$ and
\[
\mathrm{supp}\, v_k, \mathrm{supp}\, f_k, \mathrm{supp}\, f \subset  K
\]
for all $k$, where $K$ is a compact subset of $P$, and that
\[
 f_k \to f, \quad \text{ and }\quad \nabla f_k \to f \quad \text{ uniformly in } P.
\]
Then
\[
\nabla p [v_k ,f_k] - \nabla p[v_k , f] \to 0 
\]
uniformly on $\RR^2$.
\end{lemma}
The point of the lemma is that we do not require any convergence of the $v_k$'s. We used the lemma (or rather its straightforward generalisation to include time dependence) in a convergence argument in \eqref{conv_1}. 
  \begin{proof}
  Let $z=(z_1,z_2)\in \overline{P}$ and observe that for $u=u[v,f]$ (for some $v=(v_1,v_2)$, $f$) we have the identity
  \[
  \begin{split}
  \sum_{i,j=1}^3 \p_i u_j (z_1,z_2,0) \p_j &u_i (z_1,z_2 ,0) \\
  &=  ( \p_1 v_1  )^2 + (\p_2 v_2 )^2 + 2 \p_2 v_1 \p_1 v_2 - \left( \p_2 f^2 - \p_2 |v|^2 \right)/ 2z_2 + v_2^2/z_2^2
  \end{split}
  \]
  where (for brevity) we skipped the argument $z$ of the functions on the right-hand side.
  Therefore, setting 
  \[\begin{split}
  F_k (y) &\coloneqq \sum_{i,j=1}^3 \p_i u_j [v_k,f_k](y) \p_j u_i [v_k,f_k](y),  \\
  G_k(y) &\coloneqq \sum_{i,j=1}^3 \p_i u_j [v_k,f](y) \p_j u_i [v_k,f](y) ,
  \end{split}
  \]
  where $y\in \RR^3$, and letting $z=(z_1,z_2,0)\coloneqq R^{-1} (y)$ we obtain, using axisymmetry of $F_k$, $G_k$ (recall \eqref{scalar_fcns_rotationally_invariant})
  \[
  \begin{split}
  F_k (y) - G_k(y) &= F_k (z) - G_k(z) \\
  &= \left( \p_2 f^2 - \p_2 f_k^2 \right)/2z_2,
\end{split}
  \]
  as all the term that include the components of $v_k$'s vanish. Thus 
  \[
  F_k-G_k \to 0\qquad \text{ uniformly in } \RR^3.
  \]
  Since 
  \[ 
  \nabla p^* [v_k ,f_k] (x) - \nabla p^*[v_k , f] (x) = \frac{1}{4\pi } \int_{\RR^3 } \left( F_k (y) - G_k(y) \right) \frac{x-y}{|x-y |^3} \d y
  \]
  we immediately obtain that 
 \[ \nabla p^* [v_k ,f_k]  - \nabla p^*[v_k , f] \to 0 \qquad \text{ uniformly on }\RR^3, \]
 and the claim of the lemma follows by setting $x_3=0$.
  \end{proof}
We now a result regarding continuity of $u[v,f]\cdot \Delta u[v,f]$ with respect to $f$.
\begin{lemma}\label{lem_continuity_of_u_dot_Delta_u}
Let $T>0$, $K$ be a compact subset of $P$ and suppose that $v \in C^\infty (P;\RR^2 )$ and $f_k,f \colon K \times [0,T] \to  [0,\infty )$ are such that $|v|<f(t)$ on $K$ for every $t\in [0,T]$, and let (for each $k$) $a_k\colon [0,T] \to [-1,1]$ be arbitrary. If 
\eqnb\label{f_k_converge}
D^\alpha f_k \to D^\alpha f, \text{ uniformly in } K\times [0,T],
\eqne
for any multiindex $\alpha=(\alpha_1, \alpha_2 )$ with $|\alpha |\leq 2$, then
\[
u[a_k (t) v, f_k (t)] \cdot \Delta u[a_k (t) v , f_k (t)] -u[a_k (t) v, f(t) ] \cdot \Delta u[a_k (t) v , f (t)]  \to 0 
\]
uniformly on $R(K)$ and in $t\in [0,T]$.
\end{lemma}  
\begin{proof}
Since in cylindrical coordinates
\[
\Delta = \p_{x_1x_1} + \p_{\rho \rho } + \frac{1}{\rho } \p_\rho + \frac{1}{\rho^2 } \p_{\varphi \varphi }
\]
we obtain, using \eqref{what_is_u1_u2_u3},
\[
\begin{split}
\Delta u_1[v,f] (x_1,\rho ,0) &= \left( \p_{x_1x_1} + \p_{\rho\rho } + \rho^{-1} \p_\rho  \right) v_1, \\
\Delta u_2[v,f] (x_1,\rho ,0) &= \left( \p_{x_1x_1} + \p_{\rho\rho } + \rho^{-1} \p_\rho -\rho^{-2} \right) v_2,\\
\Delta u_3[v,f] (x_1,\rho ,0) &= \left( \p_{x_1x_1} + \p_{\rho\rho } + \rho^{-1} \p_\rho -\rho^{-2} \right) \sqrt{f^2 - |v|^2},
\end{split}
\]
compare with \eqref{temp_calculating_diuj}. Thus
\[
u [v,f] \cdot \Delta u[v,f]= (v\text{-terms}) + \sqrt{f^2 - |v|^2}\left( \p_{x_1x_1} + \p_{\rho\rho } + \rho^{-1} \p_\rho -\rho^{-2} \right) \sqrt{f^2 - |v|^2}
\]
where we skipped the argument ``$(x_1,\rho ,0 )$'' on the left-hand side and we denoted all terms involving components of $v$ (and its derivatives) by ``$(v\text{-terms})$''. Expanding the last term on the right-hand side we obtain

\[
(\p_{x_1x_1} +\p_{\rho \rho}  )f^2/2 -\frac{\left( \p_{x_1} (f^2-|v|^2)\right)^2+\left( \p_{\rho} (f^2-|v|^2)\right)^2}{4(f^2-|v|^2)} +\rho^{-1} f\p_\rho f - \rho^{-2} f^2 + (v\text{-terms})
\]
Hence, since both
\[
\begin{split}
F_k (x,t) &\coloneqq u[a_k (t) v, f_k (t)](x) \cdot \Delta u[a_k (t) v , f_k (t)](x),\quad \text{ and }\\
G_k (x,t) &\coloneqq u[a_k (t) v, f (t)] (x)\cdot \Delta u[a_k (t) v , f (t)](x)
\end{split}
\]
are axisymmetric (recall \eqref{scalar_fcns_rotationally_invariant}) we can write 
\[\begin{split}
F_k(x,t)-G_k(x,t) &= F_k (x_1,\rho ,0,t) - G_k (x_1,\rho ,0,t)\\
&= (\p_{x_1x_1} +\p_{\rho \rho}  )(f_k^2-f^2)/2 \\
&\hspace{0.4cm}+\frac{\left( \p_{x_1} (f^2-|a_k v|^2)\right)^2+\left( \p_{\rho} (f^2-|a_k v|^2)\right)^2}{4(f^2-|a_kv|^2)} \\
& \hspace{0.4cm}-\frac{\left( \p_{x_1} (f_k^2-|a_kv|^2)\right)^2+\left( \p_{\rho} (f_k^2-|a_k v|^2)\right)^2}{4(f_k^2-|a_k v|^2)}  \\
&\hspace{0.4cm}+\rho^{-1} (f_k \p_\rho f_k -f \p_\rho f) - \rho^{-2}( f_k^2-f^2 )
\end{split}\]
since the $v$-terms cancel out. Here $\rho \coloneqq \sqrt{x_2^2+x_3^2}$. Thus the claim of the lemma (i.e. that $F_k -G_k \to 0$ uniformly in $R(K)\times [0,T]$) follows if we can show that the difference of the two fractions above converges uniformly to $0$ (the other terms converge by assumption). We will focus on the terms with derivatives with respect to $\rho $, and we will write $\p\equiv \p_\rho $ for brevity (the terms with the $x_1$ derivatives are analogous). Bringing the two fractions under the common denominator we obtain
\[
\frac{\left( \p (f^2-|a_k v|^2)\right)^2(f_k^2-|a_kv|^2) -\left( \p (f_k^2-|a_k v|^2)\right)^2(f^2-|a_kv|^2)}{4 (f_k^2-|a_kv|^2)(f^2-|a_kv|^2)}
\]
The denominator is bounded away from $0$ (as $|a_k|\leq 1$, $|v|<f$ on the compact set $K\times [0,T]$ and $f_k\to f $ uniformly), and so it is sufficient to verify that the numerator converges uniformly, which is clear after expanding the brackets: the terms with $v$ only cancel out and (since no derivative falls on $a_k$; recall $a_k$ depends on time only) each of the other terms is of the form 
\[
(\text{something converging uniformly by \eqref{f_k_converge}})*(\text{term involving }a_k \text{ and derivatives of }v)
\] 
which converge uniformly to $0$ (as $|a_k|\leq 1$), as required.
\end{proof}

%'''''''''''''''''''''''''''''''''''''''''''''''''''''''''''''''''''''''''''''''''''''''''''''''''''''''''''''''''''

%************************************************************
%Bibliography
%************************************************************
\bibliography{liter}{}
\end{document}